%% file: main.tex
\documentclass[10pt]{amsart}

\makeatletter
\def\@settitle{\begin{center}%
  \baselineskip14\p@\relax
    \normalfont\LARGE%<- NEW

  \@title
  \end{center}%
}
\makeatother

\usepackage[dvipsnames]{xcolor}
\usepackage[hyperfootnotes=false]{hyperref}
\usepackage{enumerate}
\usepackage{tikz}
\usepackage{tikz-cd}
\usepackage{hyperref}
\usepackage[capitalize, noabbrev, nosort]{cleveref}
\usepackage{amsthm,amsfonts,amsmath,amscd,amssymb}
\usepackage{xcolor,float}
\usepackage[all]{xy}
\usepackage{mathrsfs}
\usepackage{graphics,graphicx}
\usepackage{caption}
\usepackage{comment}
\usepackage{array}
\newcolumntype{P}[1]{>{\centering\arraybackslash}p{#1}}
\newcolumntype{M}[1]{>{\centering\arraybackslash}m{#1}}

\DeclareMathAlphabet\EuScript{U}{eus}{m}{n}
\SetMathAlphabet\EuScript{bold}{U}{eus}{b}{n}

\usepackage{epigraph}
\let\oldmarginpar\marginpar
\renewcommand\marginpar[1]{\-\oldmarginpar[\raggedleft\footnotesize #1]%
	{\raggedright\footnotesize #1}}
\theoremstyle{plain}
\newtheorem{thm}{Theorem}[section]
\newtheorem{lemma}[thm]{Lemma}
\newtheorem*{theorem*}{Theorem}
\newtheorem*{corollary*}{Corollary}
\newtheorem{example}[thm]{Example}
\newtheorem{prop}[thm]{Proposition}

\newtheorem{cor}[thm]{Corollary}

\theoremstyle{definition}

\newtheorem{definition}[thm]{Definition}

\newtheorem{remark}[thm]{Remark}

\numberwithin{equation}{section}

\newcommand{\RHom}{\text{RHom}}
\newcommand{\g}{\gamma}

\renewcommand{\P}{\mathbb{P}}
\renewcommand{\L}{\mathbb{L}}

\newcommand{\D}{\mathscr{D}}
\newcommand{\N}{\mathbb{N}}
\newcommand{\Z}{\mathbb{Z}}

\newcommand{\R}{\mathbb{R}}
\newcommand{\C}{\mathbb{C}}

\newcommand{\SO}{\mathcal{O}}

\newcommand{\SC}{\mathcal{C}}

\newcommand{\SF}{\mathscr{F}}

\newcommand{\SL}{\mathscr{L}}

\newcommand{\La}{\Lambda}

\newcommand{\la}{\lambda}

\newcommand{\dd}{\partial}

\newcommand{\op}{\operatorname{Op}}

\newcommand{\sse}{\subset}
\newcommand{\lr}{\longrightarrow}

\newcommand{\msh}{\operatorname{\operatorname{\mu Sh}}}
\newcommand{\tr}{\operatorname{tr}}
\newcommand{\HH}{\operatorname{HH}}
\newcommand{\HO}{\operatorname{HO}}

\newcommand{\Aut}{\operatorname{Aut}}

\newcommand{\GL}{\operatorname{GL}}
\newcommand{\BGL}{\operatorname{BGL}}

\newcommand{\Sh}{\operatorname{Sh}}

\newcommand{\modk}{\operatorname{\mbox{Mod}}(k)}

\newcommand{\Aug}{\operatorname{Aug}}
\newcommand{\Spec}{\operatorname{Spec}}

\newcommand{\Perf}{\operatorname{Perf}}
\newcommand{\Mod}{\operatorname{Mod}}
\newcommand{\Qcoh}{\operatorname{QCoh}}
\newcommand{\IndCoh}{\operatorname{IndCoh}}

\newcommand{\st}{\text{st}}

\def\Op{{\mathcal O}{\it p}}
\newcounter{daggerfootnote}

\newcommand{\bC}{\mathbb{C}}
\newcommand{\bD}{\mathbb{D}}

\newcommand{\bK}{\mathbb{K}}
\newcommand{\bL}{\mathbb{L}}
\newcommand{\bN}{\mathbb{N}}

\newcommand{\bS}{\mathbb{S}}
\newcommand{\bR}{\mathbb{R}}
\newcommand{\bZ}{\mathbb{Z}}

\newcommand{\cA}{\mathcal{A}}

\newcommand{\cC}{\mathcal{C}}
\newcommand{\cD}{\mathcal{D}}
\newcommand{\cE}{\mathcal{E}}

\newcommand{\cM}{\mathcal{M}}

\newcommand{\cO}{\mathcal{O}}

\newcommand{\LC}{{\EuScript C}}

\renewcommand{\tt}{\mathfrak{t}}

\newcommand{\ww}{\mathfrak{w}}

\newenvironment{RC}{\color{blue}{\bf RC:} \footnotesize}{}

\newenvironment{WL}{\color{red}{\bf WL:} \footnotesize}{}

\newcommand{\M}{\mathfrak{M}}

\newcommand{\Loc}{\mathrm{Loc}}

\newcommand{\Hom}{\mathrm{Hom}}
\newcommand{\End}{\mathrm{End}}
\newcommand{\Sym}{\mathrm{Sym}}
\newcommand{\Tr}{\mathrm{Tr}}
\newcommand{\Id}{\mathrm{Id}}
\usepackage{dsfont}
\allowdisplaybreaks

\newcommand{\id}{\mathrm{id}}
\def \vertbar [#1](#2,#3,#4){
    \draw [#1] (#2,#3) -- (#2,#4);
    \draw [fill=white] (#2,#3) circle [radius=0.1];
    \draw [fill=black] (#2,#4) circle [radius=0.1];
}

\usepackage{mathdots}
\usepackage{afterpage}
\makeatletter
\providecommand{\leftsquigarrow}{%
  \mathrel{\mathpalette\reflect@squig\relax}%
}
\newcommand{\reflect@squig}[2]{%
  \reflectbox{$\m@th#1\rightsquigarrow$}%
}
\makeatother

\makeatletter
\def\Ddots{\mathinner{\mkern1mu\raise\p@
\vbox{\kern7\p@\hbox{.}}\mkern2mu
\raise4\p@\hbox{.}\mkern2mu\raise7\p@\hbox{.}\mkern1mu}}
\makeatother

\def \horline [#1](#2,#3,#4){
    \draw [#1] (#2,#4) -- (#3,#4);
    \draw [fill=white] (#2,#4) circle [radius=0.1];
    \draw [fill=black] (#3,#4) circle [radius=0.1];
}

\def \crossing (#1,#2)(#3,#4){
\draw (#1,#2) -- (#3,#4);
\draw (#1,#4) -- (#3,#2);
}

\DeclareFontFamily{U}{mathb}{}
\DeclareFontShape{U}{mathb}{m}{n}{
  <-5.5> mathb5
  <5.5-6.5> mathb6
  <6.5-7.5> mathb7
  <7.5-8.5> mathb8
  <8.5-9.5> mathb9
  <9.5-11.5> mathb10
  <11.5-> mathbb12
}{}

\usetikzlibrary{decorations.markings, arrows}
\tikzset{tangent/.style={decoration={markings,mark=at position #1 with {
      \coordinate (tangent point-\pgfkeysvalueof{/pgf/decoration/mark info/sequence number}) at (0pt,0pt);
      \coordinate (tangent unit vector-\pgfkeysvalueof{/pgf/decoration/mark info/sequence number}) at (1,0pt);
      \coordinate (tangent orthogonal unit vector-\pgfkeysvalueof{/pgf/decoration/mark info/sequence number}) at (0pt,1);
      }},postaction=decorate},
    use tangent/.style={
        shift=(tangent point-#1),
        x=(tangent unit vector-#1),
        y=(tangent orthogonal unit vector-#1)
    },
    use tangent/.default=1
    }

\begin{document}

	\title{Positive microlocal holonomies are globally regular}

\author{Roger Casals}
	\address{University of California Davis, Dept. of Mathematics, USA}
	\email{casals@ucdavis.edu}

 \author{Wenyuan Li}
	\address{University of Southern California, Dept. of Mathematics, USA}
	\email{wenyuan.li@usc.edu}
 
		
\maketitle
\vspace{-1.2cm}
\begin{abstract} We establish a geometric criterion for local microlocal holonomies to be globally regular on the moduli space of Lagrangian fillings. This local-to-global regularity result holds for arbitrary Legendrian links and it is a key input for the study of cluster structures on such moduli spaces. Specifically, we construct regular functions on derived moduli stacks of sheaves with Legendrian microsupport by studying the Hochschild homology of the associated dg-categories via relative Lagrangian skeleta. In this construction, a key geometric result is that local microlocal merodromies along positive relative cycles in Lagrangian fillings yield global Hochschild 0-cycles for these dg-categories.
\end{abstract}

%\tableofcontents

%%%%%%%%%%%%%%%%%%%%%%%%%%%%%%%%%%%%%%%%%%%%%%%%%%%%%%%%%%%%%%%%%
%%%%%%%%%%%%%%%%%%%%%%%%%%%%%%%%%%%%%%%%%%%%%%%%%%%%%%%%%%%%%%%%%
%%%%%%%%%%%%%%%%%%%%%%%%%%%%%%%%%%%%%%%%%%%%%%%%%%%%%%%%%%%%%%%%%

\input{0_introduction}

%%%%%%%%%%%%%%%%%%%%%%%%%%%%%%%%%%%%%%%%%%%%%%%%%%%%%%%%%%%%%%%%%
%%%%%%%%%%%%%%%%%%%%%%%%%%%%%%%%%%%%%%%%%%%%%%%%%%%%%%%%%%%%%%%%%
%%%%%%%%%%%%%%%%%%%%%%%%%%%%%%%%%%%%%%%%%%%%%%%%%%%%%%%%%%%%%%%%%

\input{2_SheafCat}

%%%%%%%%%%%%%%%%%%%%%%%%%%%%%%%%%%%%%%%%%%%%%%%%%%%%%%%%%%%%%%%%%
%%%%%%%%%%%%%%%%%%%%%%%%%%%%%%%%%%%%%%%%%%%%%%%%%%%%%%%%%%%%%%%%%
%%%%%%%%%%%%%%%%%%%%%%%%%%%%%%%%%%%%%%%%%%%%%%%%%%%%%%%%%%%%%%%%%

\input{3_KSStack}

%%%%%%%%%%%%%%%%%%%%%%%%%%%%%%%%%%%%%%%%%%%%%%%%%%%%%%%%%%%%%%%%%
%%%%%%%%%%%%%%%%%%%%%%%%%%%%%%%%%%%%%%%%%%%%%%%%%%%%%%%%%%%%%%%%%
%%%%%%%%%%%%%%%%%%%%%%%%%%%%%%%%%%%%%%%%%%%%%%%%%%%%%%%%%%%%%%%%%

\input{4_HH}

%%%%%%%%%%%%%%%%%%%%%%%%%%%%%%%%%%%%%%%%%%%%%%%%%%%%%%%%%%%%%%%%%
%%%%%%%%%%%%%%%%%%%%%%%%%%%%%%%%%%%%%%%%%%%%%%%%%%%%%%%%%%%%%%%%%
%%%%%%%%%%%%%%%%%%%%%%%%%%%%%%%%%%%%%%%%%%%%%%%%%%%%%%%%%%%%%%%%%

\appendix
\input{5_Appendix}

\bibliographystyle{alpha}
\bibliography{main}

\end{document}

%% file: 0_introduction.tex
\section{Introduction}\label{sec:intro}

%add a couple of paragraphs long additional introductory description (or a bit more if you wish) at the beginning of the paper at the level of the first 10-15 minutes of a mathematics departmental colloquium, placing the general area and the results into context. Such an introduction request has been a policy of the journal.

Symplectic topology is the study of smooth functions and their first derivatives, broadly understood. Infinitesimally, symplectic geometry studies the differential structure associated to the Lie algebra $\mathfrak {sp}_{2n}$, characterized as the geometry that preserves a non-degenerate skew-symmetric bilinear form. Along with the geometry that preserves volume, corresponding to ${\displaystyle {\mathfrak {sl}}_{n}}$, and (symmetric) inner products, corresponding to ${\displaystyle {\mathfrak {so}}_{n} }$, these constitute the three infinite families in the classification of complex simple Lie algebras. Locally, symplectic geometry adds the integrability condition that the non-degenerate 2-form is closed, resulting in a unique normal form near any point, up to symmetries preserving the symplectic 2-form. Globally, the intricacies of symplectic 2-forms and their symmetries remain an active subject of study, often referred to as symplectic topology. A salient instance of symplectic manifolds are cotangent bundles $T^*Q$ of smooth manifolds $Q$, whereby the smooth topology of $Q$ and its submanifolds is in part reflected by the intersection theory of Lagrangian submanifolds in $T^*Q$. Here, Lagrangian submanifolds are defined to be those submanifolds where the restriction of the ambient symplectic 2-form vanishes identically, and are of maximal dimension. In general, such Lagrangian submanifolds generalize the graphs $\mbox{gr}(df)\sse T^*Q$ of differentials of smooth functions $f:Q\lr\R$, in particular allowing for non-graphical behavior and parametric families of such functions.

In modern symplectic topology, Lagrangians and their intersection theory yield useful global invariants of symplectic manifolds. In addition, any closed symplectic manifold admits a symplectic divisor whose complement can be appropriately understood as the cotangent bundle of a possibly singular Lagrangian skeleton. The study of Lagrangian submanifolds is thus of central importance in our understanding of symplectic topology, e.g.~ the classification of Lagrangian submanifolds, up to either Hamiltonian isotopies or symplectomorphisms, is a desired goal in the current scientific context. This article examines one of the simplest models: understanding embedded exact Lagrangian surfaces of the local normal form $(\R^4,\omega_\st)$, with a fixed boundary condition given by a Legendrian knot. Intuitively, the main goal of the article is to find {\it global regular} functions on the derived moduli stack of such Lagrangian surfaces. Such functions allow for a better understanding of these intricate spaces (e.g.~sometimes allowing for the computation of their cohomology rings) and yield connections to other areas, including the study of cluster algebras. The main strategy for constructing such global functions is to first construct local regular functions, by using certain parallel transports, and then establish a criterion for these locally regular functions to be globally regular. The title of the manuscript refers to the criterion we present: if such a local function is {\it positive} with respect to a Lagrangian skeleton, then it extends to a global regular function.

%{\bf ADD BROAD DESCRIPTION HERE\\}

%%%%%%%%%%%%%%%%%%%%%%%%%%%%%%%%%%%%%%%%%%%%%%%%%%%%%%%%%%%%%%%%%%%%%%%%
%%%%%%%%%%%%%%%%%%%%%%%%%%%%%%%%%%%%%%%%%%%%%%%%%%%%%%%%%%%%%%%%%%%%%%%%
%%%%%%%%%%%%%%%%%%%%%%%%%%%%%%%%%%%%%%%%%%%%%%%%%%%%%%%%%%%%%%%%%%%%%%%%

\subsection{Summary} The object of this article will be to prove the global regularity of positive microlocal merodromies for Lagrangian fillings of Legendrian links. Microlocal merodromies are local functions defined on an open subset of the moduli of pseudo-perfect objects in the smooth dg-category of sheaves with singular support on a Legendrian link. Intuitively, these are functions defined on an open subset of the moduli of Lagrangian fillings. In important situations, such {\it local} functions can be extended to {\it global} regular functions on the entire moduli space. For instance, this is the case for cluster variables in braid varieties and, more generally, elements of canonical bases. This {\it local} to {\it global} regularity lies at the core of many recent results on Lagrangian fillings. Heretofore, it has been difficult to argue that such regular extensions exist, e.g.~it is already challenging in the case of cluster variables. This article provides a general criterion for such global regularity based on the study of $\L$-compressing systems. The conditions for this criterion are verifiable and based on geometric data.

Specifically, our main result is to show that such regular extensions exist if the relative cycle intersects positively an $\L$-compressing system. It applies to arbitrary Legendrian links and Lagrangian fillings (and arbitrary microlocal rank), both generalizing and independently recovering our previous results on global regularity of cluster variables. A key part of our argument is the construction of Hochschild cycles for the aforementioned category that lift these microlocal merodromies. The proof highlights the advantages of working within the dg-categorical framework when studying Legendrian links. Even in known cases, our argument bypasses previously required algebraic computations needed to argue for regularity and, using the derived stack of pseudo-perfect objects, provides a clearer categorical argument for regularity. In brief, apart from its greater level of generality, a strength of the result is that it provides new conceptual and geometric reasons for the global regularity of such functions.

\begin{center}
	\begin{figure}[H]
		\centering
		\includegraphics[scale=0.6]{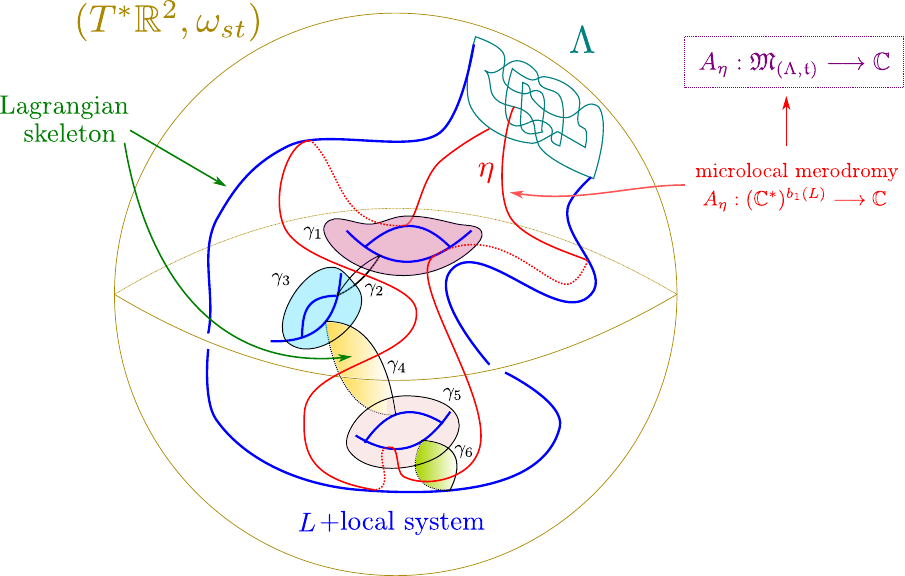}
		\caption{A Lagrangian filling $L$, in \textcolor{blue}{blue}, of $\La$, in \textcolor{cyan}{cyan}, with an $\L$-compressing system $\D=\{\g_1,\ldots,\g_6\}$. Relative cycle $\eta$ in \textcolor{red}{red}. In Theorem \ref{thm:main} we lift the microlocal merodromy from a {\it local} regular function to a Hochschild 0-chain of $\Sh^c_{\La,\tt}(\R^2)$. The latter allows us to construct a {\it global} regular function $A_\eta$ in $\Gamma(\M{(\La,\tt)},\SO_{\M{(\La,\tt)}})$.}
		\label{fig:Filling_Merodromy1}
	\end{figure}
\end{center}

%%%%%%%%%%%%%%%%%%%%%%%%%%%%%%%%%%%%%%%%%%%%%%%%%%%%%%%%%%%%%%%%%%%%%%%%
%%%%%%%%%%%%%%%%%%%%%%%%%%%%%%%%%%%%%%%%%%%%%%%%%%%%%%%%%%%%%%%%%%%%%%%%
%%%%%%%%%%%%%%%%%%%%%%%%%%%%%%%%%%%%%%%%%%%%%%%%%%%%%%%%%%%%%%%%%%%%%%%%

\subsection{Scientific context} Let $\La\sse(T^*_\infty\R^2,\xi_\st)$ be a Legendrian link in the ideal contact boundary of the cotangent bundle of $\R^2$, e.g.~any Legendrian link in a contact Darboux chart. The study of Lagrangian fillings of $\La$ is a pillar of low-dimensional contact topology, see e.g.~\cite{BourgeoisSabloffTraynor15,EkholmEtnyreNgSullivan13,EHK,EliashbergPolterovich96,EtnyreNg22_Survey,GKS_Quantization,Leverson16_AugRulings,NRSSZ,Polterovich_Surgery,STZ_ConstrSheaves}. Recently, the results from \cite{CasalsHonghao,CasalsGao24} and the trilogy \cite{CGGLSS,CW,CasalsZaslow} provided a better understanding of the moduli space $\M(\La)$ of Lagrangian fillings and its geometric structures. Note that, in addition to new results in contact topology, studying these moduli spaces $\M(\La)$ from this contact topological perspective has lead to a number of applications to cluster algebras, see e.g.~\cite{CGGS,CGGLSS,CLSW23,CW}. Now, a key technique to understand such moduli spaces $\M(\La)$ is the construction of regular functions on them. The main result of this article is a general construction of such regular functions coming from a microlocal version of parallel transport.

In a nutshell, from \cite[Section 4]{CW}, certain relative 1-cycles $\eta\in H_1(L,\La)$ in a Lagrangian filling $L$ (related to $\L$-compressing systems) allow us to define regular functions $A_\eta:T(L)\lr \C$ on an open subset $T(L)\sse\M(\La)$. If $\La$ is particular enough and certain combinatorics are present, e.g.~weaves in the case of \cite{CGGLSS}, grid plabic graphs in the case of \cite{CW} and 3D plabic graphs in the case of \cite{GLSBS}, one can sometimes argue that $A_\eta$ extends to a regular function on the entirety of $\M(\La)$. This has significant consequences, including the construction of cluster structures \cite{CGGLSS,CW}, holomorphic symplectic structures \cite{CGGS}, the existence of Lagrangian fillings in each cluster seed \cite{CasalsGao24} and geometric realizations of Donaldson-Thomas transformations \cite{CLSW23,CW}.

That said, the arguments employed in the above works are tied to the underlying combinatorics and do not generalize to Legendrian links that are not closures of positive braids. From the perspective of contact topology, it is desirable to find a general and more conceptual proof, based on symplectic geometry, that both applies to arbitrary Legendrian links and sheds geometric intuition on when and why global regularity holds. This article contains a first such proof.

%%%%%%%%%%%%%%%%%%%%%%%%%%%%%%%%%%%%%%%%%%%%%%%%%%%%%%%%%%%%%%%%%%%%%%%%
%%%%%%%%%%%%%%%%%%%%%%%%%%%%%%%%%%%%%%%%%%%%%%%%%%%%%%%%%%%%%%%%%%%%%%%%
%%%%%%%%%%%%%%%%%%%%%%%%%%%%%%%%%%%%%%%%%%%%%%%%%%%%%%%%%%%%%%%%%%%%%%%%

\subsection{Main result} 

Let $\La\sse (T^*_\infty\R^2,\xi_{st})$ be a Legendrian link and $\tt\sse\La$ a set of basepoints, with at least one basepoint per component of $\La$. We say $\La$ is $\tt$-pointed if such a choice $\tt$ of basepoints has been made. In Section \ref{sec:sheaf_cat_moduli} we introduce a refinement $\Sh^c_{\La,\tt}(\R^2)_0$ of the dg-category $\Sh_{\La}(\R^2)_0$ of compactly supported sheaves in $\R^2$ with singular support on $\La$. The objects of $\Sh^c_{\La,\tt}(\R^2)_0$ are categorically compact sheaves in $\Sh_{\La}(\R^2)_0$ together with trivializations of their microstalks at the basepoints $\tt$. By construction, $\Sh^c_{\La,\tt}(\R^2)_0$ is a smooth dg-category and we denote by $\M({\La,\tt})$ its moduli of pseudo-perfect objects, cf.~\cite[Section 3.1]{ToenVaquie07}. The derived stack $\M({\La,\tt})$ is the model we use for the moduli space of Lagrangian fillings of $(\La,\tt)$: intuitively, its classical closed points parametrize Lagrangian fillings of $\La$ endowed with local systems, trivialized at some boundary points. We will use the data of the microstalk trivializations in $\Sh^c_{\La,\tt}(\R^2)_0$ to obtain regular functions on $\M({\La,\tt})$, by using the parallel transport of the local systems along relative cycles. For precise details, see e.g.~\cite{CasalsLi,CW} and \cite{ToenVaquie07}. Now, the construction of \cite{ToenVaquie07} is such that there exists a map
\begin{equation}\label{eq:HO_map}
\HO:\HH_*(\SC)\lr \Gamma(\M_\SC,\SO_{\M_\SC})
\end{equation}
for any smooth dg-category $\SC$ and $\M_\SC$ its derived moduli stack, where $\HH_*(\SC)$ denotes the Hochschild chains of $\SC$, cf.~\cite[Sections 4~\&~5]{BravDycker21}. More generally, Hochschild chains in $\SC$ map to differential forms on $\M_\SC$, cf.~\cite[Prop. 5.2]{BravDycker21}. A key construction in this note is to enhance microlocal merodromies to the categorical level, obtaining chains in the domain of the map $\HO$ in (\ref{eq:HO_map}) for $\SC:=\Sh^c_{\La,\tt}(\R^2)$.

Let $T\sse \La\setminus\tt$ be another set of basepoints such that there exists exactly one basepoint of $\tt$ in each interval of $\La\setminus T$. We say $\La$ is $(\tt,T)$-pointed if such choices of basepoints $\tt$ and $T$ have been made. Consider an embedded exact Lagrangian filling $L\sse (T^*\R^2,\la_\st)$ of $\La$ endowed with an $\L$-compressing system $\D=\{\g_1,\ldots,\g_{b_1(L)}\}$, as defined in \cite[Section 4]{CasalsGao24}. Succinctly, a (complete) $\L$-compressing system is a set of disjoint Lagrangian disks $\{D_1,\ldots,D_{b_1(L)}\}$, each properly embedded in the complement $T^*\R^2\setminus L$ of $L$ with immersed boundaries $\g_i:=\dd \overline{D_i}\sse L$ and such that the union $\bL:=L\cup (D_1\cup\ldots D_{b_1(L)})\sse T^*\R^2$ is a relative Lagrangian skeleton for $(T^*\R^2,\La)$. See also \cite[Section 1]{CW}.\footnote{The assumption that $\bL:=L\cup (D_1\cup\ldots D_{b_1(L)})\sse T^*\R^2$ is a relative Lagrangian skeleton for $(T^*\R^2,\La)$ is not explicitly stated in \cite[Section 1]{CW}, but all known examples of $\L$-compressing systems, in the sense of \cite{CW}, satisfy this assumption. See for example \cite{CasalsLagSkel} and \cite[Proposition 5.2]{BreenRoyWang} for a proof.} By definition, a relative cycle $\eta:[0,1]\lr(L,\tt)$ is said to be $\D$-positive if $\langle\eta,\gamma_i\rangle\geq0$ for all $i\in[1,b_1(L)]$.

\noindent Let ${T}(L)\sse \M({\La,\tt})$ be the open set ${T}(L)\cong \Loc(L,T)$ given by the embedded exact Lagrangian filling $L$, see e.g.~\cite[Appendix B]{CasalsLi} or \cite[Section 2]{CW} and references therein. By construction, the microlocal monodromy $A_\eta\in \Gamma({T}(L),\SO_{{T}(L)})$ of $\eta$ is a regular function on the open set ${T}(L)$. It might or it might not extend to a regular function on the entire $\M({\La,\tt})$. There are instances where such local merodromies do {\it not} extend globally, cf.~\cite[Section 11]{CGGLSS}. Our main result allows us to establish a criterion for such a global extension to exist. It reads as follows:

\begin{thm}[``Main extension result'']\label{thm:main}
Let $\La\sse(T^*_\infty\R^2,\xi_\st)$ be a $(\tt,T)$-pointed Legendrian link and $L\sse (T^*\R^2,\la_\st)$ an embedded exact Lagrangian filling of $\La$ endowed with an $\L$-compressing system $\D$.

\noindent Then, for any $\D$-positive relative 1-cycle $\eta$ in the relative pair $(L,T)$, there exists a Hochschild 0-cycle $H_\eta\in \HH_0(\Sh^c_{\La,\tt}(\R^2))$ whose associated regular function
$\HO(H_\eta)\in \mathrm{H}^0\Gamma(\M({\La,\tt}),\SO_{\M({\La,\tt})})$
coincides with the trace of the microlocal merodromy along $\eta$ when restricted to ${T}(L)\sse \M({\La,\tt})$.
\end{thm}

\noindent Theorem \ref{thm:main} is proven in Section \ref{ssec:main_proof}. An immediate consequence of Theorem \ref{thm:main} is that, in particular, microlocal merodromies along positive relative cycles define {\it global} regular functions. In brief, it shows that {\it local positive} merodromies are always {\it globally} regular. Specifically, following the notation in \cite[Section 4]{CW}, Theorem \ref{thm:main} implies:

\begin{cor}[``Local merodromies along positive cycles are global'']\label{cor:main}
Let $\La\sse(T^*_\infty\R^2,\xi_\st)$ be a $(\tt,T)$-pointed Legendrian link and $L\sse (T^*\R^2,\la_\st)$ an embedded exact Lagrangian filling of $\La$ endowed with an $\L$-compressing system $\D$.

\noindent Then, for any $\D$-positive relative 1-cycle $\eta$ in the relative pair $(L,T)$, the trace of its microlocal merodromy $A_\eta\in \mathrm{H}^0\Gamma({T}(L),\SO_{{T}(L)})$ extends to a global regular function
$A_\eta\in \mathrm{H}^0\Gamma(\M({\La,\tt}),\SO_{\M({\La,\tt})})$.
\end{cor}

\noindent Corollary \ref{cor:main} implies the core of the results in \cite{CW} and \cite{CGGLSS}, without essentially performing any computations or relying on combinatorics. Indeed, the cluster variables defined in  \cite[Section 4]{CW} and \cite[Section 5]{CGGLSS} are particular instances of microlocal merodromies along positive relative cycles, as the relative cycles $\eta$ are the Poincar\'e duals of the basis given by Lusztig cycles. In addition, \cite{CasalsGao24} implies that for these cluster schemes, all elements of the theta canonical basis from \cite{GHKK} come from such microlocal merodromies. Corollary \ref{cor:main} readily establishes their {\it global} regularity via a verifiable {\it local} condition. Note that such extensions (to the corresponding irreducible component) are unique because $T(L)$ is an open subset of $\M({\La,\tt})$. We refer to holonomies along positive cycles as {\it positive} holonomies: the title of this article thus records the statement of Corollary \ref{cor:main}.

\begin{remark}\label{rmk:computations}
The microlocal merodromies appearing in \cite[Section 4]{CW} and \cite[Section 5]{CGGLSS} are a priori only rational functions on the coordinate rings in the ambient space. Key parts of those articles consist of algebraically manipulating such rational expressions, guided by the underlying combinatorics, %and applying the codimension-2 argument 
to prove they are in fact regular functions when restricted to braid varieties.\hfill$\Box$
\end{remark}

Corollary \ref{cor:main} has the important advantage of avoiding the computations mentioned in Remark \ref{rmk:computations}. In particular, Theorem \ref{thm:main} and Corollary \ref{cor:main} apply to arbitrary microlocal ranks and arbitrary Legendrian links in contact boundaries, not just to the moduli space of microlocal rank~1 sheaves on positive braid closures used in \cite{CGGLSS,CW}, cf.~\cref{ssec:generalization}.\\

%%%%%%%%%%%%%%%%%%%%%%%%%%%%%%%%%%%%%%%%%%%%%%%%%%%%%%%%%%%%%%%%%%%%%%%%
%%%%%%%%%%%%%%%%%%%%%%%%%%%%%%%%%%%%%%%%%%%%%%%%%%%%%%%%%%%%%%%%%%%%%%%%
%%%%%%%%%%%%%%%%%%%%%%%%%%%%%%%%%%%%%%%%%%%%%%%%%%%%%%%%%%%%%%%%%%%%%%%%

\subsection{Outline of the argument}\label{ssec:ingredients} The proof of \cref{thm:main} is rather conceptual and it can be helpful to emphasize the three key ingredients that feature in the proof. We thank the referee for suggesting that we add the following list in this introduction. The main ingredients can be summarized as follows:

\begin{enumerate}
    \item The moduli of pseudoperfect objects $\M_\SC$ of a smooth dg-category $\SC$, as introduced and studied by B.~To\"en and M.~Vaqui\'e in \cite{ToenVaquie07}. By \cite[Corollary 4.2]{ToenVezzosiHKR}, the dg-category $\SC$ and its derived stack $\M_\SC$ are equipped with a map $\HO:\HH_*(\SC)\lr \Gamma(\M_\SC,\SO_{\M_\SC})$ from the Hochschild chains of $\SC$ to the derived global functions on $\M_\SC$. Therefore, to construct functions on $\M_\SC$ it suffices to construct Hochschild chains on $\SC$. In the proof of \cref{thm:main}, this general framework is applied to the dg-category $\SC:=\Sh^c_{\La,\tt}(\R^2)$ that we introduce in \cref{ssec:dgcat_of_sheaves}, and the necessary details on $\M_\SC$ and the $\HO$ map are provided in \cref{ssec:derived_stacks} and \cref{ssec:general_framework}, respectively.\\

    \item The local model of the Lagrangian skeleton near a positive path, as depicted in \cref{fig:Filling_Merodromy3}. Given a generic relative chain in the union of a Lagrangian filling and an $\L$-compressing system, such as the relative chain $\eta$ drawn in red in \cref {fig:Filling_Merodromy1}, its neighborhood retracts to a horizontal line (or circle) with a sequence of vertical half-rays attached to it, some pointing upwards and some downwards. For a positive path, these half-rays all point upwards, as in \cref{fig:Filling_Merodromy3}. This local model is studied in detail in \cref{ssec:HH0_near_rel_cycle}, where Hochschild chains for its associated dg-category are computed, and in \cref{ssec:computation_HO_near_relative_cycle}, where the corresponding $\HO$ map is described in \cref{prop:HO_in_Rmod}.\\

    \item The corestriction functors and their compatibility with Hochschild chains, as studied in Sections \ref{ssec:corestriction}, \ref{ssec:KS_stack_corestriction_relative} and \ref{subsec:corestrict-HH}. The existence of corestriction functors for the Kashiwara-Schapira stack is established in \cref{ssec:corestriction}, and in the pointed setting in \cref{ssec:KS_stack_corestriction_relative}. These corestriction functors, which are left adjoints to restriction functors, allow for a local-to-global transition from the local computation in the local model from (2) above, near a relative chain, to a global computation on the Lagrangian skeleton. \cref{subsec:corestrict-HH} proves that such functors are compatible with Hochschild chains and the $\HO$ maps, thus allowing for the local computation of the $\HO$ map to extend globally.
\end{enumerate}

\noindent Having established the needed results for each of these three main ingredients in Sections \ref{sec:sheaf_cat_moduli}, \ref{sec:corestriction} and \ref{sec:corestriction_HH}, \cref{thm:main} is proven in \cref{ssec:main_proof}.

%%%%%%%%%%%%%%%%%%%%%%%%%%%%%%%%%%%%%%%%%%%%%%%%%%%%%%%%%%%%%%%%%%%%%%%%
%%%%%%%%%%%%%%%%%%%%%%%%%%%%%%%%%%%%%%%%%%%%%%%%%%%%%%%%%%%%%%%%%%%%%%%%
%%%%%%%%%%%%%%%%%%%%%%%%%%%%%%%%%%%%%%%%%%%%%%%%%%%%%%%%%%%%%%%%%%%%%%%%

\noindent{\bf Acknowledgements}: We are grateful to the referee, whose edits and suggestions have improved the article. R.~Casals thanks the hospitality of the Institute of Advanced Study at Princeton, where a significant part of this article was written. R.~Casals is supported by the National Science Foundation under grants DMS-2505760 and DMS-1942363, a Sloan Research Fellowship of the {\color{black} Alfred P. Sloan Foundation} and a UC Davis College of L\&S Dean's Fellowship. W.~Li is partially supported by the AMS-Simons Travel Grant.\hfill$\Box$

%%%%%%%%%%%%%%%%%%%%%%%%%%%%%%%%%%%%%%%%%%%%%%%%%%%%%%%%%%%%%%%%%%%%%%%%
%%%%%%%%%%%%%%%%%%%%%%%%%%%%%%%%%%%%%%%%%%%%%%%%%%%%%%%%%%%%%%%%%%%%%%%%
%%%%%%%%%%%%%%%%%%%%%%%%%%%%%%%%%%%%%%%%%%%%%%%%%%%%%%%%%%%%%%%%%%%%%%%%

\noindent{\bf Notation}: Given $n\in\N$, we denote $[1,n]:=\{1,\ldots,n\}$ and $k$ denotes a commutative unital ring. Throughout the article, categories will be dg-derived categories unless otherwise specified, i.e.~we work with categories enriched over chain complexes and localized at acyclic complexes. For instance, $\modk$ denotes the dg-derived category of chain complexes over $k$, and derived global sections and derived tensor product are still denoted by $\Gamma$ and $\otimes$. Similarly, $\Loc(\La)$ denotes the dg-derived category of local systems on $\La$, cf.~\cite[Appendix A]{CasalsLi} and references therein. We write $\mbox{dg-cat}_k$ for the $\infty$-category of small dg-categories over $k$, and $\mbox{dg-Cat}_k$ for the $\infty$-category of well-generated dg-categories over $k$, obtained from the corresponding model categories \cite{Tabuada05_ModelStructure,Tabuada_Wellgenerated}. In terms of derived algebraic geometry, we work within the framework of $D^-$-stacks, cf.~ \cite{HAGI,HAGII}, and specifically \cite{ToenVaquie07} for the derived moduli stack of pseudo-perfect objects.\hfill$\Box$

%% file: 2_SheafCat.tex
\section{Dg-categories and moduli of sheaves with Legendrian singular support}\label{sec:sheaf_cat_moduli} This section introduces the dg-categories and derived stacks used to state and prove Theorem \ref{thm:main}. The new concepts are defined in Definitions \ref{def:cat_decorated_sheaves} and \ref{def:stack_decorated_sheaves} in Subsections \ref{ssec:dgcat_of_sheaves} and \ref{ssec:derived_stacks}, respectively. Subsections \ref{ssec:t_action_on_moduli} and \ref{ssec:t_action_rank1} discuss an intrinsic action introduced by incorporating basepoints, relating the decorated and undecorated dg-categories. Section \ref{ssec:examples1} provides examples of such derived stacks and their actions for some families of Legendrian links. 
%%%%%%%%%%%%%%%%%%%%%%%%%%%%%%%%%%%%%%%%%%
%%%%%%%%%%%%%%%%%%%%%%%%%%%%%%%%%%%%%%%%%%
%%%%%%%%%%%%%%%%%%%%%%%%%%%%%%%%%%%%%%%%%%
\subsection{Dg-categories of sheaves with Legendrian singular support}\label{ssec:dgcat_of_sheaves} Let $\La\sse (T^*_\infty\R^2,\xi_{st})$ be a Legendrian link and $T\sse\La$ a set of basepoints, with at least one basepoint per component of $\La$. Consider a set of basepoints $\tt=\{t_1,\ldots,t_{|\tt|}\}\sse\La\setminus T$ such that each component of $\La\setminus T$ contains a unique basepoint. Since the data of $T$ is equivalent to the data of $\tt$, we also refer to $\tt$ as a set of basepoints and the roles of $T$ and $\tt$ are interchangeable. A Legendrian link with such choice of basepoints $\tt$ and $T$ is said to be $(\tt,T)$-pointed.

Let $\Sh_\La(\R^2)$ be the dg-derived category of sheaves of $k$-modules in $\R^2$ with singular support contained in $\La$, and $\Sh_\La(\R^2)_0$ the dg-derived category of compactly supported sheaves in $\R^2$ with singular support contained in $\La$, cf.~\cite[Chapter V]{KashiwaraSchapira_Book} or \cite[Section 1]{GKS_Quantization}. See also Appendix \ref{sec:appendix} or \cite[Appendix A]{CasalsLi} for details and further discussions.

The dg-categories $\Sh_\La(\R^2)$ and $\Sh_\La(\R^2)_0$ are Legendrian isotopy invariants of $\La$, cf.~\cite[Section 3]{GKS_Quantization}. Their information only allows for the definition of microlocal monodromies on a given exact Lagrangian filling $L$ of $\La$; microlocal monodromies are defined using the fully faithful embedding $\Loc(L) \hookrightarrow \Sh_\La(\bR^2)_0$, cf.~Subsection \ref{ssec:Lagrangian_fillings_KS_stack} below or \cite{JinTreumann}. These monodromies are only rational functions on the moduli of pseudo-perfect objects (intuitively, on the moduli of Lagrangian fillings of $\La)$, and are typically not regular. In particular, it is challenging to extract geometric information from them, cf.~ e.g.~\cite[Sections 2.8 \& 4.4]{CW} and \cite[Section 8]{CGGLSS} for further discussions.

Microlocal merodromies, allowing for more general parallel transports along relative cycles with decorated ends, have proven more useful, see e.g.~\cite[Section 4]{CW}, \cite[Section 5]{CGGLSS} and \cite[Section 6]{CasalsLi}. Therefore, we now introduce in Definition \ref{def:cat_decorated_sheaves}, and see also Definition \ref{def:stack_decorated_sheaves}, an enhancement of $\Sh_\La(\bR^2)$ or $\Sh_\La(\R^2)_0$ which allows for microlocal merodromies to be defined and used.\footnote{The results and arguments in this section hold for the dg-categories $\Sh_\La(\bR^2)$ and $\Sh_\La(\bR^2)_0$. In fact, by \cite[Cor.~4.22]{GPS3}, the inclusion $\Sh_\La(\bR^2)_0 \hookrightarrow \Sh_\La(\bR^2)$ has a left adjoint that preserves compact objects and thus the derived moduli stack of pseudo-perfect objects of $\Sh_\La(\bR^2)_0$ is an open and closed substack of the corresponding moduli for $\Sh_\La(\bR^2)$. We state the results for $\Sh_\La(\bR^2)_0$ only since it better fits into the comparison with the Kashiwara-Schapira stacks for $\bL$-compressing systems in Section \ref{sec:corestriction}.}\\

Consider the following functor $m_{\La,\tt}$, which records the microstalks at the basepoints $t_i\in \tt$ in the co-direction $\xi_{t_i}\in T_{t_i}^*\R^2$ of the co-normal lift of $\La$. It is defined as follows

$$\displaystyle m_{\La,\tt}:\Sh_\La(\R^2)_0\lr \prod_{i=1}^{|\tt|}\modk,\quad \SF\longmapsto m_{\La,\tt}(\SF):=m_{{\La} \times \bR_+}(t_i,\xi_{t_i})(\SF),$$
where ${\La} \times \bR_+\sse T^*\R^2$ is the conormal cone associated to $\La\sse T^*_\infty\R^2$ and $m_{{\La} \times \bR_+}$ is the microstalk functor of ${\La} \times \bR_+$, e.g.~ as defined in \cite[Definition 3.8]{Nadler16}. See also \cite[Section 6]{KashiwaraSchapira_Book} for details on the microlocal theory of sheaves. Equivalently, $m_{\La,\tt}$ assigns to an object $\SF$ the stalks of the local system $m_\La(\SF)\in\Loc(\La)$ at the points in $\tt$, where $m_\La:\Sh_\La(\R^2)_0\lr \Loc(\La)$ is the microlocalization functor, cf.~Definition \ref{def:microlocalization_functor} below or \cite[Appendix A]{CasalsLi} and references therein.\footnote{Note that here we have implicitly chosen a quasi-equivalence $\mu\mbox{sh}_\La(\La)\cong\Loc(\La)$ between the global sections of the Kashiwara-Schapira stack $\mu\mbox{sh}$ and the dg-derived category of local systems $\Loc(\La)$ on $\La$. This is possible in the case where $\La=S^1\sqcup\ldots\sqcup S^1$ by \cite[Chapter 10]{Guillermou23_SheafSummary}. This is also discussed later in Section \ref{sec:corestriction}.} In this latter perspective, we are implicitly using an equivalence
$$\displaystyle\prod_{i=1}^{|\tt|}\modk\cong \displaystyle\prod_{i=1}^{|\tt|}\Loc(\{t_1,\ldots,t_{|\tt|}\})\cong\Loc(\tt).$$

Consider also the diagonal dg-functor
$$\displaystyle\Delta:\modk\lr\prod_{i=1}^{|\tt|}\modk,$$
which is independent of $\La$ and $T$. The microstalk functor $m_{\La,\tt}$ admits a left adjoint $m^\ell_{\La,\tt}$ because it preserves products, see e.g.~\cite[Section 3.6]{Nadler16} or \cite[Lemma 4.13]{GPS3}. Similarly, $\Delta$ admits a left adjoint $\Delta^\ell$ given by the coproduct. Both functors $m_{\La,\tt}$ and $\Delta$ preserve coproducts, and thus $m^\ell_{\La,\tt}$ and $\Delta^\ell$ preserve compact objects see e.g.~\cite[Theorem 4.14]{GPS3}.

Consider the full dg-subcategories $\Sh^c_\La(\R^2)_0\sse\Sh_\La(\R^2)_0$ and $\modk^c\sse\modk$ of compact objects. The category $\modk^c$ is the dg-subcategory of perfect complexes, i.e.~$\Perf(k)$. These are both smooth dg-categories, cf.~\cite[Corollary 4.26]{GPS3}. In fact, $\Sh_\La^c(\R^2)_0$ is even of finite type by the results from \cite{Nadler17,Starkston} (or by virtue of being equivalent to a wrapped Fukaya category, as proven in \cite{GPS3}). Note that $\Sh_\La^c(\R^2)_0$ is a small dg-category, as it is the subcategory of compact objects of the compactly generated category $\Sh_\La(\R^2)_0$.

\begin{definition}\label{def:cat_decorated_sheaves} Let $\La\sse (T^*_\infty\R^2,\xi_{st})$ be a Legendrian link and $\tt\sse\La$ a set of basepoints, with at least one basepoint per component of $\La$. By definition, the dg-category $\Sh_{\La,\tt}(\R^2)_0$ is the homotopy colimit of the diagram 
\begin{center}  
    \begin{tikzcd} 
    \displaystyle\prod_{i=1}^{|\tt|}\modk\ar[r,"m_{\La,\tt}^\ell"] \ar[d,"\Delta^\ell" left] &  \Sh_\La(\R^2)_0 \\
        \modk&
    \end{tikzcd}
\end{center}
We refer to $\Sh_{\La,\tt}(\R^2)_0$ as the category of compactly supported sheaves with singular support in $(\La,\tt)$.

\hfill$\Box$
\end{definition}

\noindent Here the homotopy colimit is taken in the $\infty$-category $\mbox{dg-Cat}_k$ of well generated dg-categories over $k$, cf.~\cite{Tabuada_Wellgenerated}. Intuitively, $\Sh_{\La,\tt}(\R^2)_0$ captures compactly supported sheaves with singular support on $\La$ with the additional data of their microstalks at the basepoints of $\tt$ and a common identification of these microstalks. %Note also that the functors in the diagram of Definition \ref{def:cat_decorated_sheaves} preserve compact objects. 

%%%%%%%%%%%%%%%%%%%%%%%%%%%%%%%%%%%%%%%%%%
%%%%%%%%%%%%%%%%%%%%%%%%%%%%%%%%%%%%%%%%%%
%%%%%%%%%%%%%%%%%%%%%%%%%%%%%%%%%%%%%%%%%%

\subsection{Derived stacks of sheaves with Legendrian singular support}\label{ssec:derived_stacks} Let $\SC$ be a small dg-category and consider the $D^-$-stack $\M_\SC$ of pseudo-perfect objects, as introduced in \cite[Section 3.1]{ToenVaquie07}. It is defined by the functor of points
$$\M_\SC: \mbox{sCAlg}_k\lr\mbox{sSet},\quad \M_\SC(A)=\mbox{Map}_{\operatorname{dg-cat}_k}(\SC^{op},\Perf(A)).$$
Here $\mbox{sCAlg}_k$ is the category of simplicial commutative $k$-algebras (denoted $sk\mbox{-CAlg}$ in \cite[Section 2.3]{ToenVaquie07}), $\mbox{sSet}$ the category of simplicial sets, $\Perf(A)$ is the dg-category of
perfect A-modules (denoted $\hat{A}_{pe}$ in \cite[Section 2.4]{ToenVaquie07}), and $\mbox{Map}_{\operatorname{dg-cat}_k}$ denotes the mapping space of a model structure for the category of small dg-categories. See \cite[Section 3]{ToenVaquie07} for the necessary details.\footnote{Note that, in zero characterstic, $\mbox{sCAlg}$ is equivalent to $\mbox{cdga}_{\leq 0}$, via the appropriate version of the Dold-Kan correspondence. Therefore, in zero characteristic, the inputs for (functor of points of) the derived stack $\M_\SC$ can be taken to be non-positively graded commutative dg-algebras.} The assignment $\SC\mapsto\M_\SC$ defines a functor
\begin{equation}\label{eq:modulistack_functor}
\M:\mbox{Ho}(\mbox{dg-cat}_k)^{op}\lr D^-\mbox{St}(k)
\end{equation}
between the opposite of the homotopy category $\mbox{Ho}(\mbox{dg-cat}_k)$ of dg-categories and the category $D^-\mbox{St}(k)$ of $D^-$-stacks, the functor being enriched over the homotopy category of $\mbox{sSet}$. In particular, given a dg-functor $f: \cC \lr \cD$, there is a map $\M(f): \M(\cD) \lr \M(\cC)$, which sends a pseudo-perfect object $\cD^{op} \lr \Perf(k)$ to its pull-back to $\cC^{op} \lr \Perf(k)$.

By \cite[Prop. 3.4]{ToenVaquie07}, the functor $\M$ admits a left adjoint and therefore $\M$ preserves homotopy limits. In particular, it sends a homotopy pullback in $\mbox{Ho}(\mbox{dg-cat}_k)^{op}$ to a homotopy pullback in $D^-\mbox{St}(k)$. Since homotopy pullbacks in $\mbox{Ho}(\mbox{dg-cat}_k)^{op}$ are homotopy pushouts in $\mbox{Ho}(\mbox{dg-cat}_k)$, the functor $\M$ applied to the homotopy pushout
\begin{center}
\begin{equation}\label{def:htpy_pushout_dgcats}
\begin{tikzcd} 
    \displaystyle\prod_{i=1}^{|\tt|}\Perf(k) \ar[r,"m_{\La,\tt}^\ell"] \ar[d,"\Delta^\ell" left] &  \Sh^c_\La(\R^2)_0 \ar[d] \\
         \Perf(k)\ar[r]& \textcolor{blue}{\Sh^c_{\La,\tt}(\R^2)_0}
    \end{tikzcd}
\end{equation}
\end{center}
from Definition \ref{def:cat_decorated_sheaves}, after restricting to compact objects, yields a homotopy pullback
\begin{center}
\begin{equation}\label{def:htpy_pullback_stacks}
\begin{tikzcd} 
\M\left(\displaystyle\prod_{i=1}^{|\tt|}\Perf(k)\right) &  \M(\Sh^c_\La(\R^2)_0) \ar[l,"\M(m_{\La,\tt}^\ell)" above] \\
    \M(\Perf(k)) \ar[u,"\M(\Delta^\ell)"] &  \textcolor{blue}{\M(\Sh^c_{\La,\tt}(\R^2)_0)}.\ar[l] \ar[u]
    \end{tikzcd}
\end{equation}
\end{center}
Both in the homotopy pushout (\ref{def:htpy_pushout_dgcats}) and in the homotopy pullback (\ref{def:htpy_pullback_stacks}) we have highlighted the homotopy colimit and homotopy limit in color blue, for clarity. By \cite[Thm.~3.21]{Nadler16} or \cite[Cor.~4.23]{GPS3}, via the Yoneda embedding, the pseudo-perfect objects in $\Sh^c_\La(\R^2)_0$ are those sheaves in $\Sh^c_\La(\R^2)_0$ with perfect stalks. Thus, the map 
$$\M(\Sh^c_\La(\R^2)_0) \lr \M\left(\prod_{i=1}^{|\tt|}\Perf(k)\right)$$
on $k$-points is given by the microstalk functor $m_{\Lambda,\tt}$ at the basepoints of $\tt$. Similarly, identifying pseudo-perfect objects in $\Mod(k)$ with $\Perf(k)$ by the Yoneda embedding, the map
$$\M(\Perf(k)) \lr \M\left(\prod_{i=1}^{|\tt|}\Perf(k)\right)$$
on $k$-points is also given by the diagonal map $\Delta$. To ease notation, we denote $\M(\La):=\M(\Sh^c_\La(\R^2)_0)$.

\begin{definition}\label{def:stack_decorated_sheaves} Let $\La\sse (T^*_\infty\R^2,\xi_{st})$ be a Legendrian link and $\tt\sse\La$ a set of basepoints, with at least one basepoint per component of $\La$. By definition, the $D^-$-stack $\M({\La,\tt})$ is the moduli $\M(\Sh^c_{\La,\tt}(\R^2)_0)$ of pseudo-perfect objects of the smooth dg-category $\Sh^c_{\La,\tt}(\R^2)_0$. We refer to $\M({\La,\tt})$ as the moduli of (compactly supported) sheaves with singular support in $(\La,\tt)$.% with framing at the base point set $\tt$.
\hfill$\Box$
\end{definition}

\noindent Note that $\M({\La,\tt})$ depends on the choice of $\tt$, e.g.~the dimension of its $0$-truncation $t_0\M({\La,\tt})$ depends on the cardinality of $\tt$. That said, there is a natural choice for the cardinality of $\tt$ given $\La$: namely, we can always choose exactly one basepoint per component, i.e.~$|\tt|=|\pi_0(\La)|$. For each fixed microlocal rank $n \in \bN$, the corresponding connected component of $\M_n({\La}) \subset \M({\La})$ is a (stack) quotient of a component $\M_n({\La,\tt}) \subset \M({\La,\tt})$ under an algebraic action. See Lemma \ref{lem:microlocalmonodromy_via_stalks} for details and note that this action is typically not free.

%%%%%%%%%%%%%%%%%%%%%%%%%%%%%%%%%%%%%%%%%%
%%%%%%%%%%%%%%%%%%%%%%%%%%%%%%%%%%%%%%%%%%
%%%%%%%%%%%%%%%%%%%%%%%%%%%%%%%%%%%%%%%%%%

\subsection{A $\tt$-action on the derived moduli stacks of sheaves}\label{ssec:t_action_on_moduli} The freedom of the choice for the location of the basepoints $\tt\sse\La$ yields interesting automorphisms of $\M(\Sh^c_{\La,\tt}(\R^2)_0)$. Though these automorphisms are not strictly needed to prove Theorem \ref{thm:main}, they provide a conceptually useful perspective on the torus actions studied in \cite{CGGS,CGGLSS} and, in some cases, allow for neater descriptions of the derived stacks $\M{(\La,\tt)}$ and $\M(\La)$, see Subsection \ref{ssec:examples1}. These automorphisms are constructed as follows.

Given a set of base points $\tt \subset \Lambda$, let us describe an object in $\Sh_{\La,\tt}(\bR^2)_0$ by a pair $(\SF,(\phi_i))$ consisting of a sheaf $\SF$ and a set of common trivializations $(\phi_i)_{i\in[1,|\tt|]}$ for the microstalks of $\SF$ at the basepoints in $\tt$, where $\phi_i: V \lr m_{\Lambda,t_i}(\SF)$ are quasi-isomorphisms from a given fixed chain complex $V$. Consider the $\infty$-group of automorphisms 
$$G(\tt):=\Aut\left(\displaystyle\prod_{i=1}^{|\tt|}\Perf(k)\right),$$
understood as a group object in $\infty$-groupoids, see e.g.~\cite[Section 6.1.2]{HTT} or \cite[Chapter V]{GoerssJardine99_SHT}. Such automorphisms yield automorphisms of $\Sh_{\La,\tt}^c(\bR^2)_0$ that reparametrize the trivialization at the microstalks, as made precise in the following:

\begin{lemma}\label{lem:microlocalmonodromy_via_stalks}
    Let $\La \subset (T^*_\infty\bR^2, \xi_{st})$ be a Legendrian link and $\tt \subset \La$ a set of basepoints. Then there is a natural morphism of $\infty$-groups
    \[
    P:G(\tt) \lr \Aut\left(\Sh_{\La,\tt}^c(\bR^2)_0\right)
    \]
    such that $\displaystyle(c_i)\in G(\tt)$ acts on an object $(\SF,(\phi_i))\in \Sh_{\La,\tt}^c(\bR^2)_0$ by
    $$P_{(c_i)}(\mathscr{F}, (\phi_i)) = (\mathscr{F}, (c_i \circ \phi_i)).$$
\end{lemma}
\begin{proof}
The homotopy colimit diagram (\ref{def:htpy_pushout_dgcats}) of dg-categories reads
    \[\begin{tikzcd} 
    \displaystyle\prod_{i=1}^{|\tt|}\Perf(k) \ar[r,"m_{\La,\tt}^\ell"] \ar[d,"\Delta^\ell" left] &  \Sh^c_\La(\R^2)_0 \ar[d] \\
    \Perf(k) \ar[r]& \Sh^c_{\La,\tt}(\R^2)_0.
    \end{tikzcd}\]
By construction, $G(\tt)$ acts on the upper-left category. The $\infty$-group morphism $P$ is built by extending this action to $\Perf(k)$ and $\Sh^c_\La(\R^2)_0$ so that $\Delta^\ell$ and $m_{\La,\tt}^\ell$ are respectively equivariant, thus inducing an action on the homotopy pushout $\Sh^c_{\La,\tt}(\R^2)_0$.  For that, we first consider the trivial action of $G(\tt)$ on the bottom left dg-category $\Perf(k)$. Now, for any linear map $\phi: V \lr W$ and $(c_i)\in G(\tt)$, the corresponding map $\phi^{|\tt|}: V^{|\tt|} \lr W^{|\tt|}$ induced by the diagonal functor is homotopic to its conjugation $(c_i^{-1}) \circ \phi^{|\tt|} \circ (c_i): V^{|\tt|} \lr W^{|\tt|}$. Therefore, the natural action of $G(\tt)$ on the domain of $\Delta^\ell$ and the trivial action of $G(\tt)$ on its codomain are compatible. Second, similar considerations apply to $m_{\La,\tt}^\ell$. Indeed, by construction, this functor factors as
$$m_{\La,\tt}^\ell: \displaystyle\prod_{i=1}^{|\tt|}\Perf(k) \lr \Loc^c(\Lambda) \lr \Sh_\La^c(\bR^2)_0,$$
and we consider the trivial action of $G(\tt)$ on $\Loc^c(\Lambda)$ and $\Sh_\La^c(\bR^2)_0$. As before, given any morphism of local system $\phi: (V, \mu) \lr (W, \eta)$ and $(c_i)\in G(\tt)$, the morphism $\phi^{|\tt|}: V^{|\tt|} \lr W^{|\tt|}$ induced by the restriction functor to $\tt$ is homotopic to the morphism $(c_i^{-1}) \circ \phi^{|\tt|} \circ (c_i): V^{|\tt|} \lr W^{|\tt|}$. Hence the natural action of $G(\tt)$ on the domain of $m_{\La,\tt}^\ell$ is compatible with the trivial action of $G(\tt)$ on its codomain.

\noindent Since the dg-nerves of the three dg-categories involved\footnote{Here we use again that $m_{\La,\tt}^\ell$ factors through $\Loc^c(\La)$.} in the homotopy colimit diagram (\ref{def:htpy_pushout_dgcats}), i.e.~
$$\displaystyle\prod_{i=1}^{|\tt|}\Perf(k),\,\Perf(k)\,\mbox{ and }\,\Loc^c(\La),$$ are all 1-dimensional simplicial sets, the higher coherence data is automatically compatible under the above functors. Therefore, we obtain an action on their homotopy colimit $\Sh^c_{\La,\tt}(\bR^2)_0$.
    
\noindent The formula for the action on objects follows by construction. Indeed, since an object $(\SF,(\phi_i))\in \Sh^c_{\La,\tt}(\bR^2)_0$ is determined by a sheaf $\mathscr{F} \in \Sh_{\La}(\bR^2)_0$ together with a common trivialization $(\phi_i)$ of its microstalks, via
    $$m_{\La,\tt}: \Sh_\La(\bR^2)_0 \lr \Loc(\tt) \xrightarrow{\sim} \displaystyle\prod_{i=1}^{|\tt|}\Mod(k),$$
    any element $(c_i)\in G(\tt)$ acts as the identity on the sheaf in $\Sh_\La(\bR^2)_0$ and acts on the trivialization maps $(\phi_i)$ by left multiplication, as this is how $G(\tt)$ acts componentwise on (products of) perfect modules.
\end{proof}

\begin{remark} In the proof of \cref{lem:microlocalmonodromy_via_stalks}, we can alternatively verify the compatibility of the higher coherence data of the action as follows. Consider the universal Cartesian filtration over $\mathrm{B}G(\tt)$, whose fiber is the dg-nerve of $\prod_{i=1}^{|\tt|}\Perf(k)$ by straightening, see e.g.~ \cite[Section 3.2.1]{HTT}. Consider also the trivial Cartesian fibrations over $\mathrm{B}G(\tt)$ whose fibers are $\Sh^c_{\La,\tt}(\bR^2)_0$ and $\Perf(k)$, respectively. Then taking the homotopy colimit over the base gives a Cartesian fibration over $\mathrm{B}G(\tt)$ whose fiber is $\Sh^c_{\La,\tt}(\bR^2)_0$ and it is equivalent to the group action in \cref{lem:microlocalmonodromy_via_stalks} by unstraightening, cf.~ \cite[Section 3.1.3]{HTT}.\hfill$\Box$
\end{remark}

\begin{example}\label{ex:monodromy_under_action} To illustrate the action of \cref{lem:microlocalmonodromy_via_stalks} in a simple instance, consider the following homotopy colimit diagram
\begin{equation}\label{eq:action_unknot}
\begin{tikzcd}
\displaystyle\prod_{i=1}^{|\tt|}\Perf(k) \ar[r,"m_{\La,\tt}^\ell"] \ar[d,"\Delta^\ell" left] &  \Perf(k) \ar[d] \\
\Perf(k) \ar[r]& \displaystyle\Loc^c\left(\bigvee_{i=1}^{|\tt|-1}S^1\right).
\end{tikzcd}
\end{equation}
Here we have replaced $\Sh_\Lambda^c(\bR^2)_0$ by $\Perf(k)$ and $\Sh_{\Lambda,\tt}^c(\bR^2)_0$ by $\Loc^c(\bigvee_{i=1}^{|\tt|-1}S^1)$ in the homotopy colimit (\ref{def:htpy_pushout_dgcats}), which is for instance what happens in the case of $\La$ being the max-tb Legendrian unknot. An object in the colimit category of \cref{eq:action_unknot} is given by a pair $(U, (\phi_i))$, where $U\in\Perf(k)$ and $\phi_i: U \lr V$ are isomorphisms with a given complex $V$. Such data determines the monodromy around the $i$th circle $S^1$ in the wedge, which is given by $\phi_{i+1} \circ \phi_{i}^{-1}: V \lr V$. Therefore, the action $(V, (\phi_i)) \lr (V, (c_i \circ \phi_i))$ in \cref{lem:microlocalmonodromy_via_stalks} modifies the monodromy by conjugation, i.e.~via
\begin{equation*}(\phi_{i+1} \circ \phi_{i}^{-1})\longmapsto c_{i+1} \circ (\phi_{i+1} \circ \phi_{i}^{-1}) \circ c_i^{-1}.\eqno\qed\end{equation*}
%\hfill$\Box$
\end{example}

\noindent We denote the action of $(c_i)_{i\in[1,|\tt|]} \in G(\tt)$ on $\Sh^c_{\La,\tt}(\bR^2)_0$ via \cref{lem:microlocalmonodromy_via_stalks} by 
$$P_{(c_i)}: \Sh^c_{\La,\tt}(\bR^2)_0 \lr \Sh^c_{\La,\tt}(\bR^2)_0,$$
and the induced action on derived the moduli stack $\M(\Lambda, \tt)$ by
$$\P_{(c_i)}: \M(\Lambda, \tt) \lr \M(\Lambda, \tt).$$
Note that the $G(\tt)$-action on $\Sh^c_{\La,\tt}(\bR^2)_0$ preserves microstalks. Hence, when we restrict to the connected component $\M_n(\Lambda, \tt) \subset \M(\La, \tt)$ consisting of microlocal rank $n$ sheaves, the $G(\tt)$-action above yields an algebraic action of the algebraic group $\prod_{i=1}^{|\tt|}\GL_n$ on the derived stack $\M_n(\Lambda, \tt)$.

\begin{example}\label{ex:Fp_action_through_monodromy}
Let $(\La,\tt)$ have a unique basepoint in its $k$-th component. Given an automorphism $c_k \in \Aut(\Perf(k))$, the induced automorphism $\P_{c_k}\in\Aut(\M({\La,\tt}))$ acts on a point $(\SF,(\phi_i))$, with underlying sheaf $\SF\in \M({\La})$, by left multiplication $c_k \circ \phi_k$ on the trivialization $\phi_k$ of the microstalk at that unique basepoint in the $k$-th component. This action factors through the natural action on the local system $m_\La(\SF)\in\Loc(\La)$ and, as in \cref{ex:monodromy_under_action}, it modifies the microlocal monodromy $mon(\mathscr{F})_k$ at the basepoint by conjugation, i.e.~ $mon(\mathscr{F})_k\longmapsto c_k \circ mon(\mathscr{F})_k \circ c_k^{-1}$.\hfill$\Box$
\end{example}

\noindent The most important property of the action in \cref{lem:microlocalmonodromy_via_stalks}, and the reason we present it explicitly, is that it presents the stack $\M_n({\La})$ as a quotient of $\M_n({\La,\tt})$. Indeed, by the homotopy colimit diagram \eqref{def:htpy_pushout_dgcats}, we have a homotopy pull-back diagram
\[\begin{tikzcd}
    \M_n({\La,\tt}) \ar[r] \ar[d] & \M_n({\La}) \ar[d] \\
    \BGL_n \ar[r] & \displaystyle\prod_{i=1}^{|\tt|}\BGL_n.
\end{tikzcd}\]
By construction, the bottom functor between classifying spaces is induced by the diagonal inclusion $\GL_n \lr \prod_{i=1}^{|\tt|}\GL_n$, and thus can be presented as a homotopy quotient by $\prod_{i=1}^{|\tt|-1}\GL_n$. Therefore, the stack $\M_n({\La})$ can be constructed from $\M_n({\La,\tt})$ by quotienting with $\prod_{i=1}^{|\tt|-1}\GL_n$ via the action from \cref{lem:microlocalmonodromy_via_stalks}.

\begin{remark}[A comparison with the basepoint action in the Floer-theoretic perspective]
The action in Lemma \ref{lem:microlocalmonodromy_via_stalks} defines a natural transformation of the identity functor $\mbox{id}:\Sh_{\Lambda,\tt}^c (\bR^2)_0 \lr \Sh_{\Lambda,\tt}^c(\bR^2)_0$ which acts on the trivializations by left multiplication by the monodromy around that component. Such action can be understood by taking based loops on the morphism $P$ of \cref{lem:microlocalmonodromy_via_stalks}. In particular, for microlocal rank 1, we have morphisms
$$\Omega \GL_1 \lr \Omega_{\mathrm{id}} \Aut(\Sh_{\Lambda,\tt}^c(\bR^2)_0) \xrightarrow{\sim} \Aut(\mathrm{id}_{\Sh_{\Lambda,\tt}^c(\bR^2)_0}).$$
Since $\Omega \GL_1 \cong \bZ$, we have a natural automorphism $t_\SF$ for any $\SF \in \Sh_{\Lambda,\tt}^c(\bR^2)_0$. This automorphism can be described geometrically as follows.

\noindent Consider a family of base points $\{\tt_s\}_{s\in[0,1]}$, $\tt_s\sse\La$, varying smoothly on the parameter $s\in[0,1]$. Let us say that such family goes around the $k$-th clock if $\tt_0=\tt_1$ and
\begin{enumerate}
    \item each point $t_{k}\in \tt_0$ in the $k$-th component $\La_k\sse\La$ moves around once from $s=0$ to $s=1$ and does not intersect $\tt_0$ for $s \in (0, 1)$,
    \item each point in $\tt_s$ in the $j$-th component stays fixed for all $s\in[0,1]$ if $j\neq k$.
\end{enumerate}
Then the above automorphism $t_\SF$ is componentwise given by moving the basepoint around the $k$th clock, which changes the trivialization of the microstalk by left multiplication of the monodromy around that $k$th component. This coincides with the Floer-theoretic description in \cite[Section 5]{CasalsNg}, which presents such actions from the perspective of the Legendrian contact dg-algebra of $(\La,\tt)$.\hfill$\Box$
\end{remark}
\color{black}

%%%%%%%%%%%%%%%%%%%%%%%%%%%%%%%%%%%%%%%%%%
%%%%%%%%%%%%%%%%%%%%%%%%%%%%%%%%%%%%%%%%%%
%%%%%%%%%%%%%%%%%%%%%%%%%%%%%%%%%%%%%%%%%%

\subsubsection{Multiple basepoints per component} For the case of $\tt$ with multiple basepoints in one component of $\La$, we can reduce to the case of one basepoint per component as follows. Consider a subset $\tt' \subset \tt \subset \Lambda$ that still has at least one point per component. The homotopy colimit diagram (\ref{def:htpy_pushout_dgcats}) can be broken down into the following diagram
\begin{center}
    \begin{tikzcd}
\displaystyle\prod_{i=1}^{|\tt|}\Perf(k) \ar[r] \ar[d] & \displaystyle\prod_{i=1}^{|\tt|-|\tt'|}\Perf(k) \times \Perf(k) \ar[r] \ar[d] & \Perf(k) \ar[d] \\
    \Sh_{\Lambda}^c(\bR^2)_0 \ar[r] & \Sh_{\Lambda,\tt'}^c(\bR^2)_0 \ar[r] & \Sh_{\Lambda,\tt}^c(\bR^2)_0.
    \end{tikzcd}
\end{center}
Then the maps of derived stacks from (\ref{def:htpy_pullback_stacks}) can be broken into the diagram
\begin{center}
    \begin{tikzcd}
    \M_n({\La,\tt}) \ar[r] \ar[d] & \M_n({\La,\tt'}) \ar[r] \ar[d] & \M_n({\La}) \ar[d] \\
    \BGL_n \ar[r] & \displaystyle\prod_{i=1}^{|\tt|-|\tt'|}\BGL_n \times \BGL_n \ar[r] & \displaystyle\prod_{i=1}^{|\tt|}\BGL_n.
    \end{tikzcd}
\end{center}
\noindent In this particular situation, $\M_n({\La,\tt})\lr\M_n({\La,\tt'})$ is a $\displaystyle\GL_n^{|\tt|-|\tt'|}$-principal bundle with a section given by identifying the microstalks at the extra basepoints in $\tt \setminus \tt'$ with their adjacent basepoints in $\tt'$ (in their same component). Thus we have an isomorphism
\begin{equation}\label{eq:multipl_basepoints_component}\M_n({\La,\tt})\cong \M_n({\La,\tt'})\times \GL_n^{|\tt|-|\tt'|}.
\end{equation}
Due to this isomorphism, we often directly study the case where $\tt$ has exactly one basepoint per component of $\La$, as adding basepoints to a given component does not particularly enrich the symplectic geometry in this framework, only adding $\GL_n$-factors.

\begin{comment}
\begin{remark} The case of microlocal rank 1 %, in which this action factors through the Abelianization of $F_p$, 
is of particular interest to study commutative cluster algebras, and we will momentarily focus on it. Indeed, by Example \ref{ex:Fp_action_through_monodromy} (and Lemma \ref{lem:microlocalmonodromy_via_stalks}), if we restrict to microlocal rank 1 we obtain an action factoring through the monodromies of (microlocal) Abelian local systems \textcolor{blue}{WL: Now I feel I do not understand this sentence anymore (starting from `Indeed'). Maybe delete the sentence?}. In terms of algebraic geometry, this translates into some of the $\GL_1$-actions studied in the literature, see e.g.~\cite[Section 2.2]{CGGS}. In the literature, these are often referred to as torus actions because $k=\C$ is chosen and $\GL_1(\C)\cong\C^*$ is an algebraic torus.\hfill$\Box$
\end{remark}
\end{comment}

%%%%%%%%%%%%%%%%%%%%%%%%%%%%%%%%%%%%%%%%%%
%%%%%%%%%%%%%%%%%%%%%%%%%%%%%%%%%%%%%%%%%%
%%%%%%%%%%%%%%%%%%%%%%%%%%%%%%%%%%%%%%%%%%

\subsection{Microlocal rank 1 case of the basepoint action}\label{ssec:t_action_rank1} Let us focus on $\M_1({\La,\tt})\sse \M({\La,\tt})$, the component of $\M({\La,\tt})$ consisting of sheaves singularly supported in $(\La,\tt)$ with microlocal rank $1$. This is the component that has been most studied in the literature and it is of particular interest due to its connection to cluster algebras, see e.g.~\cite{CasalsHonghao,CW,CasalsZaslow,CGGS,CGGS2,CGGLSS}. The component $\M_1({\La,\tt})$ parametrizes sheaves with singular support on $\La$ such that their microstalks at the basepoints of $\tt$ are 1-dimensional $k$-modules, all commonly identified. Similarly, we denote by $\M_1({\La})$ the component of $\M({\La})$ consisting of microlocal rank $1$ sheaves singularly supported in $\La$.

\noindent Consider the components $\BGL_1(k)$ and $\BGL_1(k)\times\stackrel{(|\tt|)}{\ldots}\times \BGL_1(k)$ of %{\WL $\M(\modk)$ is larger than $\bigsqcup_{n \in \bN}\BGL_n(k)$. It also contains perfect chain complexes not concentrated in a single degree.}
$$\M(\Perf(k))\quad\mbox{and}\quad \displaystyle\M\left(\prod_{i=1}^{|\tt|}\Perf(k)\right).$$
To ease notation, let us denote $H:=\GL_1(k)$ and $G:=H\times\stackrel{(|\tt|)}{\ldots}\times H$. Here each $\GL_1$ stabilizer in $\BGL_1\times\stackrel{(|\tt|)}{\ldots}\times \BGL_1$ is geometrically understood as the automorphisms of the microstalk at the corresponding basepoint of $\tt$. Note also that in this Abelian $\GL_1$-case, the action restricted to $H$ is trivial. The homotopy pullback diagram (\ref{def:htpy_pullback_stacks}) restricted to microlocal rank 1 reads
\begin{center}
\begin{equation}\label{def:htpy_pullback_stacks_rank1}
\begin{tikzcd} 
           \M_1(\La,\tt) \ar[r] \ar[d]&  \M_1(\La)\ar[d,"\M(m_{\La,\tt}^\ell)"] \\
    BH \ar[r,"B\Delta"]  &  BG,
    \end{tikzcd}
\end{equation}
\end{center}
where we have used that $\M(\Delta^\ell):BH\lr BG$ is equivalent to the map $B\Delta$ of classifying stacks induced by the diagonal group morphism $\Delta:H\lr G$. The homotopy fiber of $B\Delta$ is $G/H$, and in fact $EG\lr EG/H\cong EG\times_G (G/H) $ is a universal principal $H$-bundle, so $EG/H$ is
a classifying space for $H$. Therefore, $B\Delta$ can be modeled by the natural projection $EG\times_G (G/H)\lr BG$. Hence, by construction, $\M_1({\La,\tt}) \lr \M_1({\La})$ is a $(G/H)$-principal bundle.

%\begin{remark}
%In the general non-Abelian $\GL_n$ case, the diagonal action of $H$ is non-trivial and $H$ is not a normal subgroup. In appropriate coordinates, this diagonal subgroup acts via conjugation. In the $n=1$ case this simplifies, as the action becomes trivial, and we can reduce to a $(G/H)$-principal bundle and merely have a factor of $BH$ in the moduli stack.\hfill$\Box$
%\end{remark}

%%%%%%%%%%%%%%%%%%%%%%%%%%%%%%%%%%%%%%%%%%
%%%%%%%%%%%%%%%%%%%%%%%%%%%%%%%%%%%%%%%%%%
%%%%%%%%%%%%%%%%%%%%%%%%%%%%%%%%%%%%%%%%%%

\subsection{A few examples}\label{ssec:examples1} Here are descriptions of the derived stacks $\M_1({\La})$ and $\M_1({\La,\tt})$ and the basepoint action in some cases of interest. In all the following examples, $\Lambda$ is a Legendrian $(-1)$-closure of a positive braid, and by a direct generalization of \cite[Prop.~6.1]{ChantraineNgSivek}, for any pseudo-perfect object in the category of sheaves valued in a simplicial commutative ring, the cotangent complexes of $\M_1({\La})$ and $\M_1({\La,\tt})$ are concentrated in non-negative degrees. Hence by \cite[Theorem 2.2.2.6, Corollary 2.2.4.6]{HAGII} and \cite[Proposition 1.38]{Pridham} the derived stacks $\M_1({\La})$ and $\M_1({\La,\tt})$ are isomorphic to (underived) Artin stacks which parametrize microlocal rank 1 sheaves.
%%%%%%%%%%%%%%%%%%%%%%%%%%%%%%%%%%%%%%%%%%
%%%%%%%%%%%%%%%%%%%%%%%%%%%%%%%%%%%%%%%%%%
%%%%%%%%%%%%%%%%%%%%%%%%%%%%%%%%%%%%%%%%%%

\subsubsection{Legendrian knots}\label{ex:Leg_knots} Let $\La\sse(\R^3,\xi_\st)\sse (T^*_\infty\R^2,\xi_\st)$ be a Legendrian knot with $\tt=\{t\}$ consisting of a single basepoint. Then $\Delta:H\lr G$ is the identity morphism and the homotopy pullback in Equation (\ref{def:htpy_pullback_stacks_rank1}) is the base change along the identity morphism. Therefore $\M_1({\La,\tt})\cong\M_1({\La})$ are isomorphic. Two remarks: %{\WL Shall we write the statement as a proposition?}{\RC I wish I knew how to write this in a simple self-contained manner. I do not, that's why I just cited the references. It is not ideal, but if I wrote it as a proposition with proof there would be a lot of extra notation and results imported from outside.}{\WL I see. Let's keep it like this then.}

\begin{itemize}
    \item[(i)] Let $\La\sse(\R^3,\xi_\st)$ be a Legendrian knots of the form $\La=\La_\beta$, $\beta$ a positive braid and $\La_\beta$ its $(-1)$-closure, see \cite[Section 2.2]{CasalsNg}. Then the derived stacks $\M_1({\La_\beta})$ are isomorphic to quotient stacks $[X_\beta/\GL_1]$ where $X_\beta$ are smooth affine varieties and $\GL_1$ acts trivially on $X_\beta$, cf.~\cite[Section 4.1]{CW}. In short, in this case we have
    $$\M_1({\La_\beta})\cong [X_\beta/\GL_1] \cong X_\beta \times [*/\GL_1] \cong X_\beta\times \BGL_1$$
    where $X_\beta$ is a smooth affine scheme. As stated in (ii) below, these $X_\beta$ are closely related to the braid varieties $X(\beta)$ studied in \cite{CGGS,CGGS2,CGGLSS}. Also, equations expressing $X_\beta$ as a complete intersection can be readily read from a braid word for $\beta$, cf.~\cite[Section 3]{CGGLSS} (or \cite[Section 5]{CasalsNg} from the perspective of the Legendrian contact dg algebra).\\

    \item[(ii)] For a Legendrian knot with basepoints $\tt=\{t_1,\ldots,t_{n}\}$, the action of $G/H$ on $\M_1({\La,\tt})$ is free and $\M_1({\La,\tt})\cong\M_1({\La})\times\GL_1^{n-1}$, as per Equation (\ref{eq:multipl_basepoints_component}). Let $\La\sse(\R^3,\xi_\st)$ be a Legendrian knot of the form $\La=\La_{\beta\delta(\beta)}$ , with one basepoint per strand of $\beta$. By \cite[Section 4.1]{CW} or a direct generalization of \cite[Proposition 5.8]{CasalsLi}, we have
    $$\M_1({\Lambda_{\beta\delta(\beta)},\tt}) \cong [X(\beta)/\GL_1] \cong X(\beta) \times [*/\GL_1] \cong X(\beta) \times \BGL_1,$$
    where $X(\beta)$ is the braid variety associated to $\beta$, see e.g.~\cite[Section 2]{CGGS} or \cite[Section 3]{CGGLSS}.
    Specifically, per Equation (\ref{eq:multipl_basepoints_component}), we have a regular isomorphism $X(\beta)\cong X_{\beta\delta(\beta)}\times \GL_1^{n-1}$, where $\delta(\beta)$ is a braid lift of the Demazure product of $\beta$ and $n$ is the number of strands.
\end{itemize}

%%%%%%%%%%%%%%%%%%%%%%%%%%%%%%%%%%%%%%%%%%
%%%%%%%%%%%%%%%%%%%%%%%%%%%%%%%%%%%%%%%%%%
%%%%%%%%%%%%%%%%%%%%%%%%%%%%%%%%%%%%%%%%%%

\subsubsection{Hopf link, Part I: $\M_1({\La,\tt})$ and $\M_1({\La})$}\label{ex:Hopf_link} Let $\La=\La_1\cup \La_2\sse(\R^3,\xi_\st)\sse (T^*_\infty\R^2,\xi_\st)$ be the max-tb Hopf link, with both components $\La_1,\La_2$ max-tb unknots. Consider the set of basepoints $\tt=\{t_1,t_2\}$, with $t_i\in \La_i$, one per component. The homotopy pullback in Equation (\ref{def:htpy_pullback_stacks_rank1}) and the basepoint action are explicitly described as follows, cf.~%~\cite[Section 4.3]{ShendeTreumannWilliams} 
\cite[Section 4.1]{CW}. Consider the polynomial ring $\Z[x,y]$ and its localization $\Z[x,y]_{(1+xy)}$ at the principal ideal $(1+xy)$. Define the affine scheme $X:=\Spec \Z[x,y]_{(1+xy)}$. One can show that the basepoint action from Lemma \ref{lem:microlocalmonodromy_via_stalks} can be described as follows: consider the algebraic action of $G:=\GL_1\times \GL_1$ on $X$ defined by %{\WL Is there a reference where this computation is done?}{\RC Not that I know of: I have done it. It kind of follows from the way basepoints act on the dga (floer-theoretically) or by direct computation on the sheaf side.}{\WL I see. I can also imagine how to prove it using \cite[Section 4.3]{ShendeTreumannWilliams}, which is why I mention that. Let's just leave it like this for now.}
$$G\times X\lr X,\quad (t_1,t_2;x,y)\longmapsto (t_1xt_2^{-1},t_2yt_1^{-1}).$$
Here we interpret the scheme $X$ as a 0-geometric Artin 0-stack, cf.~\cite[Prop. 2.1.2.1]{HAGII}. Note also that under $\Delta:H\lr G$ the diagonal morphism, the diagonal subgroup $\Delta(H)=\{(t_1,t_2):t_1=t_2\}\sse G$ acts trivially on $X$. Then the moduli stacks of sheaves appearing in Equation (\ref{def:htpy_pullback_stacks_rank1}) are
$$\M_1({\La,\tt})\cong[X/H]\cong X\times BH\cong X\times \BGL_1\quad\mbox{and}\quad \M_1({\La})\cong[X/G].$$
The two input morphisms $\M_1(m^\ell_{\La,\tt}):[X/G]\lr BG$ and $B\Delta:BH\lr BG$ in the homotopy pullback Equation (\ref{def:htpy_pullback_stacks_rank1}) are given by the canonical map $[X/G]\lr BG$, induced by $X\lr \{\mbox{pt}\}$, and the diagonal map $B\Delta:\BGL_1\lr\BGL_1\times\BGL_1$. The induced morphism $\M_1({\La,\tt})\lr BH$ is also given by the canonical map $[X/H]\lr BH$. The last induced $\M_1({\La,\tt})\lr \M_1({\La})$ in Equation (\ref{def:htpy_pullback_stacks_rank1}) is an algebraic quotient by $G/H$. This latter quotient can be realized by the action of $K:=\GL_1$ on $[X/H]$ induced by the action
\begin{equation}\label{eq:K_action_HopfLink}
K\times X\lr X,\quad (\tau;x,y)\longmapsto (\tau x,\tau^{-1}y),
\end{equation}
where we have taken the slice $\tau=t_1t_2^{-1}$ to represent the quotient $K\cong G/H$. In summary, with the notation and actions above, the homotopy pullback Equation (\ref{def:htpy_pullback_stacks_rank1}) reads
\begin{center}
\begin{equation}\label{def:htpy_pullback_stacks_rank1_example}
\begin{tikzcd} 
           {[X/H]} \ar[r] \ar[d]&  {[X/G]}\ar[d] \\
    BH \ar[r,"B\Delta"]  &  BG,
    \end{tikzcd}
\end{equation}
\end{center}
where both vertical maps are the canonical projections induced by $X\lr\{\mbox{pt}\}$ and both horizontal maps can be modeled after $(G/H)$-principal bundles. In this context, the $(G/H)$-principal bundle $[X/H]\lr [X/G]$ is classified by the composition $[X/G]\lr BG \lr B(G/H)$.

%%%%%%%%%%%%%%%%%%%%%%%%%%%%%%%%%%%%%%%%%%
%%%%%%%%%%%%%%%%%%%%%%%%%%%%%%%%%%%%%%%%%%
%%%%%%%%%%%%%%%%%%%%%%%%%%%%%%%%%%%%%%%%%%

\subsubsection{Hopf link, Part II: Lagrangian fillings}\label{ex:Hopf_link_fillings}
The quotient stack $\M_1({\La})\cong[X/G]$ in Example \ref{ex:Hopf_link} has a natural interpretation in terms of Lagrangian fillings of the Hopf link, as follows. There are three types of orbits in $X=\Spec \Z[x,y]_{(1+xy)}$ for the action Equation (\ref{eq:K_action_HopfLink}):

\begin{enumerate}
    \item If $xy=\alpha$, $\alpha\in k$ and $\alpha\neq0,-1$, then $K$ acts freely and transitively on $O_\alpha:=\{xy=\alpha\}\sse X$. Therefore, the orbit $O_\alpha$ gets quotiented down to a point $[O_\alpha]\in\M_1({\La})$ with no stabilizer coming from this action.\\

    \item If $y=0$ and $x\neq0$, then $K$ acts freely and transitively on $O_x:=\{x\neq0,y=0\}\sse X$. This orbit $O_x$ gets quotiented down to a point $[O_x]\in\M_1({\La})$ with no stabilizer coming from this action. Similarly, if $x=0$ and $y\neq0$, then $K$ acts freely and transitively on $O_y:=\{x=0,y\neq0\}\sse X$ and $O_y$ gets quotiented down to a point $[O_y]\in\M_1({\La})$ with no stabilizer coming from this action.\\

    \item The point $O_0:=\{(0,0)\}\in X$ is a fixed points of the action. It is thus represented as a point $[O_0]\in\M_1({\La})$ with stabilizer $K$ coming from this action.
\end{enumerate}

\noindent As before, in all the three cases above, there is always the additional stabilizer $H$ in all points of $\M_1({\La})$ coming for $BH$. Therefore, the stabilizers in $\M_1({\La})=[X/G]$ of the points $O_\alpha$, $\alpha\neq0$, $O_x$ and $O_y$ are $H$ and the stabilizer of $O_0$ is $G$. Now, the Hopf link admits the following three exact Lagrangian fillings:
\begin{enumerate}
    \item Two embedded exact Lagrangian cylinders $L$ and $L'$, see e.g.~\cite{EHK,Pan-fillings}, which are not Hamiltonian isotopic to each other (relative to their boundary $\La$).\\

    \item An (unobstructed) immersed Lagrangian union $L_0$ of two embedded Lagrangian disks, with a unique transverse intersection point. Each embedded Lagrangian disk fills a max-tb unknotted component of the Hopf link $\La$.
\end{enumerate}

\noindent Interpreting $\M_1({\La})$ as the moduli stack of Lagrangian fillings with $\GL_1$-local systems, we have:
\begin{enumerate}
    \item The Lagrangian filling $L$ is encoded by the (reduced) substack
    $$\left(\bigcup\nolimits_{\alpha\neq0,-1} [O_\alpha]\right)\cup [O_x]\sse \M_1({\La})$$
    where each point in this substack is a choice of $\GL_1$-local system in $L$. Similarly, the Lagrangian filling $L'$ is encoded by the (reduced) substack
    $$\left(\bigcup\nolimits_{\alpha\neq0,-1} [O_\alpha]\right)\cup [O_y]\sse \M_1({\La}),$$
    where again each point in this substack is a choice of $\GL_1$-local system in $L'$. Note that for a choice of commutative ring $k$, both these substacks are isomorphic to $k^*/k^*$, which are indeed $\GL_1(k)$-local systems on a cylinder; see e.g.~\cite[Section 2]{TreumannZaslow}.\\

    \item If one considers local systems on $L_0$, then $L_0$ is represented by the point $[O_0]\in \M_1({\La})$.\footnote{In the appropriate sense, one could consider more general sheaves in $L_0$ so that this filling would be represented by the entirety of $\M_1({\La})$. That would be the case if we considered constructible sheaves on one of the disks stratified by the intersection point with the other disk and its complement. The fact that such sheaves give the entire $\M_1({\La})$ can be justified using Section \ref{sec:corestriction} below, cf.~Theorem \ref{thm:invariance}.} It is relevant to note that $[O_0]$ cannot be separated from either $[O_x]$ or $[O_y]$. We interpret this non-separatedness as an algebraic incarnation of the fact that both fillings $L$ and $L'$, endowed with local systems giving $[O_x]$ and $[O_y]$, can be obtained from $L_0$ by resolving the immersed double point via Polterovich surgery. Therefore, in this sense, $[O_0]$ should lie as close as possible to both $[O_x]$ and $[O_y]$ inside of the moduli $\M_1({\La})$.\hfill$\Box$
\end{enumerate}

%%%%%%%%%%%%%%%%%%%%%%%%%%%%%%%%%%%%%%%%%%
%%%%%%%%%%%%%%%%%%%%%%%%%%%%%%%%%%%%%%%%%%
%%%%%%%%%%%%%%%%%%%%%%%%%%%%%%%%%%%%%%%%%%

\subsubsection{Legendrian links $\La_\beta$} Example \ref{ex:Hopf_link} generalizes to Legendrian links of the form $\La=\La_\beta$, $\beta$ a positive braid in $n$-strands. By using Equation (\ref{eq:multipl_basepoints_component}), we may assume that the set $\tt$ of basepoints has exactly one basepoint per component. Set $G:=\GL_1^{|\tt|}$, $H:=\GL_1$ and identify $H$ with the diagonal subgroup $\Delta(H)\sse G$. By \cite[Section 4.1]{CW}, see also \cite[Corollary 3.7]{CGGLSS}, there exists a smooth affine variety $X_\beta$ such that
$$\M_1({\La_\beta,\tt}) \cong [X_\beta/H]\cong X_\beta\times BH.$$
where $H$ acts trivially on $X_\beta$. That said, as illustrated in Example \ref{ex:Hopf_link}, the $(G/H)$-action on $X_\beta$ is not free and the quotient $\M_1({\La_\beta})$ typically has non-trivial stabilizers (and different from $H$) at some points. As in Example \ref{ex:Leg_knots}, $X_\beta$ is closely related to a braid variety: specifically, per Equation (\ref{eq:multipl_basepoints_component}), the braid variety $X(\beta)$ is isomorphic to $X_{\beta\delta(\beta)}\times \GL_1^{n-|\tt|}$, cf.~\cite[Section 4.1]{CW} or \cite[Proposition 5.8]{CasalsLi}. An explicit description of the basepoint action follows from \cite[Section 2.2]{CGGS}, cf.~also \cite[Section 5]{CasalsNg} from the perspective of Legendrian contact dg algebra.\hfill$\Box$

%%%%%%%%%%%%%%%%%%%%%%%%%%%%%%%%%%%%%%%%%%
%%%%%%%%%%%%%%%%%%%%%%%%%%%%%%%%%%%%%%%%%%
%%%%%%%%%%%%%%%%%%%%%%%%%%%%%%%%%%%%%%%%%%

%% file: 3_KSStack.tex
\section{Kashiwara-Schapira stacks and corestriction functors}\label{sec:corestriction} This section starts with Subsection \ref{ssec:KS_stack}, which reviews the Kashiwara-Schapira stacks. It includes the necessary details and properties about these stacks used to prove Theorem \ref{thm:main}. This first subsection also contains examples and some discussions relating our definitions to previous appearances of this stack in the literature. Subsections \ref{ssec:corestriction} and \ref{ssec:KS_stack_corestriction_relative} study corestriction functors for Kashiwara-Schapira stacks, a key ingredient in our proof of Theorem \ref{thm:main}. Subsection \ref{ssec:Lagrangian_fillings_KS_stack} relates Lagrangian fillings and Kashiwara-Schapira stacks: this allows us to relate microlocal merodromies with actual parallel transports of local systems, which is used in Subsection \ref{ssec:HH0_near_rel_cycle} as part of the argument for Theorem \ref{thm:main}.

%%%%%%%%%%%%%%%%%%%%%%%%%%%%%%%%%%%%%%%%%%
%%%%%%%%%%%%%%%%%%%%%%%%%%%%%%%%%%%%%%%%%%
%%%%%%%%%%%%%%%%%%%%%%%%%%%%%%%%%%%%%%%%%%
\subsection{The Kashiwara-Schapira stack}\label{ssec:KS_stack}
The category $\Sh_\La(\R^2)$ is the global sections of a sheaf of dg-categories, known as the Kashiwara-Schapira stack. This subsection reviews such stack and its properties. Let $X$ be a Weinstein sector with subanalytic Lagrangian skeleton $\bL$. In this work, the expression {\it sheaves of dg-categories} precisely means $\infty$-sheaves, cf.~\cite[Def.~7.3.3.1]{HTT}, valued in the $\infty$-category $\mbox{dg-Cat}_k$ of well generated dg-categories, cf.~\cite{Tabuada_Wellgenerated} (we will see later in the proof of Theorem \ref{prop:corestriction} that these sheaves on Lagrangian skeleta are in fact valued in the $\infty$-category of compactly generated dg-categories).

Following \cite[Section 6]{KashiwaraSchapira_Book}, we will now consider a certain $\infty$-sheaf of dg-categories
$$\msh_{\bL}: \op(\bL)^{op} \lr \mbox{dg-Cat}_k,$$
often referred to as the Kashiwara--Schapira stack, cf.~ \cite[Section 10.1]{Guillermou23_SheafSummary}. (It is also known as microlocal sheaves, cf.~ \cite[Section 3.4]{Nadler16} or \cite[Section 6]{NadlerShende}.) We work with dg-categories primarily so that we can use the derived moduli stack from \cite{ToenVaquie07}. In addition, we use dg-categories as sheaf coefficients because triangulated (derived) categories do not typically satisfy descent. The $\infty$-sheaf $\msh_{\bL}$ is defined as follows.

First, we define an $\infty$-sheaf of dg-categories in the cotangent bundle $T^*M$ of a smooth manifold $M$. By definition, a subset $\widehat\Lambda \subset T^*M$ is conic if it is invariant under the following scaling $\bR_+$-action: $\bR_+\times T^*M \to T^*M, (s;x, \xi) \to (x, s\xi)$. A typical example of $\widehat\La$ will be the union of the Lagrangian cone over a Legendrian submanifold $\La\sse T^{*}_{\infty}M$ and the zero section $M\sse T^*M$. In particular, $\widehat\La$ might be singular.

\begin{remark}
We consider the conic topology on $T^*M$, generated by all conic open sets. We define the sheaf of categories using the conic topology, which is equivalent to an $\bR_+$-equivariant sheaf of categories on the standard (Euclidean) manifold topology. Indeed, for any open subset in the manifold topology whose intersection with all $\bR_+$-orbits are contractible, we declare sections on it to be the sections on the conification. Since such open subsets form a topological basis, we then obtain an $\bR_+$-equivariant sheaf on the manifold topology.\hfill$\Box$
\end{remark}

We denote by $\op(T^*M)$ the category of open conic subsets on $T^*M$, with morphisms given by inclusions of the open conic subsets.

\begin{definition}\label{def:KS_stack}
Let $\widehat\Lambda \subset T^{*}M$ be a conic subset. The $\infty$-presheaf $\msh_{\widehat\Lambda}^\text{pre}:\op(T^*M)^{op} \lr \mbox{dg-Cat}_k$ of dg-categories on $T^{*}M$ associated to $\widehat\Lambda$ is
$$\msh_{\widehat\Lambda}^\text{pre}(\widehat\Omega) := \Sh_{\widehat\Lambda \cup (T^{*}M \backslash \widehat\Omega)}(M) /\Sh_{T^{*}M \backslash \widehat\Omega}(M),\quad \widehat\Omega \subset T^{*}M\mbox{ open conic subset}.$$
By definition, the $\infty$-sheaf $\msh_{\widehat\Lambda}:\op(T^*M)^{op}\lr\mbox{dg-Cat}_k$ of dg-categories on $T^{*}M$ is the $\infty$-sheafification of the $\infty$-presheaf $\msh_{\widehat\Lambda}^\text{pre}$. We refer to the $\infty$-sheaf $\msh_{\widehat\Lambda}$ as the Kashiwara-Schapira stack of $\widehat\La$.
\hfill$\Box$
\end{definition}

Note that for any (conic) open set $\widehat\Omega\sse T^*M$ with $\widehat\Omega \cap \widehat\Lambda = \emptyset$ we have a dg-equivalence $\msh_{\widehat\Lambda}^\text{pre}(\widehat\Omega) \cong 0$, where the right hand side is understood as the one-object category with the zero vector space as its endomorphisms. Therefore $\msh_{\widehat\Lambda}^\text{pre}$ is supported on $\widehat\Lambda$. Hence, $\msh_{\widehat\Lambda}$ is also supported on $\widehat\Lambda$.  We often abuse notation and still denote the restriction of this sheaf to $\widehat\Lambda$ by $\msh_{\widehat\Lambda}:\op(\widehat\Lambda)^{op}\lr\mbox{dg-Cat}_k$. 

The conical aspect can be effectively ignored (the only additional choices lie in the zero section) by restricting to the ideal contact boundary of $T^*M$ (or a cosphere bundle):

\begin{definition}\label{def:microlocalization_functor}
\noindent Let $\Lambda \subset T^{*}_{\infty}M$ be a subset. By definition, the $\infty$-presheaf of categories $\msh_{\Lambda}^\text{pre}:\op(\La)^{op}\lr\mbox{dg-Cat}_k$ is the restriction of $\msh_{\Lambda \times \bR_+}^\text{pre}$ to $\La$. Similarly, the $\infty$-sheaf $\msh_{\Lambda}:\op(\La)^{op}\lr\mbox{dg-Cat}_k$ is the restriction of $\msh_{\Lambda \times \bR_+}$ to $\La$. Equivalently, the $\infty$-presheaf $\msh_{\Lambda}^\text{pre}$ is defined by
$$\msh_{\Lambda}^\text{pre}(\Omega) := \Sh_{\Lambda \cup (T^{*}_{\infty}M \backslash \Omega)}(M) /\Sh_{T^{*}_{\infty}M \backslash \Omega}(M),\quad \Omega \subset T^{*}_{\infty}M\mbox{ open
subset}.$$
By definition, the microlocalization functor
$$m_\Lambda: \Sh_\Lambda(M) \lr \msh_\Lambda(\Lambda)$$
is the functor induced by the natural quotient functor on the $\infty$-presheaf of categories:
\begin{equation*}
\hspace{42pt}\Sh_\Lambda(M) \hookrightarrow \Sh_{\Lambda \cup (T^{*}_{\infty}M \backslash \Omega)}(M) \to \Sh_{\Lambda \cup (T^{*}_{\infty}M \backslash \Omega)}(M) /\Sh_{T^{*}_{\infty}M \backslash \Omega}(M)=\msh_\Lambda^\text{pre}(\Omega). \hspace{37pt}\Box
\end{equation*}
%\hfill$\Box$
\end{definition}

\begin{remark}
If we considered them as $\infty$-sheaves on $T^{*}_{\infty}M$, instead of $\La$, the $\infty$-presheaf $\msh_\Lambda^\text{pre}$ would be supported on $\Lambda$ and the $\infty$-sheaf on $\Lambda$ would coincide with the $\infty$-sheafification of the restriction of that $\infty$-presheaf $\msh_\Lambda^\text{pre}$ to $\Lambda$.\hfill$\Box$
\end{remark}

\subsubsection{Comments and examples related to the Kashiwara-Schapira stack} The following connects the content introduced above to key parts of the existing literature on Kashiwara-Schapira stacks.

\begin{remark}\label{rem:microsheaf-different-definition} The literature contains alternative descriptions for such sheaves $\msh_\La$, see e.g.~\cite[Section 10]{Guillermou23_SheafSummary} or \cite[Section 3.4]{Nadler16}. For comparison to the former (defined for smooth Legendrians), consider $\Lambda \subset T^{*}_{\infty}M$ and an open subset $\Lambda_0 \subset \Lambda$. Using the notation $\Sh_{(\Lambda_0)}(M) = \bigcup_{\Lambda_0 \subset \Omega}\Sh_{\Lambda_0 \cup T^{*}_{\infty}M \backslash \Omega}(M)$ from \cite[Section 10]{Guillermou23_SheafSummary}, we can show by \cite[Lemma 10.2.2]{Guillermou23_SheafSummary} that
\begin{align*}
\msh_\Lambda^\text{pre}(\Lambda_0) = \operatorname{colim}_{\Lambda_0\subset\Omega}\Sh_{\Lambda_0 \cup T^{*}_{\infty}M \backslash \Omega}(M)/\Sh_{T^{*}_{\infty}M \backslash \Omega}(M) \cong \Sh_{(\Lambda_0)}(M)/\Sh_{T^{*}_{\infty}M \backslash \Lambda_0}(M),
\end{align*}
which shows that our definition agrees with \cite[Def.~10.1.1]{Guillermou23_SheafSummary}. It can be verified that the morphism in $\msh_\Lambda$ can be computed by the sheaf $\mu hom$ in $T^*M$: this follows from \cite[Prop.~6.1.2]{KashiwaraSchapira_Book} or \cite[Cor.~10.1.5]{Guillermou23_SheafSummary}; this proves that the stalks of the two sheaves coincide.

\noindent The description in \cite[Section 3.4]{Nadler16} is as follows. Consider pairs $(B,\Omega)$ of a (small) open ball $B \subset M$ and a (small) open subset $\Omega \subset T^{*}_{\infty}B$ such that $\Omega$ is a neighbourhood of a component of $\Lambda \cap T^{*}_{\infty}B$. Then \cite[Section 3.4]{Nadler16} defines
$$\msh_\Lambda^\text{pre}(\Omega) = \Sh_{\Lambda \cup (T^{*}_{\infty}B \backslash \Omega)}(B) / \Sh_{T^{*}_{\infty}B \backslash \Omega}(B).$$
Since the restriction $r^*: \Sh(M) \to \Sh(B)$ admits left adjoint $r_!$ and right adjoint $r_*$, and the inclusion $\Sh_{\Lambda \cup (T^{*}_{\infty}M \backslash \Omega)}(M) \hookrightarrow \Sh(M)$ admits left and a right adjoints by \cite[Thm.~1.2]{Kuo}, we know that the restriction on the full subcategories $r^*: \Sh_{\Lambda \cup (T^{*}_{\infty}M \backslash \Omega)}(M) \to \Sh_{\Lambda \cup (T^{*}_{\infty}B \backslash \Omega)}(B)$ also admits left and right adjoints. Using the formula of the left adjoint functor $r^{*\ell}$ in \cite[Thm.~1.2]{Kuo}, one can show that, given $\SF \in \Sh_{\Lambda \cup (T^{*}_{\infty}M \backslash \Omega)}(M)$, the map $\SF \to r^{*\ell}r^*\SF$ is an isomorphism in $\Omega$ (i.e.~the mapping cone has singular support in $T^{*}_{\infty}M \setminus \Omega$), and, similarly, given $\SF_B \in \Sh_{\Lambda \cup (T^{*}_{\infty}B \backslash \Omega)}(B)$, the map $r^*r^{*\ell}\SF_B \to \SF_B$ is also an isomorphism in $\Omega$. Thus we obtain the equivalence between the definition in \cite[Section 3.4]{Nadler16} and Definition \ref{def:microlocalization_functor} above.\hfill$\Box$
\end{remark}

If the projection $\pi: \Lambda \to M$ is finite-to-one, we can connect our definition with \cite[Section 3.4]{Nadler16} and \cite[Section 3.9--11]{JinTreumann}:

\begin{lemma}\label{rem:microsheaf-different-definition2}
For $\Lambda \subset T^*_\infty M$, if the projection $\pi: \Lambda \to M$ is finite-to-one, then for any sufficiently small open ball $B \subset M$ and $\Omega \subset T^*_\infty B$,
$$\msh_\Lambda^\text{pre}(\Omega) \cong \Sh_\Lambda(B) / \Loc(B).$$
\end{lemma}
\begin{proof}
Per Remark \ref{rem:microsheaf-different-definition}, it suffices to show that $\Sh_{\Lambda \cup (T^*_\infty B \setminus \Omega)}(B)/\Sh_{T^*_\infty B \setminus \Omega}(B) \cong \Sh_\Lambda(B) / \Loc(B)$. This is an application of the refined microlocal cut-off lemma, cf.~ \cite[Prop.~6.1.4]{KashiwaraSchapira_Book} and \cite[Lem.~10.2.5]{Guillermou23_SheafSummary}: The inclusion $\iota_*: \Sh_\Lambda(B) \hookrightarrow \Sh_{\Lambda \cup (T^{*}_{\infty}B \backslash \Omega)}(B)$ admits a left adjoint $\iota^*$ called the (refined) microlocal cut-off functor. Given $\SF_B' \in \Sh_\Lambda(B)$, the map $\iota^*\iota_*\SF_B' \to \SF_B'$ is always an isomorphism, and given $\SF_B \in \Sh_{\Lambda \cup (T^{*}_{\infty}B \backslash \Omega)}(B)$, by the refined microlocal cut-off lemma, $\SF_B \to \iota_*\iota^*\SF_B$ is also an isomorphism on $\Omega$. This completes the proof. (See also \cite[Lem.~4.28 \& Cor.~4.32]{KuoLi-spherical}.)
\end{proof}

\begin{lemma}\label{rem:microsheaf-different-definition3}
For $\Lambda \subset T^*_\infty M$, if the projection $\pi: \Lambda \to M$ is finite-to-one, then for any sufficiently small open ball $B \subset M$, $\Omega \subset T^*_\infty B$,
$$\msh_\Lambda^\text{pre}(\Omega) \cong \msh_\Lambda^\text{pre}(\Omega) \cong \Sh_\Lambda(B) / \Loc(B).$$
\end{lemma}
\begin{proof}
We show that for sufficiently small open balls $B' \subset B$, the restriction functor
$$r^*: \Sh_{\Lambda}(B)/\Loc(B) \lr \Sh_{\Lambda}(B')/\Loc(B')$$
induces an equivalence. Then, by Lemma \ref{rem:microsheaf-different-definition2}, it follows that for $\Omega' \subset \Omega$, the restriction functor $\msh_\Lambda^\text{pre}(\Omega) \lr \msh_\Lambda^\text{pre}(\Omega')$ also induces an equivalence. Hence for the $\infty$-sheafification we also have
$$\msh_\Lambda^\text{pre}(\Omega) \cong \msh_\Lambda^\text{pre}(\Omega) \cong \Sh_\Lambda(B) / \Loc(B).$$
For the restriction $r^*: \Sh_{\Lambda}(B)/\Loc(B) \lr \Sh_{\Lambda}(B')/\Loc(B')$, it is fully faithful by the constructibility of the internal Hom in $\Sh_\Lambda(B)$ when $\Lambda \subset T^{*}_{\infty}M$ is subanalytic Legendrian \cite[Prop.~8.4.6]{KashiwaraSchapira_Book}, which implies local constancy of the Hom when we restrict to smaller and smaller neighbourhoods around a point \cite[Lem.~8.4.7]{KashiwaraSchapira_Book}. Since the front projection of $\Lambda$ defines a Whitney stratification, such that the inclusions of the strata in $B'$ to $B$ are homotopy equivalences \cite[Sec.~7 \& 8]{Mather}. Hence any constructible sheaf in $\Sh_\Lambda(B')/\Loc(B')$ extends to a constructible sheaf in $B$ which is contained in $\Sh_\Lambda(B)/\Loc(B)$ by checking the condition \cite[Cor.~4.23]{GPS3}. This shows the restriction $r^*$ is essential surjective. (See also \cite[Lem.~4.28 \& Cor.~4.32]{KuoLi-spherical}.)
\end{proof}

\begin{remark}
In fact, in \cite[Section 3.4]{Nadler16}, the sheaf of categories is directly defined as a constructible sheaf of categories such that
$$\msh_\Lambda(\Omega) = \Sh_{\Lambda \cup (T^{*}_{\infty}B \backslash \Omega)}(B) / \Sh_{T^{*}_{\infty}B \backslash \Omega}(B).$$
and it is mentioned that when the projection $\pi: \Lambda \to M$ is finite-to-one, one can furthermore write
$$\msh_\Lambda(\Omega) = \Sh_{\Lambda \cup (T^{*}_{\infty}B \backslash \Omega)}(B) / \Sh_{T^{*}_{\infty}B \backslash \Omega}(B) \cong \Sh_\Lambda(B) / \Loc(B).$$
The above lemma shows that our definition agrees with the description there. \qed
\end{remark}

\noindent Here are examples that illustrate properties of the $\msh$ stack. We start with a definition.

\begin{definition}\label{def:K_skeleton} Let $K\sse M$ be a finite set of $n$ positively co-oriented points on a 1-dimensional smooth manifold $M$. By definition, $\mathbb{K}_n(M):= M\cup \nu(K)\sse (T^* M,\la_\st)$ is the 1-dimensional (arboreal) Lagrangian skeleton given by union of the zero section $M\sse T^*M$ and the $n$ positive conormal rays to the points in $K$.\hfill$\Box$
\end{definition}

This Lagrangian skeleton $\bK_n$ corresponds to the case that $\La\sse T^{*}_{\infty} M$ is a set of $n$ points, and $\widehat\La=\bK_n(M)$ is given by the union of the (Lagrangian) cone over $\La$ union the zero section $M\sse T^*M$. We use the cases $M=S^1$ and $M=[0,1]$ in Subsection \ref{ssec:HH0_near_rel_cycle}, cf.~Lemmas \ref{lem:HH_near_cycle} and \ref{lem:HH_near_cycle_pointed}.

\begin{example}\label{ex:K_skeleton}
The dg-category $\msh_{\mathbb{K}_n(M)}({\mathbb{K}_n}(M))$ is equivalent to $\Sh^c_{\mathbb{K}_n}(M)$. It is computed in \cite[Theorem 1.8]{Nadler17}. For instance, for $M=S^1$, loc.~cit.~shows that $\Sh^c_{\mathbb{K}_n}(S^1)$ is equivalent to the dg-category of representations of the path algebra $k\langle Q_K\rangle$ of the cyclic quiver $Q_K$ with vertices $\{v_1,\ldots,v_n\}$. We use this fact in the proof of Lemma \ref{lem:HH_near_cycle}.\hfill$\Box$
\end{example}

\begin{example}\label{ex:ks-stack-local-system}
Let $\Lambda \subset T^{*}_{\infty}\bR^2$ be a Legendrian link with zero Maslov class and choose $\widehat\La$ to be the Lagrangian cone over $\La$. The results of \cite[Part 10]{Guillermou23_SheafSummary} imply that,
for any open subset $\Omega \subset T^{*}_{\infty}\bR^2$,
$$\msh_{\widehat\Lambda}(\Omega) \cong \Loc(\Lambda \cap \Omega).$$
%{\RC Can you add a sentence proving that equality with local systems?}{\WL Sorry, this is already using Guillermou's theorem.}
In particular, since $\msh_\Lambda$ is supported on $\Lambda$, the global sections of the Kashiwara--Schapira stack are
$$\msh_{\Lambda}(T^{*}_{\infty}\bR^2) = \msh_{\Lambda}(\Lambda) \cong \Loc(\Lambda).$$
Note that these global sections depend only on (the homotopy type of) $\La$, and not on its Legendrian embedding into $T^{*}_{\infty}\bR^2$.
\hfill$\Box$
\end{example}

\begin{example}
Let $\Lambda \subset T^{*}_{\infty}\bR^2$ be a Legendrian link and choose $\widehat\Lambda \subset T^*\bR^2$ to be the union of the Lagrangian cone over $\La$ and all the bounded regions of $\bR^2 \setminus \pi(\Lambda)$, where $\pi:T^*M\lr M$ is the natural projection. For any $U \subset \bR^2$, we have
$\msh_{\widehat\Lambda}^\text{pre}(T^*U) = \Sh_{\widehat\Lambda }(U)$. Now, $\Sh_{\widehat\Lambda}$ forms a sheaf of categories and the universal property of sheafification implies that $\msh_{\widehat\Lambda}(T^*U) = \Sh_{\widehat\Lambda }(U)$. $($See also \cite[Appendix B]{JinTreumann}.$)$ %{\RC Can you add a short explanation of why that equality holds?}
In particular, since $\msh_{\widehat\Lambda}$ is supported on $\widehat\Lambda$, global sections of the Kashiwara--Schapira stack are
$$\msh_{\widehat\Lambda}(T^*\bR^2) = \msh_{\widehat\Lambda}(\widehat\Lambda) = \Sh_{\widehat\Lambda}(\bR^2) = \Sh_{\Lambda}(\bR^2)_0,$$
where the subscript $0$ denotes the full subcategory of sheaves with bounded support in $\bR^2$.\hfill$\Box$
\end{example}

For general Weinstein manifolds or Weinstein pairs\footnote{A Weinstein manifold with the data of a Weinstein hypersurface at the ideal contact boundary.}, we can deform the Weinstein structure to obtain a subanalytic Lagrangian skeleton onto which the manifold (or pair) retracts via the inverse Liouville flow, cf.~ \cite[Cor.~7.27]{GPS3}. For the arboreal skeleta introduced in \cite{Nadler17,Starkston}, one can define microlocal sheaves locally and glue the local categories together to define the microlocal sheaf category associated to the Weinstein manifold, cf.~ \cite{Nadler15Nonchar,Nadler17}. This generalizes the Kashiwara-Schapira stack. By \cite{Starkston}, all Weinstein 4-manifolds admit arboreal skeleta. More generally, the microlocal sheaf category can be defined for any subanalytic Lagrangian skeleton of a Weinstein pair endowed with a (stable) polarization, i.e.~a Lagrangian fibration in the tangent bundle.\footnote{We work in this general setting to state Theorem \ref{thm:invariance} below. Weinstein pairs with arboreal singularities are endowed with a natural polarization and are special cases of this general setting.}

\subsubsection{Two properties of the Kashiwara-Schapira stack} We need the following two results on the $\infty$-sheaf $\msh$: Theorem \ref{thm:invariance} and Proposition \ref{prop:stabilization}. The former states invariance of the stack under non-characteristic deformations of the Lagrangian subsets, see e.g.~ \cite{Nadler15Nonchar,NadlerShende,ZhouSheaf}. Such deformations include Liouville homotopies of Weinstein manifolds or pairs, see e.g.~ %\cite[Prop.~2.42]{LazarevSylvanTanaka1} and
\cite[Thm.~9.14]{NadlerShende}. (Confer also \cite[Thm.~1.1 \& Rem.~1.5]{LiNote} or, alternatively, it follows from \cite{GPS3}.)

\begin{thm}[{\cite[Thm.~9.14]{NadlerShende} or \cite[Thm.~1.1]{LiNote}}]\label{thm:invariance}
Let $\bL$ and $\bL'$ be subanalytic Lagrangian skeleta of Weinstein pairs $(X, F, \lambda)$ and $(X, F, \lambda')$ endowed with a polarization. Suppose the pairs $(X, F, \lambda)$ and $(X, F, \lambda')$ are Liouville homotopic. Then there is a quasi-equivalence
$$\msh_{\bL}(\bL) \cong \msh_{\bL'}(\bL').$$
\end{thm}

\begin{example}
$(i)$ Let $f \in \bC[x, y]$ be an isolated plane curve singularity and $\Lambda_f \subset (\bR^3, \xi_{st}) \subset (\bS^3, \xi_{st})$ be the associated Legendrian link, whose positive transverse push-off is the link of the singularity $f$. By  \cite[Theorem 1.1]{CasalsLagSkel}, a choice of real Morsification $\tilde{f}$ of $f$ yields an arboreal skeleton $\bL_{\tilde{f}}$ for the Weinstein pair $(\bC^2, \Lambda_f)$. The skeleton $\bL_{\tilde{f}}$ is given by the union of an exact Lagrangian filling $L_{\tilde{f}}$ of $\Lambda_f$ (smoothly given by the Milnor fiber of $f$) endowed with an $\bL$-compressing system $\D$ whose Lagrangian disks are $\bD^*$ (given by the Lagrangian vanishing thimbles of the Morsification). Therefore, we have
$$\Sh_{\Lambda_f}(\bR^2)_0 \cong \msh_{\bL_{\tilde{f}}}(\bL_{\tilde{f}}),$$
where the subscript $0$ denotes the full subcategory of sheaves with bounded supports in $\bR^2$. Different choices of real Morsifications lead to different Lagrangian skeleta. Theorem \ref{thm:invariance} implies that the global sections of the Kashiwara-Schapira stack on these skeleta are all equivalent: they are quasi-equivalent to $\Sh_{\Lambda_f}(\bR^2)_0$.

\noindent $(ii)$ More generally, let $\La_\beta\sse(\R^3,\xi_\st)$ be the $(-1)$-closure of a positive braid $\beta$, cf.~\cite[Section 2]{CasalsNg}. Suppose that the Lusztig cycles $\D_{\ww}$ of a weave $\ww$ for $\beta\delta(\beta)$ provide a complete $\bL$-compressing system with Lagrangian disks $\D^*_\ww$, e.g.~if $\beta$ contains a $w_0$ subword, and let $L_\ww$ be the filling of $\La_\beta$ associated to $\ww$. See \cite[Section 2]{CasalsZaslow}, \cite[Sections 2\&3]{CW} and \cite[Section 4]{CGGLSS} for details on weaves and $\L$-compressing systems. Then the union $\bL_\ww:=L_\ww\cup\D_\ww^*$ is a relative arboreal Lagrangian skeleton for $(\R^4,\La)$. In this case
$$\Sh_{\Lambda_\beta}(\bR^2)_0 \cong \msh_{\bL_\ww}(\bL_\ww).$$
Note that there are many choices for $\ww$ given a $\beta$. Typically infinitely many if one considers $\beta$ cyclically. In either case, there are many choices of Lagrangian skeleta $\L_\ww$ for a fixed such $\La_\beta$. By Theorem \ref{thm:invariance}, the global sections of their Kashiwara-Schapira stacks are all equivalent: they are quasi-equivalent to $\Sh_{\Lambda_\beta}(\bR^2)_0$, which is independent of $\ww$.
\hfill$\Box$
\end{example}

For the Kashiwara-Schapira stack $\msh_\L$, the existence of restrictions to any open subsets of $\L$ follows from definition. However, we also need to further restrict to certain (lower-dimensional) closed subsets of $\La$. We use the following stabilization formula, which is a special case of the K\"unneth formula for microlocal sheaves, cf.~ \cite{KuoLi-duality}. It can be proved using the invariance under contactomorphisms, cf.~ \cite[Thm.~7.2.1]{KashiwaraSchapira_Book}.

\begin{prop}[{\cite[Thm.~7.2.1]{KashiwaraSchapira_Book} or \cite[Lem.~6.2]{NadlerShende}}]\label{prop:stabilization}
Let $\bL$ be a subanalytic Lagrangian subset endowed with polarization. Then there is a quasi-equivalence
$$\msh_{\bL \times (-1, 1)}(\bL \times (-1, 1)) \cong \msh_\bL(\bL).$$
\end{prop}

%%%%%%%%%%%%%%%%%%%%%%%%%%%%%%%%%%%%%%%%%%
%%%%%%%%%%%%%%%%%%%%%%%%%%%%%%%%%%%%%%%%%%
%%%%%%%%%%%%%%%%%%%%%%%%%%%%%%%%%%%%%%%%%%
\subsection{Corestriction functors for Kashiwara-Schapira stacks}\label{ssec:corestriction}

For two open subsets $\Omega', \Omega \subset \L$ with $\Omega' \subset \Omega$, we now show that the restriction functor $\rho^*: \msh_{\L}(\Omega) \lr \msh_{\L}(\Omega')$ admits both left and right adjoints, and consequently the left adjoint preserves compact objects% \cite[Section 3.6]{Nadler16}
. The left adjoint is denoted by
$$\rho_{!}: \msh_{\L}(\Omega') \lr \msh_{\L}(\Omega)$$
and referred to as the {\it corestriction functor}. See also \cite[Section 3.6]{Nadler16}, \cite[Remark 5.1]{NadlerShende} or \cite[Section 3.2]{KuoLi-spherical}. More generally, we have restriction and corestriction functors for any open immersions $\Omega' \looparrowright \Omega$ (which are local homeomorphisms). The precise result we use reads as follows:

\begin{prop}\label{prop:corestriction}
Let $\L$ be a compact subanalytic Lagrangian subset, endowed with a polarization. Consider two open Lagrangian subsets $\Omega, \Omega'\sse \L$ such that $\Omega' \looparrowright \Omega$ and $\Omega \hookrightarrow \L$. Then the restriction functor
%{\RC Is this restriction functor $\rho^*$ going the right way if $\Omega'\sse\Omega$?} {\WL You are right. I messed up with the notations.}
$$\rho^*: \msh_{\Omega}(\Omega) \lr \msh_{\Omega'}(\Omega')$$
admits a left adjoint $\rho_!$ and a right adjoint $\rho_*$. In particular, the left adjoint, which is called the corestriction functor, preserves compact objects
$$\rho_!: \msh_{\Omega'}(\Omega') \lr \msh_{\Omega}(\Omega).$$
\end{prop}
\begin{proof}
By genericity and invariance, it suffices to consider the case where the Legendrian embedding $\L \subset T^{*}_{\infty}M$ is such that the front projection $\pi: \L \lr M$ is finite-to-one. We claim that there exists a topological basis of $T^{*}_{\infty}M$ on which $\msh_{\L}$ takes values in $\mathrm{Pr}^\mathrm{R}_{\omega,\mathrm{dg},k}$, the $\infty$-category whose objects are compactly generated dg-categories and whose functors preserve limits and (filtered) colimits, cf.~ \cite[Def.~5.5.7.5]{HTT}. The claim is proven as follows. By Lemma \ref{rem:microsheaf-different-definition3}, when $\pi: \L \to M$ is finite-to-one, for any point in $T^*M$, there exists an open ball $B \subset M$ together with a union of open balls $\Omega \subset T^{*}_{\infty}B$ such that $\L \cap T^{*}_{\infty}B \subset \Omega$, and 
$$\msh_{\L}(\Omega) = \Sh_{\L \cup (T^{*}_{\infty}M \setminus \Omega)}(M) / \Sh_{T^{*}_{\infty}M \setminus \Omega}(M) = \Sh_{\L}(B) / \Loc(B).$$
First, \cite[Cor.~4.22]{GPS3} implies that these categories are compactly generated, so sections of $\msh_\L$ indeed are objects of $\mathrm{Pr}^\mathrm{R}_{\omega,\mathrm{dg},k}$. Specifically, these dg-categories are compactly generated by the corepresentatives of the microstalk functors. Second, consider the restriction functor for such open balls $\rho: \msh_{\L}(\Omega) \lr \msh_{\L}(\Omega')$: it is induced by restrictions of sheaves for $r: B' \hookrightarrow B$ via
$$r^*: \Sh_{\L}(B) / \Loc(B) \lr \Sh_{\L}(B') / \Loc(B').$$
Since the restriction functor on all sheaves $r^*: \Sh(B) \lr \Sh(B')$ admits left adjoint $r_!$ and right adjoint $r_*$, and the inclusion $\Sh_{\L}(B) \hookrightarrow \Sh(B)$ admits left and right adjoints by \cite[Thm.~1.2]{Kuo}, we know that on the full subcategories the restriction $r^*: \Sh_{\L}(B) \lr \Sh_{\L}(B')$ also admits left and right adjoints. Since the restriction also preserves local systems, the adjunction descends to localizations $r^*: \Sh_{\L}(B) / \Loc(B) \to  \Sh_{\L}(B') / \Loc(B')$. This shows that, for this basis, $\msh_\L$ sends morphisms to those in $\mathrm{Pr}^\mathrm{R}_{\omega,\mathrm{dg},k}$ and thus establishes the claim.

Therefore, the sections of the sheaf $\msh_{\L}$ on arbitrary open subsets take values in $\mathrm{Pr}^\mathrm{R}_{\omega,\mathrm{dg},k}$ since the category admits all small limits, cf.~\cite[Prop.~5.5.7.6]{HTT}. Hence, for arbitrary open embeddings $\Omega' \hookrightarrow \Omega \hookrightarrow \L$, the restriction functor $\rho: \msh_{\L}(\Omega) \lr \msh_{\L}(\Omega')$, being a morphism in $\mathrm{Pr}^\mathrm{R}_{\omega,\mathrm{dg},k}$, preserves limits and colimits and thus admits both left and right adjoints. The left adjoint $\rho_!: \msh_{\L}(\Omega') \lr \msh_{\L}(\Omega)$ now has a right adjoint which admits a further right adjoint, and hence preserves compact objects. Note that we have a natural equivalence $\msh_{\L}(\Omega) = \msh_\Omega(\Omega)$ if $\Omega\sse\L$, thus proving the result for open embeddings $\Omega' \hookrightarrow \Omega \hookrightarrow \L$.

Finally, for an open immersion (which is a local homeomorphism) $\Omega' \looparrowright \Omega \hookrightarrow \L$, we find an open cover of $\Omega'$ by $\{\Omega_i'\}_{i \in I}$ such that $\Omega_i' \hookrightarrow \Omega$. We have corestriction functors $({\rho_i})_!: \msh_{\Omega_i'}(\Omega_i') \lr \msh_\Omega(\Omega)$ preserving compact objects. By the universal property of global sections on $\Omega'$, we thus obtain a corestriction functor ${\rho}_!: \msh_{\Omega'}(\Omega') \lr \msh_\Omega(\Omega)$ that preserves compact objects.
\end{proof}

\begin{remark}
Via the equivalence between microlocal sheaves and partially wrapped Fukaya categories, established in \cite{GPS3}, we expect that in good situations, the corestriction functor for open inclusions of stratified Lagrangian subsets coincides with the pushfoward functor for proper inclusions of Weinstein sectors, cf.~\cite{GPS1}. In line with \cite[Theorem 1.2 \& Section 3.2]{Kuo} and \cite[Section 3.5]{GPS1}, we also expect that corestriction functors can be characterized by sheaf-theoretic wrappings, i.e.~we expect that applying each such corestriction functor can be expressed as a certain colimit of a diagram of sheaves indexed by positive Hamiltonian isotopies of the Weinstein sector. This would be the symplectic topological incarnation of the pointwise colimit formula for left Kan extensions, see e.g.~\cite[Prop.~4.3.3.10]{HTT}. \hfill$\Box$
\end{remark}

%%%%%%%%%%%%%%%%%%%%%%%%%%%%%%%%%%%%%%%%%%
%%%%%%%%%%%%%%%%%%%%%%%%%%%%%%%%%%%%%%%%%%
%%%%%%%%%%%%%%%%%%%%%%%%%%%%%%%%%%%%%%%%%%

\subsection{Corestriction functors in the pointed setting}\label{ssec:KS_stack_corestriction_relative} Let $\L$ be a Lagrangian skeleton endowed with polarization and $\tt\sse\dd\L$ a set of points. In our case of interest, $\L=L\cup\D^*\sse T^*\R^2$ is the union of an exact Lagrangian filling $L\sse T^*\R^2$ of a Legendrian link $\Lambda=\dd\L\sse T^*_\infty\R^2$, $\D$ is an $\L$-compressing system for $L$ with Lagrangian disks $\D^*$ and $\tt\sse\La$ is a set of basepoints. In many cases the Lagrangian disks in the $\L$-compressing system $\D$ lie in the zero section of $T^*\R^2$.

We study the Kashiwara-Schapira stack supported in $(\L,\tt)$, cf.~ Definitions \ref{def:cat_decorated_sheaves} and \ref{def:stack_decorated_sheaves}, and we use the corestriction functors as in Proposition \ref{prop:corestriction} in this pointed setting, as follows.

\begin{definition}\label{def:cat_decorated_KS}
    Let $\L$ be a compact subanalytic Lagrangian subset, endowed with a polarization, with boundary $\dd \L$ and $\tt \sse \dd \L$ a set of basepoints. By definition, the dg-category $\msh_{\L,\tt}(\L, \tt)$ is the homotopy colimit of the diagram
    \begin{center}
        \begin{tikzcd} 
         \displaystyle\prod_{i=1}^{|\tt|}\modk\ar[r,"m_{\L,\tt}^\ell"] \ar[d,"\Delta^\ell" left] & \msh_\L(\L) \\
        \modk&
    \end{tikzcd}
    \end{center}
    where $m_{\L,\tt}^\ell: \Loc(\tt) \lr \msh_\L(\L)$ is left adjoint of the restriction functor or microstalk functor $m_{\L,\tt}: \msh_\L(\L) \lr \Loc(\tt)$.\hfill$\Box$
\end{definition}

\noindent If $\L \sse T^*\R^2$ is a compact Lagrangian skeleton in the cotangent bundle with boundary $\La = \dd\L \sse T^*_\infty \R^2$, Definition \ref{def:cat_decorated_KS} gives an equivalence
$$\msh_{\L,\tt}(\L, \tt) \cong \Sh_{\La,\tt}(\R^2)_0.$$
The existence of corestriction functors in this pointed setting reads as follows:

\begin{cor}\label{cor:corestriction_Lprime_L}
Let $\L$ be a Lagrangian skeleton with boundary $\dd\L$, endowed with a polarization,  $\Omega'\looparrowright\Omega$ and $\tt \sse \dd \Omega$, $\tt'\sse\dd\Omega'$ two sets of points such that $\tt'=\tt\cap\Omega'$. Then there exists a corestriction functor
$$\rho_!:\msh_{\Omega',\tt'}(\Omega',\tt')\lr\msh_{\Omega,\tt}(\Omega, \tt)$$
and it preserves compact objects.
\end{cor}

\begin{proof}
By Proposition \ref{prop:corestriction}, there is a corestriction functor 
\begin{equation}\label{eq:corestriction_Lprime_L}
\rho_!:\msh_{\Omega'}(\Omega')\lr\msh_{\Omega}(\Omega).
\end{equation} 
Consider the three terms for the homotopy pushout in Definition \ref{def:cat_decorated_KS} for $\Omega'$ and $\Omega$, and let us specify a functor from each term for $\Omega'$ to each corresponding term for $\Omega$. The terms $\msh$, written as sheaf categories in Definition \ref{def:cat_decorated_KS}, map to each other via corestriction, as in (\ref{eq:corestriction_Lprime_L}) above. The terms with the product of the microstalks map to each other via the natural inclusion since $\tt'=\tt\cap\Omega'$ and the functor
$$m^\ell_{\Omega,\tt}:\prod_{i=1}^{|\tt|}\modk^c\lr\msh^c_{\Omega}(\Omega)$$
factors through
$$m^\ell_{\Omega',\tt'}:\prod_{i=1}^{|\tt'|}\modk^c\lr\msh^c_{\Omega'}(\Omega')$$
when we consider the geometric restriction from $\Omega$ to $\Omega'$. The terms with $\modk$ map to each other via the identity functor. Therefore each term in the homotopy pushout for $\Omega'$ maps, as just described, to the corresponding term in the homotopy pushout for $\Omega$. Since homotopy colimits are functorial, the result follows.\end{proof}

\subsection{Lagrangians fillings and Kashiwara-Schapira stacks}\label{ssec:Lagrangian_fillings_KS_stack} In order to relate microlocal merodromies to an actual parallel transport of a local system, we implicitly use Corollary \ref{cor:Lagrangian_fillings_KS_stack} stated below. We first prove Proposition \ref{prop:viterbo-restrict}, which we use to deduce Corollary \ref{cor:Lagrangian_fillings_KS_stack}.

Consider a closed subset $\L' \subset \L$. By Definition \ref{def:KS_stack}, defining the Kashiwara-Schapira stack, there is a natural inclusion $\iota_*: \msh_{\L'} \hookrightarrow \msh_{\L}$. In particular, taking global sections, we have a functor $\iota_*: \msh_{\L'}(\L') \hookrightarrow \msh_{\L}(\L)$. We now show that this inclusion functor admits left and right adjoints, and consequently the left adjoint preserves compact objects.

\begin{prop}\label{prop:viterbo-restrict}
Let $\L',\L$ be compact subanalytic Lagrangian subsets such that $\L' \subset \L$, and each is endowed with a polarization. Then the inclusion functor
$$\iota_*: \msh_{\L'}(\L') \hookrightarrow \msh_{\L}(\L)$$
admits a left adjoint $\iota^*$ and a right adjoint $\iota^!$. In addition, the left adjoint
$$\iota^*: \msh_{\L}(\L) \lr \msh_{\L'}(\L')$$
preserves compact objects.
\end{prop}

\begin{proof}
The argument is similar to that for Proposition \ref{prop:corestriction}. Without loss of generality, we consider a Legendrian embedding $\L \subset T^{*}_{\infty}M$ such that the front projection $\pi: \L \to M$ is finite-to-one. There is a topological basis on which the inclusion functor $\iota_*: \msh_{\L'}(\Omega) \hookrightarrow \msh_\L(\Omega)$ is a morphism in the $\infty$-category $\operatorname{Pr}^\text{R}_{\omega,dg,k}$ of compactly generated dg-categories, i.e.~it preserves limits and (filtered) colimits. Following Remark \ref{rem:microsheaf-different-definition}, for any point in $T^*M$, there exists an open ball $B \subset M$ together with a union of open balls $\Omega \subset T^{*}_{\infty}B$ such that $\L \cap T^{*}_{\infty}B \subset \Omega$, and 
$\msh_{\L}(\Omega) = \Sh_{\L}(B) / \Loc(B)$.
In that case, given the inclusion $\L' \subset \L$, the inclusion functor $\iota_*: \msh_{\L'}(\Omega) \hookrightarrow \msh_{\L}(\Omega)$ is given by the inclusion of sheaves
$$\iota_*: \Sh_{\L'}(B)/\Loc(B) \lr \Sh_{\L}(B)/\Loc(B),$$
which admits left and right adjoints by \cite[Cor.~4.22]{GPS3}. Together with the compact generation of the categories, cf.~\cite[Cor.~4.21]{GPS3}, this shows that on the topological basis the inclusion functor lands in the $\infty$-category $\operatorname{Pr}^\text{R}_{\omega,dg,k}$. Since $\operatorname{Pr}^\text{R}_{\omega,dg,k}$ admits all small limits, for global sections, the inclusion functor also lands in $\operatorname{Pr}^\text{R}_{\omega,dg,k}$ and the statement follows.
\end{proof}

\begin{cor}\label{cor:Lagrangian_fillings_KS_stack}
Let $(X, \lambda_{st})$ be a 4-dimensional Weinstein domain and $L \subset (X, \lambda_{st})$ an exact Lagrangian surface with vanishing Maslov class. Suppose that $L$ is equipped with an $\bL$-compressing system $\D$ whose Lagrangian disks are $\D^*$ such that $\bL=L\cup\D^*$ is a Lagrangian skeleton for $(X, \lambda_{st})$. Then there exists a localization functor
$$\iota^*: \msh_\bL(\bL) \lr \Loc(L).$$
\end{cor}
\begin{proof}
Consider the inclusion $L \subset \bL$. Proposition \ref{prop:viterbo-restrict} applied to $\L'=L$ and $\bL$ gives a localization functor $\iota^*: \msh_\bL(\bL) \lr \msh_L(L)$. Since $L$ is assumed to have vanishing Maslov class, \cite[Part 10]{Guillermou23_SheafSummary} implies the equivalence $\msh_L(L) \cong \Loc(L)$. %(Cf.~also Example \ref{ex:ks-stack-local-system}.)
\end{proof}

\begin{remark} The argument in \cite[Thm.~4.13 and Cor.~4.21]{GPS3} shows that the localization functor $\iota^*: \msh_{\L}(\L) \hookrightarrow \msh_{\L'}(\L')$ is given by localization along all the corepresentatives of microstalk functors on $\L \setminus \L'$. % using the argument in \cite[Thm.~4.13]{GPS3}.
From the perspective of wrapped Fukaya categories, the localization functor can be viewed as a Viterbo restriction functor from $X$ to a Weinstein neighbourhood of the Lagrangian $L$ (symplectomorphic to a disk cotangent bundle $T^*_{<\varepsilon}L$, $\varepsilon\in\R_+$ small enough).\hfill$\Box$
\end{remark}

\begin{remark}
It can be proven that the inclusion functor $\iota_*: \Loc(L) \hookrightarrow \msh_\bL(\bL)$ agrees with the sheaf quantization functor $\Loc(L) \hookrightarrow \Sh_{\Lambda}(\bR^2)_0$ from \cite{NadlerShende} under the equivalence $\msh_\bL(\bL) \cong \Sh_{\Lambda}(\bR^2)_0$, cf.~\cite[Thm.~1.1]{LiNote}. One can also show that the sheaf quantization functors of \cite{NadlerShende} and \cite{JinTreumann} agree in simple cases, cf.~\cite[Prop.~1.8]{LiCobordism1} and \cite[Rem.~B.23]{CasalsLi}. The latter functor is the sheaf quantization functor used to define algebraic torus charts and microlocal holonomies in the moduli of sheaves $\M(\Lambda)$ and $\M(\Lambda, \tt)$, cf.~\cite[Section 4.3]{CW} and \cite[Section 5.1]{CasalsLi}.\hfill$\Box$
\end{remark}

Finally, when we have basepoints on the boundary of the Lagrangian filling, we proceed as follows. Let $L$ be a smooth surface with boundary $\Lambda:=\dd L$ and $\tt \sse \Lambda$ a set of basepoints. By Definition \ref{def:cat_decorated_KS}, the dg-category $\Loc(L, \tt)$ is the homotopy colimit of the diagram
    \begin{center}
        \begin{tikzcd} 
         \displaystyle\prod_{i=1}^{|\tt|}\modk\ar[r,"(i_\tt^*)^\ell"] \ar[d,"\Delta^\ell" left] & \Loc(L) \\
        \modk&
    \end{tikzcd}
    \end{center}
    where $(i_\tt^*)^\ell: \Loc(\tt) \lr \Loc(L)$ is corestriction functor, that is, the left adjoint of the restriction functor $i_\tt^*: \Loc(L) \lr \Loc(\tt)$. The pointed version of Corollary \ref{cor:Lagrangian_fillings_KS_stack} then reads as follows.

\begin{cor}\label{cor:Lagrangian_fillings_KS_stack_pointed}
    Let $(X, \lambda_{st})$ be a 4-dimensional Weinstein domain, $\Lambda \subset \partial_\infty X$ a Legendrian link in the ideal contact boundary with basepoints $\tt \sse \Lambda$, and $L \subset (X, \lambda_{st})$ an exact Lagrangian filling of $\Lambda$ with vanishing Maslov class. Suppose that $L$ is equipped with an $\bL$-compressing system $\D$ such that $\bL:=L\cup\D$ is a Lagrangian skeleton for $(X, \lambda_{st})$. Then there exists a localization functor that preserves compact objects:
    $$\iota^*: \msh_{\bL,\tt}(\bL,\tt) \lr \Loc(L, \tt)$$
\end{cor}

%% file: 4_HH.tex
\section{Hochschild chains and regular functions on moduli stacks}\label{sec:corestriction_HH} This section concludes the proof of Theorem \ref{thm:main}. In order to do that, we first establish a series of results on Hochschild homology and the $\HO$ map (\ref{eq:HO_map}), used in our argument for Theorem \ref{thm:main}. Specifically, Subsection \ref{ssec:general_framework} presents the necessary framework on Hochschild chains. The two key results are then established in Lemma \ref{lem:HH_near_cycle}, in Subsection \ref{ssec:HH0_near_rel_cycle}, and Proposition \ref{prop:HO_in_Rmod}, in Subsection \ref{ssec:computation_HO_near_relative_cycle}.
Subsection \ref{subsec:corestrict-HH} establishes the necessary results that combine the corestriction functors from Section \ref{sec:corestriction} with the Hochschild chains functor.
Theorem \ref{thm:main} is then proven in Subsection \ref{ssec:main_proof}.

%%%%%%%%%%%%%%%%%%%%%%%%%%%%%%%%%%%%%%%%%%%%%%%%%%%%%%%
%%%%%%%%%%%%%%%%%%%%%%%%%%%%%%%%%%%%%%%%%%%%%%%%%%%%%%%
%%%%%%%%%%%%%%%%%%%%%%%%%%%%%%%%%%%%%%%%%%%%%%%%%%%%%%%

\subsection{The general framework}\label{ssec:general_framework} Let $\SC$ be a smooth dg-category and $\M_\SC$ its moduli stack of pseudoperfect objects. By \cite[Prop. 3.4]{ToenVaquie07}, the functor $\SC\longmapsto\M_\SC$ in (\ref{eq:modulistack_functor}) is right adjoint to the functor $L_\textit{perf\,}$ sending a $D^-$-stack to its dg-category of perfect complexes. In particular, we have the equality
$$\Hom_{D^-\operatorname{St}(k)}(\M_\SC,\M_\SC)=\Hom_{\operatorname{Ho(dg-cat)}^{op}}(\SC,L_\textit{perf\,}(\M_\SC)).$$
Therefore, the identity morphism on $\M_\SC$ gives a canonical functor $\Id^\ell:\SC\lr L_\textit{perf\,}(\M_\SC)$. Applying Hochschild chains to this functor we obtain a map
$$\HH_*(\Id^\ell):\HH_*(\SC)\lr \HH_*(L_\textit{perf\,}(\M_\SC)).$$
By \cite[Corollary 4.2]{ToenVezzosiHKR}, in line with the HKR theorem, this yields a map
\begin{equation}\label{eq:HH_to_regular_functions}
\HO:\HH_*(\SC)\lr \Gamma(\M_\SC,\SO_{\M_\SC}).
\end{equation}
For more details, see \cite[Section 5.2]{BravDycker21} or \cite[Section 6.1.2]{BozecCalaqueScherotzke24_RelCritical}, specifically \cite[Example 6.8]{BozecCalaqueScherotzke24_RelCritical}. We choose the notation $\textcolor{blue}{\HO}$ for (\ref{eq:HH_to_regular_functions}) as a mnemonic for {\it ``from \textcolor{blue}{H}ochschild to \textcolor{blue}{{\bf$\SO$}}''}. We often also denote by $\HO$ the restriction $\HO_0$ of $\HO$ in (\ref{eq:HH_to_regular_functions}) to Hochschild 0-chains $\HH_0$, if the context makes that clear.

\noindent In the context where we study a Legendrian link $(\La,\tt)$, we set $\SC=\Sh^c_{\La,\tt}(\R^2)_0$ and consider its moduli of sheaves $\M_\SC=\M({\La,\tt})$. The goal is to obtain (global) regular functions of $\M({\La,\tt})$. By (\ref{eq:HH_to_regular_functions}), we can achieve that by constructing Hochschild classes in $\HH_*(\SC)$. Our next step is thus the construction of classes in $\HH_*(\SC)$, which we achieve using the geometry of Lagrangian fillings and $\L$-compressing systems.

\subsubsection{Some context}\label{rmk:advantag_using_HH}
As it is apparent from our previous works, see e.g.~ \cite{CGGLSS,CW}, it can be challenging to directly construct such regular functions on $\M({\La,\tt})$. Thus far, the only successful strategy has been as follows:
\begin{enumerate}
    \item Find an embedded exact Lagrangian filling $L$ with an $\L$-compressing system. In the chart of local systems $T(L)\sse\M({\La,\tt})$ associated to this filling $L$, one constructs regular functions on that chart: these are natural from the geometry of the $\L$-compressing system and readily seen to be regular on $T(L)$. There is often no simple argument deciding whether such regular functions extend to regular functions on $\M({\La,\tt})$.\\

    \item Employ ad-hoc methods, often exploiting the specific combinatorics (such as plabic graphs in \cite{CW} or weaves in \cite{CGGLSS}) and relation to cluster mutations, so as to argue that the regular functions built on $T(L)$ extend to regular functions on $\M({\La,\tt})$. At core, this is an argument that requires explicit formulas related to the underlying combinatorics. 
    %and applying Hartog's extension theorem (often in the form of the Starfish Lemma). 
    See for instance \cite[Section 4.9 \& Prop. 4.38]{CW} or \cite[Section 5.3]{CGGLSS}, specifically Lemmas 5.26, 5.27 and 5.28 in {\it loc.~cit}.
\end{enumerate}
The strength of using Hochschild classes in $\HH_*(\SC)$ and the $\HO$ map in (\ref{eq:HH_to_regular_functions}) is that Step (2) can be bypassed entirely. The results thus obtained apply in great generality and with barely a need for explicit formulas. In particular, Theorem \ref{thm:main} applies to Legendrian links that are not closures of positive braids.

\begin{comment}
\begin{remark}
The principle of conservation of difficulty is not violated: the work is instead found in the results on existence and properties of certain categories and functors between them and the few key computations in this article. Overall, these include the results developed within the framework of homotopical algebraic geometric, cf.~ \cite{HAGI,HAGII,ToenVaquie07}, the existence of corestriction functors and the properties of the Kashiwara-Schapira stack in microlocal sheaf theory, cf.~ \cite{Guillermou23_SheafSummary,GKS_Quantization,KashiwaraSchapira_Book,Nadler16}, the computations in Section \ref{sec:corestriction} and the results in the subsections below.\hfill$\Box$
\end{remark}
\end{comment}

%%%%%%%%%%%%%%%%%%%%%%%%%%%%%%%%%%%%%%%%%%%%%%%%%%%%%%%
%%%%%%%%%%%%%%%%%%%%%%%%%%%%%%%%%%%%%%%%%%%%%%%%%%%%%%%
%%%%%%%%%%%%%%%%%%%%%%%%%%%%%%%%%%%%%%%%%%%%%%%%%%%%%%%

\color{black}

%%%%%%%%%%%%%%%%%%%%%%%%%%%%%%%%%%%%%%%%%%%%%%%%%%%%%%%
%%%%%%%%%%%%%%%%%%%%%%%%%%%%%%%%%%%%%%%%%%%%%%%%%%%%%%%
%%%%%%%%%%%%%%%%%%%%%%%%%%%%%%%%%%%%%%%%%%%%%%%%%%%%%%%
\subsection{Explicit description of $\HH_0$ near a relative cycle}\label{ssec:HH0_near_rel_cycle} Let $\La\sse (T^*_\infty\R^2,\xi_\st)$ be a Legendrian link with basepoints $T$, and $L\sse (T^*\R^2,\la_\st)$ a Lagrangian filling. Let $\D=\{\g_1,\ldots,\g_{b_1(L)}\}$ be an $\L$-compressing system and $\eta:[0,1]\lr(L,\La\setminus T)$ a relative cycle. Let $\bL:=L\cup \D^*$, where $\D^*=D_1\cup\ldots\cup D_{b_1(L)}$ are the Lagrangian disks $D_i$ associated to $\D$, each $D_i$ with boundary $\dd D_i=\gamma_i$. There exists an open neighborhood $\Op(\eta)\sse \bL$ which is of the form $(0,1)\times I^\sigma_{p_1,\ldots,p_n}$ where
\begin{gather*}I^\sigma_{p_1,\ldots,p_n}:=I\cup I_{p_1}^{\sigma_1}\cup\ldots I_{p_n}^{\sigma_n}\sse[0,1]\times[-1,1],\\ 
I:=\{(x,y)\in [0,1]\times[-1,1]: y=0\},\\
I^+_{p_i}:=\{(x,y)\in [0,1]\times[-1,1]: x=p_i, y\in[0,1]\},\\
I^-_{p_i}:=\{(x,y)\in [0,1]\times[-1,1]: x=p_i, y\in[-1,0]\}
\end{gather*}
where $\sigma=(\sigma_1,\ldots,\sigma_n)\in \{\pm\}^n$ is a choice of signs and $p_i\in(0,1)$ for $i\in[1,n]$. 

See Figure \ref{fig:Filling_Merodromy3} (left) for an example. Note that $I^\sigma_{p_1,\ldots,p_n}$ can be understood as a 1-dimensional (arboreal) Lagrangian skeleton. The symplectic structure can be taken to be that of the cotangent bundle $(T^*I,\la_{st})$ modified accordingly so that the (positive or negative) conormals of the points $p_i\in I$ are part of the Lagrangian skeleton, cf.~\cite[Section 3]{Starkston}. For a {\it positive} relative cycles we have all the signs of $\sigma$ being positive, as in Figure \ref{fig:Filling_Merodromy3} (left). In other words, for a positive relative cycle, a neighborhood $\Op(\eta)\sse \bL$ of $\eta\sse\bL$ retracts to $\bK_n([0,1])$ -- as in Definition \ref{def:K_skeleton} -- for some $n$.

\begin{center}
	\begin{figure}[H]
		\centering
		\includegraphics[scale=0.95]{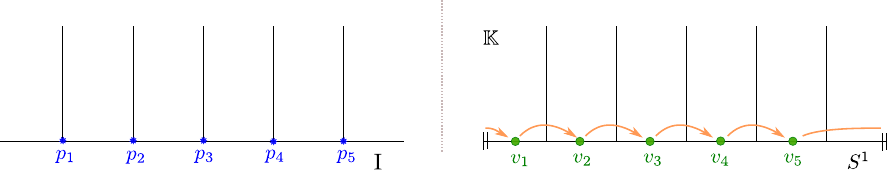}
		\caption{(Left) The 1-dimensional Lagrangian skeleton $I^\sigma_{p_1,\ldots,p_5}$, where all the signs of $\sigma$ are positive. (Right) The quiver associated to the 1-dimensional Lagrangian skeleton $\mathbb{K}$ given by a base $S^1$ and five positive spikes $\nu(K)$, with $K=\{p_1,\ldots,p_5\}$.}
		\label{fig:Filling_Merodromy3}
	\end{figure}
\end{center}

We want to explicitly describe Hochschild 0-chains of the global sections of the Kashiwara-Schapira stack on this 1-dimensional Lagrangian skeleton. (It suffices to describe the 0th Hochschild homology, which we also denote by $\HH_0$ onwards.) In fact, since the microstalks are all identified in $\M{(\La,\tt)}$, it suffices to study the case where the interval $I$ is closed up to a circle. This reduces to the case of the 1-dimensional Lagrangian skeleton $\bK_n:=\bK_n(S^1)$ in Definition \ref{def:K_skeleton}. The result we need for such Lagrangian skeleton $\bK_n$  reads as follows.

\begin{lemma}\label{lem:HH_near_cycle} Let $K\sse S^1$ be a finite set of $n$ positively co-oriented points on a circle $S^1$. %and $\theta_0\in S^1$ a basepoint.
Consider the dg-subcategory $\Sh^c_{\mathbb{K}_n}(S^1)$ of compact objects in the dg-category of sheaves on $S^1$ with singular support on $\mathbb{K}_n$.

\noindent Then there is an isomorphism $\HH_0(\Sh^c_{\mathbb{K}_n}(S^1))\cong k[\rho]\oplus k^{n}$ of vector spaces, where $\rho\in {H}_1(S^1,\Z)$ is the positive generator, geometrically corresponding to the (reduced) absolute cycle given by $S^1$ under $\HH_0(\Sh^c_{\mathbb{K}_n}(S^1)) \lr \HH_0(\Loc^c(S^1))$.
%with basepoint at $\theta_0$.
%That is, going around the zero section $S^1$ once generates $\HH_0(\Sh^c_{\mathbb{K}}(S^1))$ up to .
\end{lemma}

\begin{proof}
By \cite[Theorem 1.8]{Nadler17}, $\Sh^c_{\mathbb{K}_n}(S^1)$ is equivalent to the dg-category of representations of the path algebra $k\langle Q_K\rangle$ of the cyclic quiver $Q_K$ with vertices $\{v_1,\ldots,v_n\}$. Confer Example \ref{ex:K_skeleton} and see Figure \ref{fig:Filling_Merodromy3} (right) for the case of the 5-cycle quiver, $n=5$. For each point in $K$, this path algebra contains a unique arrow $a_i$ from $v_i$ to $v_{i+1}$. (Cyclically understood, i.e.~there is an arrow from $v_n$ to $v_1$.) In Figure \ref{fig:Filling_Merodromy3}, vertices are depicted in green and arrows in orange. Both such vertices $v_i,v_{i+1}$, come equipped with idempotents $e_i,e_{i+1}$ that act as identities at the sources and targets of arrows, each corresponding to the constant path at the corresponding vertex, see e.g.~\cite{DerksenWeyman05_Notices}. The path algebra $k\langle Q_K\rangle$ is linearly generated by these elements $a_i,e_i$, $i\in[1,n]$, and the algebra structure is given by concatenation of paths. In particular, we have $[e_i,a_i]=a_ie_i-e_ia_i=a_ie_i=a_i$ for the commutators if $n\geq2$. Geometrically, each vertex corresponds to the microlocal stalk of a sheaf at each of the intervals in $S^1\setminus K$ and the arrows capture the microlocal parallel transport between them.

\noindent Since Hochschild homology is a derived Morita invariant, e.g.~\cite[Prop.~2.4.3]{McCarthy_HH_of_category}, we have
$$\HH_*(\Sh^c_{\mathbb{K}_n}(S^1))\cong\HH_*(k\langle Q_K\rangle).$$
For any algebra $R$, there is a vector space isomorphism $\HH_0(R,R)\cong R/[R,R]$, where the trace space $R/[R,R]$ is the vector space obtained by quotienting $R$ (as a vector space) by the vector space generated by all the commutators in $R$. For the cyclic quiver $Q_K$, we claim that $\HH_0(k\langle Q_K\rangle)\cong k[\rho]\oplus k^n$, where $\rho\in k\langle Q_K\rangle$ is the cycle element $a_1a_2\ldots a_n$ and the summand $k^n$ is spanned by the idempotents. Indeed, any element $\gamma\in k\langle Q_K\rangle$ corresponding to a path in $Q_K$ whose source vertex $s(\gamma)$ is different from its target vertex $t(\gamma)$ is contained in $[k\langle Q_K\rangle,k\langle Q_K\rangle]$ as $[e_{s(\gamma)},\gamma]=\gamma$. Similarly, the cycle $a_1a_2\ldots a_n$ is equivalent to any cyclic permutations $a_ia_{i+1}\ldots a_{i-1}$. Therefore $\HH_0(k\langle Q_K\rangle)\cong k[\rho]\oplus k\langle e_1,\ldots,e_n\rangle$, as claimed. 

\noindent Finally, we note that the equivalence from \cite[Theorem 1.8]{Nadler17} sends the quiver cycle $\rho$ to a homological representative $\rho$ of $S^1\sse T^* S^1$. In fact, consider the localization functor $\Sh^c_{\bK_n}(S^1) \lr \Loc^c(S^1)$ given by Corollary \ref{cor:Lagrangian_fillings_KS_stack}. Let $Q_0$ be the quiver with one vertex and a single loop $\rho$. Then $\Loc^c(S^1)$ is equivalent to the dg-category of representations of the path algebra localized along the loop $k\left<Q_0 \right>_\rho$. Using \cite[Section 4.2]{Nadler17}, the localization functor $\Sh^c_{\bK_n}(S^1) \lr \Loc^c(S^1)$ can be described by the localization morphism on path algebras $k\left< Q_K\right> \to k\left< Q_0 \right> \to k\left< Q_0 \right>_\rho$ which localizes $a_1, a_2, \dots, a_n$ and sends the product $a_1a_2 \dots a_n$ to the invertible cycle $\rho$. Therefore, on Hochschild homology, we have 
$$\HH_0(\Sh^c_{\bK_n}(S^1)) \xrightarrow{\sim} k[\rho] \oplus k^n \to k[\rho^{\pm 1}] \xrightarrow{\sim} \HH_0(\Loc^c(S^1)).$$
Under the isomorphism $\HH_0(\Loc^c(S^1)) \cong \HH_0(k[\pi_1(S^1)]) \cong k[\bZ]$, cf.~ Example \ref{ex:localsystem_circle}, $\rho$ corresponds to the absolute cycle $\rho \in H_1(S^1, \bZ)$.
\end{proof}

\begin{remark}
In the proof of Theorem \ref{thm:main}, we only need the free positive generator $\rho \in \HH_0(\Sh_{\bK_n}^c(S^1))$. We could then avoid the computation of $\HH_0(\Sh_{\bK_n}^c(S^1))$ and reduce instead to the case $n=1$.
%in the general case by considering the Legendrian degeneration from $n$ cooriented points $K$ to a single cooriented point $K_1$ \cite[Lem.~7.1]{TreumannWilliamsZaslow} 
Indeed, joining the $n$ points together gives rise to a functor
$$\Sh_\bK^c(S^1) \lr \Sh_{\bK_1}^c(S^1).$$
See ~e.g.~\cite[Prop.~2.12]{ZhouSheaf}. Then, we could use the techniques in \cite[Section 5]{KuoLi-CY}, developed for computations of the Hochschild homology, and show that the induced map on Hochschild homology sends $\rho$ to $\rho$ under the identifications:
$$\HH_0(\Sh_{\bK_n}^c(S^1)) \xrightarrow{\sim} k[\rho] \oplus k^{n} \to k[\rho] \xrightarrow{\sim} \HH_0(\Sh_{\bK_1}^c(S^1)).$$
We mention this alternative argument for completeness: it is not needed for Theorem \ref{thm:main}.
\hfill$\Box$
\end{remark}

%{\RC Add remark about doing it via many spikes to one. Reference is ``Treumann-Williams-Zaslow'' the Kastelyn paper.}

Similarly, we can consider the $1$-dimensional skeleton $\bK_n(D^1) \subset T^*D^1$ marked with basepoints and obtained the same Hochschild homology, as follows.

\begin{definition}
Let $\bK \subset T^*D^1$ be a conic Lagrangian susbet. By definition, the dg-category $\Sh_{\mathbb{K}, \partial D^1}(D^1)$ is the homotopy colimit of the diagram
\begin{center}
    \begin{tikzcd}
         \displaystyle\prod\nolimits_{i=1}^{2}\modk\ar[r,"(i_{\partial D^1}^*)^\ell"] \ar[d,"\Delta^\ell" left] & \Sh_{\mathbb{K}}(D^1) \\
        \modk&
    \end{tikzcd}
\end{center}
where $(i_{\partial D^1}^*)^\ell: \Loc(\partial D^1) \lr \Sh_{\bK}(D^1) $ is corestriction functor which is left adjoint to the restriction functor $i_{\partial D^1}^*: \Sh_{\bK}(D^1) \lr \Loc(\partial D^1)$.\hfill$\Box$
\end{definition}

Lemma \ref{lem:HH_near_cycle} then implies the following result:

\begin{cor}\label{lem:HH_near_cycle_pointed} Let $K\sse D^1$ be a finite set of $n$ positively co-oriented points on a closed interval $D^1$. Consider the dg-subcategory $\Sh^c_{\mathbb{K}_n, \partial D^1}(D^1)$ of compact objects in the dg-category of sheaves on $D^1$ with singular support on $\mathbb{K}_n$ and basepoints given by $\partial D^1$.

\noindent Then there is an isomorphism $\HH_0(\Sh^c_{\mathbb{K}_n, \partial D^1}(D^1))\cong k[\rho]\oplus k^{n}$, where $\rho\in {H}_1(D^1, \partial D^1,\Z)$ is the positive generator corresponding to the relative 1-cycle given by $(D^1, \partial D^1)$ under $\HH_0(\Sh^c_{\mathbb{K}_n,\partial D^1}(D^1)) \lr \HH_0(\Loc^c(D^1, \partial D^1))$.
\end{cor}
\begin{proof}
Note that there is an equivalence $\Sh_{\bK_n(S^1)}(S^1) \cong \Sh_{\bK_n(D^1),\partial D^1}(D^1)$ as follows. Let $N \subset S^1 \setminus K$ be a closed interval. Then we have a homotopy colimit diagram
\begin{center}
    \begin{tikzcd}
        \Loc(\partial N) \ar[r, "(i^*_{\partial N})^\ell"] \ar[d, "(j^*_{\partial N})^\ell"] & \Sh_{\bK_n}(S^1 \setminus N) \ar[d] \\
        \Loc(N) \ar[r] & \Sh_{\bK_n}(S^1).
    \end{tikzcd}
\end{center}
Since the corestriction functor $(j^*_{\partial N})^\ell: \Loc(\partial N) \lr \Loc(N)$ agrees with the corpoduct functor $\Delta^\ell: \prod_{i=1}^2\modk \lr \modk$, and the corestriction functor $(i^*_{\partial N})^\ell: \Loc(\partial N) \lr \Sh_{\bK_n}(S^1 \setminus N)$ agrees with the corestriction functor $(i^*_{\partial D^1})^\ell: \Loc(\partial D^1) \lr \Sh_{\bK_n}(D^1)$, we conclude that
$$\Sh_{\bK_n}(S^1) \cong \Sh_{\bK_n,\partial D^1}(D^1).$$
Then the result on Hochschild homology follows from Lemma \ref{lem:HH_near_cycle}.
\end{proof}

%%%%%%%%%%%%%%%%%%%%%%%%%%%%%%%%%%%%%%%%%%%%%%%%%%%%%%%
%%%%%%%%%%%%%%%%%%%%%%%%%%%%%%%%%%%%%%%%%%%%%%%%%%%%%%%
%%%%%%%%%%%%%%%%%%%%%%%%%%%%%%%%%%%%%%%%%%%%%%%%%%%%%%%
\subsection{Computation of $\HO$ map near a relative cycle}\label{ssec:computation_HO_near_relative_cycle} The goal of this subsection is to prove Proposition \ref{prop:HO_in_Rmod}, which allows us to describe the map $\HO$ in (\ref{eq:HO_map}) in the cases necessary for Theorem \ref{thm:main}. Let $R$ be a unital ring and consider the dg-category $\SC=\Perf(R)$ of perfect $R$-modules. A point
$$x:\Spec(k)\lr\M_\SC$$
classifies a functor $\varphi_x:\Perf(R)\lr\Perf(k)$ corepresented by $x^!:=R^!\otimes_R\varphi_x^r(k)$, i.e.~$\varphi_x\simeq \Hom_R(x^!,-)$. Let $(V,\rho)$, $\rho:R\lr\End_k(V)$, be the underlying $R$-module of $\varphi_x^r(k)$ and write $x:=\varphi_x^r(k)$ to ease notation, as in \cite[Example 3.7]{BravDycker21}. Note that $V$ is a perfect $k$-module and thus a finite-dimensional $k$-vector space. Hochschild 0-chains are given by $\HH_0(\Perf(R))\cong\HH_0(R)\cong R/[R,R]$, the trace space of $R$, cf.~\cite[Prop.~2.4.3]{McCarthy_HH_of_category}. That is, $\HH_0(\Perf(R))$ is the vector space spanned by the elements of $R$ quotiented by the subvector space spanned by commutators. Given $r\in R$, we let $[r]\in R/[R,R]$ denote its associated Hochschild 0-chain. The following result computes $\HO([r])$:

\begin{prop}\label{prop:HO_in_Rmod}
Let $\SC=\Perf(R)$ be the dg-category of perfect $R$-modules over a ring $R$ that is homologically smooth. Then 
$$\HO:\HH_0(\SC)\lr \mathrm{H}^0\Gamma(\M_\SC,\SO_{\M_\SC}),\quad \HO([r])(x)=\tr(\rho(r)),$$
where $x=(V,\rho)\in\M_\SC$ is a point $x:\Spec(k)\lr\M_\SC$ and $\tr$ denotes the trace of the endomorphism $\rho(r)\in\End_k(V)$.
\end{prop}

\begin{proof}
Let us use the construction and notation from \cite[Section 3.2]{BravDycker19}, with ${\EuScript A}:=\Perf(R)$ and ${\EuScript P}:=\End_R(V)$, seen as a one object category with endomorphism $\End_R(V)$. Consider the dg-functor
$$\Perf(R)^{op}\otimes {\EuScript P}\lr\Perf(k),\quad (a,p)\mapsto\Hom_R(a,p),$$
given by the $R$-module action
\begin{equation}\label{eq:Rmod_action}
R\otimes \End_R(V)\lr \End_k(V),\quad r\otimes f\longmapsto [r f:V\to V, v\mapsto(\rho(r)\cdot f)(v)].
\end{equation}
Applying Hochschild chains to this functor yields 
$$\HH_*(R)\otimes\HH_*(\End_R(V))\lr\HH_*(\End_k(V)),$$
which is equivalent to the map
$$\Theta:\HH_*(\Perf(R)^{op})\lr \RHom(\HH_*({\EuScript P}),\HH_*(\Perf(k)))$$
in \cite[Section 3.2]{BravDycker19} after using the isomorphism $\HH_*(\Perf(R)^{op})\cong\HH_*(\Perf(R))$ and currying. Now use the Morita invariance isomorphism $\HH_*(\End_k(V))\cong k$, equiv.~$\HH_*(\Perf(k))\cong k$, see e.g.~ from \cite[Prop.~9.5.2]{Weibel94} and \cite[Prop.~2.4.3]{McCarthy_HH_of_category}, which is given by the trace. Therefore, starting with the functor (\ref{eq:Rmod_action}), applying chains $\HH$, using the two isomorphisms above, and now restricting also to 0-chains, we obtain the pairing
\begin{equation}\label{eq:Rmod_action_trace}
R\otimes \End_R(V)\lr k,\quad r\otimes f\longmapsto [\tr(\rho(r)\cdot f)].
\end{equation}

\noindent In the notation of \cite[Section 3.2]{BravDycker19}, this corresponds to the map
\begin{equation}\label{eq:ThetaMap}
\Theta:\HH_*(\Perf(R))\lr\Hom_k(\HH_*({\EuScript P}),k)
\end{equation}
appearing in the first diagram of the proof of \cite[Theorem 3.1]{BravDycker19}. If ${\EuScript P}^*$ denotes the right dual of ${\EuScript P}$, there is also a canonical equivalence $\Psi$ of dg-modules: \begin{equation}\label{eq:PsiMap}
\Psi:\Hom_k(\HH_*({\EuScript P}),k)\lr \Hom_{{\EuScript P}^e}({\EuScript P},{\EuScript P}^*),
\end{equation}
as in \cite[Equation (3.1)]{BravDycker19}, since $\HH_*({\EuScript P})={\EuScript P}\otimes_{{\EuScript P^e}}{\EuScript P}$ and $-\otimes_k{\EuScript P}^*$ is right adjoint to $-\otimes_{\EuScript P^e}{\EuScript P}$. Note that the codomain of $\Psi$ is just an $R^e$-linear map from $\End_R(V)$ to $\End_R(V)^*$. By (\ref{eq:Rmod_action_trace}), the image $(\Psi\circ\Theta)([r])$ of a Hochschild 0-chain $[r]\in R/[R,R]$ is given by the unique $R^e$-linear map that sends the identity $\id_V\in\End_R(V)$ to the linear functional
\begin{equation}\label{eq:composition_trace}
\phi_r:\End_R(V)\lr k,\quad \phi(r)(v)=\tr(\rho(r)\cdot v).
\end{equation}

By the first diagram in the proof of \cite[Theorem 3.1]{BravDycker19}, the composition $\Psi\circ\Theta$ computes the natural transformation $\alpha_r:\mbox{Id}^!\lr \mbox{Id}$ adjoint to $r$ restricted to $x$. Here we have used the isomorphism
$$\HH_*(\Perf(R))\cong \Hom_{\Mod(R^e)}(\mbox{Id}^!,\mbox{Id}),$$
from  \cite[Diagram (4.9)]{BravDycker21}, and the composition (4.14) in \cite[Section 4.2]{BravDycker21}, with $\mbox{Id}$ being the identity endofunctor of $\Mod(R)$. By \cite[Proposition 4.7]{BravDycker21}, or \cite[Proposition 5.4]{BravDycker21}, the value of $\HO([r])$ at the point $x\in\M_\SC$ is given by the image of $1\in k$ under the composition
$$k \xrightarrow{\Id} \Hom^0_{R}(M, M) \xrightarrow{(\Psi\circ\Theta)([r])} \Hom^0_{R}(M, M)^\vee \xrightarrow{u^*} k,$$
where $u^*$ is the dual of the unit in $\modk$, cf.~\cite[Corollary 2.6]{BravDycker21}. Since $(\Psi\circ\Theta)([r])$ is computed by (\ref{eq:composition_trace}), we conclude that $\HO([r])(x)=\tr(\rho(r))$.
\end{proof}

\color{black}

\begin{example}\label{ex:localsystem_circle}
A first case to apply Proposition \ref{prop:HO_in_Rmod} is that of sheaves in the circle $S^1$ with no singular support outside of the zero section, i.e.~local systems on $S^1$. From a contact viewpoint, this corresponds to studying the empty Legendrian in $T^*_\infty S^1$. Consider the dg-derived category $\SC=\Loc^c(S^1)$ of perfect local systems on $S^1$. Fix a basepoint $x\in S^1$ and note that the based loop space $\Omega_{x} S^1$ has a set of connected components indexed by the winding number of a loop, i.e.~indexed by an integer, and all components are contractible. Since $\Omega_{x} S^1$ is a group, $\pi_0(\Omega_{x} S^1)$ inherits the structure of a group which is isomorphic to $\pi_0(\Omega_{x} S^1)\cong\pi_1 S^1\cong\Z$; this coincides with the product structure of the Pontryagin product. This isomorphism sends the class of a loop winding positively around $S^1$ once to the generator $1\in\Z$.

In fact, $\SC\cong \Mod(C_{-*}(\Omega_x S^1))^c\cong \Mod(k[\Z])^c\cong \Mod(k[t^{\pm1}])^c$ where the monodromy of a local system is given, under this isomorphism, by the action of $t$ on a perfect module, see e.g.~\cite[Appendix A]{CasalsLi}. Therefore, we are within the framework in Proposition \ref{prop:HO_in_Rmod}. Note that $\M_\SC$ parametrizes the objects of the subcategory of perfect $k[t^{\pm1}]$-modules that are perfect over $k$, i.e.~local systems with stalks given by finite-dimensional bounded complexes of $k$-modules.  By derived Morita invariance, $\HH_*(\SC)\cong \HH_*(k[t^{\pm1}])$ and Proposition \ref{prop:HO_in_Rmod} implies that the map $\HO$ in (\ref{eq:HH_to_regular_functions}) is given by
$$\HH_0(\Mod(k[t^{\pm1}])^c)\cong k[t^{\pm1}]\lr \Gamma(\M_\SC,\SO_{\M_\SC}),\quad t\longmapsto((V,\rho)\mapsto\tr_{\rho}(t)).$$
That is, the regular function $\HO(t)$ on a (proper) local system $\SL=(V,\rho)$ given by the element $t\in\HH_0(\SC)$ on $\M_\SC$ is given by the trace of the monodromy of the local system $\SL$.\hfill$\Box$
\end{example}

\begin{example}\label{ex:localsystem_0dim} A variation on Example \ref{ex:localsystem_circle} is adding a unique point $\La=\{\nu(x)\}\in T^{*}_\infty S^1$ as singular support, where $x\in S^1$ is a fixed basepoint and $\nu(x)$ its positive conormal lift to $T^{*}_\infty S^1$. Consider the dg-derived category $\cC = \Sh_\Lambda^c(S^1)$ of compact sheaves on $S^1$ with microsupport on such a $\Lambda$. As in the proof of Lemma \ref{lem:HH_near_cycle}, a sheaf in $\Sh_\Lambda^c(S^1)$ is determined by a stalk $V$ in $S^1\setminus \{x\}$ together with an endomorphism $\rho: V \to V$. Namely, via \cite[Theorem 1.8]{Nadler17} it corresponds to representations of the path algebra of the one-vertex quiver with a unique (loop) edge. Therefore, $\cC \cong \Mod(k[\bN])^c \cong \Mod(k[t])^c$, where the endomorphism $\rho: V \to V$ is determined by the action of $t$ on a perfect module. Therefore, by derived Morita equivalence and Proposition \ref{prop:HO_in_Rmod}, the $\HO$ map is
$$\HH_0(\Mod(k[t])^c)\cong k[t]\lr \Gamma(\M_\SC,\SO_{\M_\SC}),\quad t\longmapsto((V,\rho)\mapsto\tr_{\rho}(t)).$$
That is, the regular function $\HO(t)$ given by the element $t\in\HH_0(\SC)$ evaluated at a (proper) sheaf $\SF=(V,\rho)\in \M_\SC$ is given by the trace of the parallel transport.\hfill$\Box$
\end{example}

%%%%%%%%%%%%%%%%%%%%%%%%%%%%%%%%%%%%%%%%%%
%%%%%%%%%%%%%%%%%%%%%%%%%%%%%%%%%%%%%%%%%%
%%%%%%%%%%%%%%%%%%%%%%%%%%%%%%%%%%%%%%%%%%

\subsection{Corestriction maps and Hochschild homology}\label{subsec:corestrict-HH}
Let $\cC$ and $\cD$ be two small idempotent-complete dg-categories and $f: \cC \lr \cD$ a dg-functor. Then there is a canonical map between their Hochschild homologies 
$$\HH_*(f): \HH_*(\cC) \lr \HH_*(\cD),$$
and a canonical map between their derived moduli stacks
$$\M(f): \M(\cD) \lr \M(\cC).$$
Let $\Omega,\Omega' \subset \L$ be open subsets in the compact subanalytic Lagrangian $\L$ such that $\Omega' \looparrowright \Omega$. Consider the corestriction functor $\rho_!: \msh_{\Omega'}^c(\Omega') \lr \msh_{\Omega}^c(\Omega)$. We have the following commutative diagram that intertwines the $\HO$-maps of the two categories.

\begin{prop}\label{prop:HH_corestriction_diagram}
Let $\L$ be a compact subanalytic Lagrangian subset with polarization and let $\Omega, \Omega' \subset \L$ be open subsets in the subanalytic Lagrangian such that $\Omega' \looparrowright \Omega$ is an open immersion. Then the corestriction functor induces a commutative diagram
%{\RC Again, is the map going the right way here?} {\WL You are right. I messed up with the notations.}
\begin{equation}\label{eq:HH_corestriction_diagram}
\begin{tikzcd}
\HH_*(\msh_{\Omega'}^c(\Omega')) \arrow[r] \arrow[dd, "\HO(\msh_{\Omega'}^c(\Omega'))"'] & \HH_*(\msh_{\Omega}^c(\Omega)) \arrow[dd, "\HO(\msh_{\Omega}^c(\Omega))"] \\
& \\
\Gamma(\M(\msh_{\Omega'}^c(\Omega')), \cO_{\M(\msh_{\Omega'}^c(\Omega'))}) \arrow[r] & \Gamma(\M(\msh_{\Omega}^c(\Omega)), \cO_{\M(\msh_{\Omega}^c(\Omega))}).
\end{tikzcd}
\end{equation}
where the morphism of moduli spaces $\M(\msh_{\Omega}^c(\Omega)) \to \M(\msh_{\Omega'}^c(\Omega'))$ sends an object $\SF \in \M(\msh_{\Omega}^c(\Omega))$ to its restriction $\rho^*\SF \in \M(\msh_{\Omega'}^c(\Omega'))$.
\end{prop}
\begin{proof}
The corestriction functor $\rho_!: \msh_{\Omega'}^c(\Omega') \lr \msh_{\Omega}^c(\Omega)$ induces a canonical map on Hochschild homologies
$$\HH_*(\rho_!): \HH_*(\msh_{\Omega'}^c(\Omega')) \lr \HH_*(\msh_{\Omega}^c(\Omega))$$
and a canonical map between their derived moduli stacks
$$\M(\rho_!): \M(\msh_{\Omega}^c(\Omega)) \lr \M(\msh_{\Omega'}^c(\Omega')),$$
where a proper object $\SF \in \M(\msh_{\Omega}^c(\Omega))$ is sent to the restriction $\rho^*\SF \in \M(\msh_{\Omega'}^c(\Omega'))$. Since the $\HO$-map from Hochschild homology to regular functions on the derived moduli stack is functorial, we can obtain the natural commutative diagram.
\end{proof}

\begin{cor}\label{cor:HH_corestriction_K_to_L}
Let $L \subset (X, \lambda_{st})$ be an embedded exact Lagrangian surface with Legendrian boundary $\Lambda$ and $\tt \sse \Lambda$ a set of basepoints. Suppose that $L$ is endowed with an $\bL$-compressing system $\D$ whose Lagrangian disks are $\D^*$ and denote by $\bL:=L\cup\D^*$ the associated Lagrangian skeleton.

\noindent Then for a relative 1-chain $\eta \looparrowright (L, \tt)$ with small neighborhood $\bK\sse\L$, there exists a commutative diagram
%\arrow[dd, "\HO(\msh_{\bK,\tt}^c(\bK,\tt))"
%\arrow[dd, "\HO(\msh_{\bL,\tt}^c(\bL,\tt))"]
\begin{equation}\label{eq:diagram_HO_subskeleton}
\begin{tikzcd}
\HH_*(\msh_{\bK,\tt}^c(\bK,\tt)) \arrow[r,"\HH_*(\rho_!)"] \arrow[dd,"\HO_\bK" left] & \HH_*(\msh_{\bL,\tt}^c(\bL,\tt)) \arrow[dd,"\HO_\bL"] \\
& \\
\Gamma(\M(\msh_{\bK,\tt}^c(\bK,\tt)), \cO_{\M(\msh_{\bK,\tt}^c(\bK,\tt))}) \arrow[r] & \Gamma(\M(\msh_{\bL,\tt}^c(\bL,\tt)), \cO_{\M(\msh_{\bL,\tt}^c(\bL,\tt))})
\end{tikzcd}
\end{equation}
where the vertical maps are $\HO_\bK:=\HO(\msh_{\bK,\tt}^c(\bK,\tt))$ and $\HO_\bL:=\HO(\msh_{\bL,\tt}^c(\bL,\tt))$, and the corestriction functor $\rho_!$ in the horizontal map is the left adjoint to the restriction from $\L$ to the open subskeleton $\bK\times(-\delta,\delta)$ composed with the stabilization isomorphism from $\bK\times(-\delta,\delta)$ to $\bK$.
\end{cor}

\begin{cor}\label{cor:HH_Lagrangian_filling}
Let $L$ be a smooth surface with boundary $\Lambda$ and $\tt \sse \Lambda$ a set of basepoints. Consider an immersed relative 1-chain $\eta \looparrowright (L, \tt)$. Then there exists a commutative diagram
\begin{equation}\label{eq:diagram_HO_localsystems}
\begin{tikzcd}
\HH_*(\Loc^c(\eta, \tt\cap\eta)) \arrow[r] \arrow[dd,"\HO_\eta"] & \HH_*(\Loc^c(L, \tt)) \arrow[dd,"\HO_L"] \\
& \\
\Gamma(\M(\Loc^c(\eta, \tt\cap\eta)), \cO_{\M(\Loc^c(\eta, \tt\cap\eta))}) \arrow[r] & \Gamma(\M(\Loc^c(L, \tt)), \cO_{\M(\Loc^c(L, \tt))}).
\end{tikzcd}
\end{equation}
where $\HO_\eta:=\HO(\Loc^c(\eta, \tt\cap\eta))$ and $\HO_L:=\HO(\Loc^c(L, \tt))$, and $\tt\cap\eta=\dd\eta$ by construction.

In addition, suppose that $L$ is an exact Lagrangian surface with Legendrian boundary endowed with an $\bL$-compressing system and $\eta \looparrowright (L, \tt)$ is a positive relative 1-chain. Then the diagram (\ref{eq:diagram_HO_localsystems}) above is compatible with diagram (\ref{eq:diagram_HO_subskeleton}) in Corollary \ref{cor:HH_corestriction_K_to_L} via the localization functors $\msh^c_{\L,\tt}(\L,\tt) \lr \Loc^c(L,\tt)$ and $\msh_{\bK,\tt}^c(\bK,\tt) \lr \Loc^c(\eta, \tt)$.
%{\WL Ideally, we may want to draw a diagram of 8 terms, but I am not sure if that will make it look too complicated.}
\end{cor}

The compatibility of diagrams (\ref{eq:diagram_HO_subskeleton}) and (\ref{eq:diagram_HO_localsystems}), as stated in Corollary \ref{cor:HH_Lagrangian_filling}, refers to the fact that the corresponding cube diagram, with the corresponding eight terms and functors between them, is commutative.

%%%%%%%%%%%%%%%%%%%%%%%%%%%%%%%%%%%%%%%%%%%%%%%%%%%%%%%
%%%%%%%%%%%%%%%%%%%%%%%%%%%%%%%%%%%%%%%%%%%%%%%%%%%%%%%
%%%%%%%%%%%%%%%%%%%%%%%%%%%%%%%%%%%%%%%%%%%%%%%%%%%%%%%

\subsection{Proof of Theorem \ref{thm:main}}\label{ssec:main_proof} Consider the relative Lagrangian skeleton $\bL:=L\cup\D^*$, where $\D^*$ is the union of Lagrangian disks. By hypothesis, $\eta$ is $\D$-positive and therefore a small enough neighborhood $\Op(\eta)\sse\L$ of $\eta\sse\bL$ is of the form $I^\sigma_{p_1,\ldots,p_n}\times(-\delta,\delta)$, with $I^\sigma_{p_1,\ldots,p_n}$ as in Subsection \ref{ssec:HH0_near_rel_cycle}, $\sigma$ all positive signs, and $\delta\in\R_+$ small enough. The points $p_1,\ldots,p_n\in \eta\sse L$ are the intersection points of $\eta$ with the boundaries of the curves in the $\L$-compressing system $\D$. In other words, $\Op(\eta)\cong \bK_n\times(-\delta,\delta)$ for some $n\in\N$, where $\bK_n$ as in Definition \ref{def:K_skeleton}.

By Proposition \ref{prop:stabilization}, the global sections of the Kashiwara-Schapira stack $\msh$ coincide for the 2-dimensional Lagrangian skeleton $\bK_n\times(-\delta,\delta)$ and the 1-dimensional Lagrangian skeleton $\bK_n$. We therefore choose $\bK_n$ as a model for a closed neighborhood of $\eta$, i.e.~we use $\eta\sse\bK_n\sse\L$ instead of $\eta\sse\Op(\eta)\sse\L$. To ease notation, we write $\bK:=\bK_n$, as the particular value of $n$ has little role in the argument. Since $\La$ is $\tt$-pointed and $\eta$ is a relative 1-cycle in $(L,\tt)$, $\bK$ inherits two natural basepoints, one at each boundary point $\dd\eta$. We denote this pointed skeleton by $(\bK,\tt)$. We construct the Hochschild cycle $H_\eta\in\HH_0(\Sh_{\La,\tt}(\R^2)_0)$ in the statement of Theorem \ref{thm:main} as follows:
\begin{enumerate}
    \item Construct a cycle $h_\eta\in\HH_0(\msh_{\bK,\tt}(\bK,\tt))$, using Corollary \ref{lem:HH_near_cycle_pointed} in Subsection \ref{ssec:HH0_near_rel_cycle}.

    \item Map the cycle $h_\eta\in\HH_0(\msh_{\bK,\tt}(\bK,\tt))$ to $\HH_0(\msh_{\L,\tt}(\L,\tt))$ via the corestriction functor $\rho_!$ induced by the immersion $\bK\looparrowright \L$. This uses Corollary \ref{cor:HH_corestriction_K_to_L} in Subsection \ref{subsec:corestrict-HH}.
\end{enumerate}

For Step (1), the cycle $h_\eta\in\HH_0(\msh_{\bK,\tt}(\bK,\tt))$ is defined as follows. By a pointed version of Example \ref{ex:K_skeleton}, $\msh_{\bK,\tt}(\bK,\tt)$ can be identified with $\Sh^c_{\mathbb{K},\tt}(\eta)\cong\Sh^c_{\mathbb{K},\tt}(D^1)$, where $\bK\sse T^*\eta$ is considered as a relative Lagrangian skeleton and we are identifying $\eta\cong D^1$. Since $\eta\cap\tt=\dd \eta$, % the data of the commonly trivialized microstalks at $\tt$ allows us to identify the microstalks at the endpoints of $\bK$. By
 Corollary \ref{lem:HH_near_cycle_pointed} implies
$$\HH_0(\Sh^c_{\mathbb{K},\tt}(\eta))\cong k[\rho]\oplus k^{n}.$$
We then choose $h_\eta := \rho \in \HH_0(\Sh^c_{\mathbb{K}_n,\tt}(\bR))$ to be the explicit generator $\rho$ as our local cycle $h_\eta$. By Corollary \ref{cor:HH_corestriction_K_to_L}, the corestriction functor $\rho_!$ gives a map
$$\HH_0(\Sh^c_{\mathbb{K}_n,\tt}(\eta)) \lr \HH_0(\msh_{\L,\tt}(\L,\tt)).$$
We denote by $H_\eta$ the image of $h_\eta$ under this corestriction map. By Theorem \ref{thm:invariance}, the category in the codomain is $\msh_{\L,\tt}(\L,\tt)\cong \Sh_{\La,\tt}(\R^2)_0$ and we obtain a Hochschild cycle $H_\eta \in \HH_0(\Sh_{\La,\tt}(\R^2)_0)$, as required. By Subsection \ref{ssec:general_framework}, this cycle $H_\eta$ defines a global regular function 
$$\HO(H_\eta) \in \mathrm{H}^0(\Gamma(\M(\La,\tt), \cO_{\M(\La,\tt)})).$$

In order to conclude the proof of Theorem \ref{thm:main} we must show that $\HO(H_\eta)$ coincides with the trace of the microlocal merodromy along $\eta$ when restricted to the toric chart $T(L)$ associated to the filling $L$ via Section \ref{ssec:Lagrangian_fillings_KS_stack}. This is done as follows. Let $T(L) \sse \M(\La,\tt)$ be the moduli space of local systems on $L$ determined by the localization
\begin{equation}\label{eq:microlocal_functor_proof}
\Sh_{\La,\tt}^c(\R^2)_0 \xrightarrow{\sim} \msh_{\L,\tt}^c(\L,\tt) \to \Loc^c(L, \tt)
\end{equation}
in Corollary \ref{cor:Lagrangian_fillings_KS_stack_pointed}. Let us show that under the restriction map
$$\Gamma(\M(\La,\tt), \cO_{\M(\La,\tt)}) \lr \Gamma({T}(L), \cO_{{T}(L)}),$$
the function $\HO(H_\eta)$ restricts to the trace of the microlocal merodromy of the local system along $\eta \looparrowright L$. For that, we argue that via the map 
$$\HH_0(\Sh_{\La,\tt}^c(\R^2)_0) \lr \HH_0(\Loc^c(L, \tt))$$
induced by (\ref{eq:microlocal_functor_proof}) above, the class $H_\eta$ restricts to the relative 1-chain $\eta \in \HH_0(\Loc(L, \tt))$. Indeed, we have a commutative diagram
\begin{equation}\label{eq:diagram_HH_from_eta_to_L}
\begin{tikzcd}
\HH_0(\Sh_{\bK_n,\tt}^c(\eta)) \arrow[d] \arrow[r] & \HH_0(\Loc^c(\eta, \tt\cap\eta)) \arrow[d] \\
\HH_0(\Sh_{\La,\tt}^c(\R^2)_0) \arrow[r] & \HH_0(\Loc^c(L, \tt))
\end{tikzcd}
\end{equation}
where the vertical morphisms are induced by the corestriction functors from Corollary \ref{cor:corestriction_Lprime_L} and the horizontal morphisms are induced by the localization functors from Corollary \ref{cor:Lagrangian_fillings_KS_stack_pointed}. By identifying $\HH_0(\Loc^c(\eta, \tt\cap\eta)) \cong \HH_0(\Loc^c(S^1)) \cong k[\rho^{\pm 1}]$, the top horizontal morphism in diagram (\ref{eq:diagram_HH_from_eta_to_L}) reads
$$\HH_0(\Sh_{\bK_n,\tt}^c(\eta)) \xrightarrow{\sim} k[\rho] \oplus k^n \to k[\rho^{\pm 1}] \xrightarrow{\sim} \HH_0(\Loc^c(\eta, \tt\cap\eta)),$$
where the middle morphism $k[\rho] \oplus k^n \to k[\rho^{\pm 1}]$ is projection onto the first factor composed with the natural inclusion $k[\rho]\to k[\rho^{\pm 1}]$. Consequently, the top horizontal morphism sends $\rho$ to $\rho$, where the first $\rho$ is understood via the identification in Corollary \ref{lem:HH_near_cycle_pointed}. Therefore, the restriction of $H_\eta$ is the image of $\rho \in \HH_0(\Loc^c(\eta, \tt\cap\eta))$ and hence it is identified with the relative 1-chain $\eta \in \HH_0(\Loc(L, \tt))$. By Proposition \ref{prop:HO_in_Rmod}, the regular function associated to $\eta \in \HH_0(\Loc(L, \tt))$ computes the trace of the microlocal merodromy along $\eta$ in $\Loc(L, \tt)$. This completes the proof.\hfill$\Box$

%%%%%%%%%%%%%%%%%%%%%%%%%%%%%%%%%%%%%%%%%%%%%%%%%%%%%%%%%%%%%%%%%%%%%%%%
%%%%%%%%%%%%%%%%%%%%%%%%%%%%%%%%%%%%%%%%%%%%%%%%%%%%%%%%%%%%%%%%%%%%%%%%
%%%%%%%%%%%%%%%%%%%%%%%%%%%%%%%%%%%%%%%%%%%%%%%%%%%%%%%%%%%%%%%%%%%%%%%%

\subsection{A comment on generality}\label{ssec:generalization} We have stated and proven \cref{thm:main} and \cref{cor:main} for Legendrian links $\La\sse(T^*_\infty\R^2,\xi_\st)$. This specific framework is the most studied in the literature, see e.g.~\cite{CasalsHonghao,CasalsGao24,CGGLSS,CW,CasalsZaslow,EHK}. By the Darboux theorem, this case needs to be understood if one aims at studying Lagrangian fillings of an arbitrary Legendrian link $\La\sse(\dd W,\ker\la)$ in the contact boundary of a Weinstein 4-manifold $(W,\la)$. Following a suggestion of the referee, we briefly sketch a generalization of \cref{thm:main}, as follows.

The proof of \cref{thm:main}, as written, works almost identically if one considers embedded exact Lagrangian fillings $L\sse (W,\la_\st)$ of a Legendrian link $\La\sse(\dd W,\ker\la_\st)$. The main modifications are:

\begin{enumerate}
    \item The definition of an $\L$-compressing system $\D$ for $L$ generalizes as follows: an $\L$-compressing system $\D$ for $L\sse(W,\la_\st)$ is still a collection of Lagrangian disks in $W$ with boundary on $L$, properly embedded on its complement $W\setminus L$, with the condition that the arboreal Lagrangian given by the union $\bL:=L\cup_{\dd\D}\D$ must now be a relative Lagrangian skeleton for the Weinstein pair given by the Weinstein 4-manifold $(W,\la_\st)$ and the Weinstein ribbon of $\La$. See e.g.~\cite[Section 2]{Eliashberg18_WeinsteinRevisited} for details on Weinstein pairs.\\
    
    \noindent Note that the Kashiwara-Schapira stack $\msh_{\bL}$ from \cref{sec:corestriction} is still defined, cf.~ e.g.~\cite[Section 3.4]{Nadler16}, and the results on $\msh_{\bL}$ and corestriction functors remain the same. Similarly, the pointed versions $\msh_{\bL,\tt}(\bL)$ and $\M(\bL,\tt)$ of the category $\msh_{\bL}(\bL)$ and its moduli stack $\M(\bL)$ are defined via the same homotopy colimits and limits.\\

    \item The condition of the relative 1-cycle $\eta$ being $\D$-positive can be generalized to the condition that a small enough neighborhood of $\eta\sse L$ in the Lagrangian skeleton $\bL=L\cup_{\dd \D}\D$ retracts to the model in \cref{ssec:HH0_near_rel_cycle}, i.e.~to a horizonal line (or circle) with a sequence of half-rays attached to it {\it all pointing in the upwards direction}, as in \cref{fig:Filling_Merodromy3} (left). We still refer to such a relative 1-cycle $\eta$ as a $\D$-positive cycle.
\end{enumerate} 

\noindent The same ingredients as in \cref{ssec:ingredients}, with the modified concepts above, leads to a generalized version of \cref{thm:main}. The statement reads as follows:

\begin{thm}[Generalized extension result]\label{thm:main_generalized}
Let $\La\sse(\dd W,\ker \la_\st)$ be a $(\tt,T)$-pointed Legendrian link in the contact boundary of a Weinstein 4-manifold $(W,\la_\st)$ and $L\sse (W,\la_\st)$ an embedded exact Lagrangian filling of $\La$ endowed with an $\L$-compressing system $\D$.

Then, for any $\D$-positive relative 1-cycle $\eta$ in the relative pair $(L,T)$, there exists a Hochschild 0-cycle $H_\eta\in \HH_0(\msh_{\bL,\tt}(\bL))$ whose associated regular function
$\HO(H_\eta)\in \mathrm{H}^0\Gamma(\M({\bL,\tt}),\SO_{\M({\bL,\tt})})$
coincides with the trace of the microlocal merodromy along $\eta$ when restricted to ${T}(L)\sse \M({\bL,\tt})$.\qed
\end{thm}

\noindent Note that the symplectic invariance and properties of the category $\msh_{\bL}(\bL)$ for an arbitrary Lagrangian skeleton $\bL$ are still being explored, see e.g.~\cite{NadlerShende}. For the case of $(W,\la)=(T^*S,\la_\st)$ of the standard cotangent bundle of a smooth surface $S$, where $S$ is a smooth Lagrangian skeleton, the results from \cite{GKS_Quantization} apply, see also \cite[Section 10]{Guillermou23_SheafSummary}, and the resulting categories are classically known to be invariant under contact isotopies and compactly supported Weinstein homotopies.

%% file: 5_Appendix.tex
\section{Some background on categories and sheaves}\label{sec:appendix}

We hope that this short appendix, through Subsections \ref{ssec:infty_cat_microlocal_sheaves} and \ref{ssec:dgcat_contact_and_symp}, might be of help to some readers when navigating parts of the algebraic frameworks related to this article. In this appendix, we follow the framework of $\infty$-categories as developed in \cite{HTT}. Though not needed for our article, we recommend \cite{RiehlVerity19_infty_categories_scratch} and references therein for further discussions on different models for $\infty$-categories, see e.g.~quasi-categories (weak Kan complexes), simplicial categories and Segal spaces (and complete Segal categories).

\subsection{$\infty$-categories and the microlocal theory of sheaves}\label{ssec:infty_cat_microlocal_sheaves} Let $X$ be a locally compact Hausdorff space and $\LC$ a compactly-generated stable $\infty$-category, which serves as the coefficient category for $\infty$-sheaves on $X$. The $\infty$-topos $\Sh(X;\LC)$ of $\LC$-valued $\infty$-sheaves on $X$ is discussed in detail in \cite[Section 6.2.2]{HTT}. This higher-categorical framework is well-adapted for merging sheaves and homotopy theory.

If $X$ is a real smooth manifold the notion of singular support of a sheaf, a certain conical subset of $T^*X$, can be introduced: the study of sheaves and their singular support is known as the microlocal theory of sheaves or microlocal sheaf theory.\footnote{In microlocal sheaf theory, {\it microlocal} is an adjective for the noun {\it theory}.} For the microlocal theory of sheaves on (real analytic) smooth manifolds, the standard reference is \cite{KashiwaraSchapira_Book}.

\begin{remark}
Since its introduction, the study of singular support has been extended to more settings, including \'etale constructible sheaves on algebraic varieties over
an arbitrary field $k$ \cite{Beilinson_SSetale}, $\ell$-adic sheaves \cite{Barret23_SSladic}, and there is also a related notion of singular support for coherent sheaves, cf.~ \cite{ArinkinGaitsgory15_SSCoh}.\hfill$\Box$
\end{remark}

Technically, the results in \cite{KashiwaraSchapira_Book} are stated and proven in the setting of bounded derived categories, not in the more modern context of stable $\infty$-categories. (There are limitations to using derived categories, see e.g.~\cite[Section 2.2]{Toen08_LecturesDG}.) It is nevertheless possible to upgrade the results we need from \cite{KashiwaraSchapira_Book} to this setting by using \cite{RobaloSchapira18_ExtensionLemmaInfinity,Volpe}. Specifically, in \cite{Volpe}, the 6-functor formalism\footnote{Including the existence of the derived functors and the validity of various formulae between them.} for sheaves on locally compact
Hausdorff topological spaces is extended to sheaves with values in any closed symmetric
monoidal $\infty$-category which is stable and bicomplete, which includes the dg-nerve of any pretriangulated dg-category. In \cite[Theorem 4.1]{RobaloSchapira18_ExtensionLemmaInfinity}, the non-characteristic deformation lemma is extended to $\infty$-sheaves with values in any compactly-generated stable $\infty$-category. In particular, this implies that the equivalences between the definitions of singular support in \cite[Prop.~5.1.1]{KashiwaraSchapira_Book} still hold in the setting of pretriangulated dg or stable $\infty$-categories.

\begin{remark}\label{rmk:coefficient_category}
Even if many aspects of the microlocal theory of sheaves extend to $\infty$-sheaves, we require that the coefficient category (for the $\infty$-sheaves) is compactly generated, instead of an arbitrary stable $\infty$-category, if we use the definition of singular support in \cite{KashiwaraSchapira_Book}. The reason is that the non-characteristic deformation lemma may fail for sheaves on manifolds defined over arbitrary presentable stable $\infty$-categories (even dualizable ones), cf.~ \cite[Rem.~4.24]{EfimovKtheory}.
\hfill$\Box$
\end{remark}

In summary, the arguments and results from \cite{KashiwaraSchapira_Book}, along with \cite{RobaloSchapira18_ExtensionLemmaInfinity,Volpe}, combined with \cite{HTT}, yield a rigorous foundational framework for the microlocal theory of $\infty$-sheaves on real analytic manifolds with values in compactly-generated stable $\infty$-categories.

\subsection{Dg-categories and sheaves in contact and symplectic topology}\label{ssec:dgcat_contact_and_symp}

In part due to historical reasons, the scientific development of contact and symplectic topology with regards to the study of Legendrian submanifolds has been more inclined towards the framework of dg-categories, and that of $A_\infty$-categories, a close relative of dg-categories. (In a sense, $A_\infty$-categories form minimal models for dg-categories.) To wit, the Legendrian contact dg-algebra \cite{Chekanov,EkholmEtnyreSullivan05b} and various Fukaya categories, cf.~ \cite[Chapter 2]{Seidel08} and \cite[Section 3]{GPS1}, naturally inherit their dg- and $A_\infty$-structures via the geometric counts extracted from certain moduli spaces of pseudo-holomorphic disks. The $\Aug_+$ and $\Aug_-$ $A_\infty$-categories, respectively studied in \cite{NRSSZ} and \cite{BourgeoisChantraine14_AugCategory}, are also instances of this phenomenon. See \cite{CanonacoOrnaghiStellari,canonaco2024localizationscategoriesainftycategories,Pascaleff,Tanaka} for more details on the relation between dg-categories and $A_\infty$-categories. If we then study symplectic topology via the microlocal theory of $\infty$-sheaves, with coefficients in an $\infty$-category, a first dissonance appears at the algebraic level: Floer theoretic constructions produce dg-structures or $A_\infty$-structures, rather than $\infty$-structures. There are ways to address this:

\begin{enumerate}
    \item Relate dg-categories (and $A_\infty$-categories) and $\infty$-categories. These relations are reasonably well-established in the literature, and subtleties often reside in the homotopy theory of the functors between them. (E.g.~on the right choices of model structures and localizations to be performed when comparing their categories of categories.) Relevant references here are \cite{Cohn_DG_to_stable} and \cite[Section 1.3.1]{HA}, especially \cite[Construction 1.3.1.6 and Prop. 1.3.1.10]{HA}. An important construction is the dg-nerve, which inputs a dg-category and outputs a simplicial set. This simplicial set is to be understood as a quasi-category, one of the models for a $\infty$-category. The dg-nerve of a pretriangulated dg-category is stable.\\

    \item Alternatively, consider categories of sheaves with given singular support within the dg-setting, i.e.~{\it not} as $\infty$-sheaves with values in an $\infty$-categories but rather as sheaves with values in a dg-category. (Often the dg-derived setting is chosen.) For instance, this is the approach taken in \cite{Guillermou23_SheafSummary}. Some relevant references for dg-categories are \cite{Drinfeld_DG,DGDerived3,Keller_ICM06,Tabuada05_ModelStructure,Tabuada_Wellgenerated,DGDerived1}. In this framework, we use that the $\infty$-category of dg-categories admits (homotopy) limits and colimits. In particular, this is used to specify the descent condition for a presheaf (of dg-categories) to be a sheaf. Similarly, being complete and cocomplete implies that the sheafification construction in \cite[Section 6.5.3]{HTT} works for presheaves of dg-categories.
\end{enumerate}

Each framework has its own features. An advantage of $(1)$ is that one can consider $\infty$-categories as a unifying framework: one can construct an $\infty$-category from a dg-category (or an $A_\infty$-category), e.g.~via the dg-nerve. There are a number of well-defined operations in the $\infty$-category of $\infty$-categories (and therefore in the $\infty$-topos $\Sh(X;\LC)$ of $\infty$-sheaves) which are more subtle in the dg-setting. In addition, having the freedom of choosing a more general $\infty$-category as coefficients (such as spectra, see ~\cite[Chapter 1]{HA}), rather than dg-categories as coefficients, can be fruitfully exploited, cf. \cite{JinTreumann}. An advantage of (2) is that explicit computations are typically more accessible: e.g.~the Legendrian contact dg-algebra of a Legendrian link can be presented with finitely many generators and relations (of polynomial type), which a computer can readily produce and manipulate from a (plat) front. This allows for more accessible approaches when extracting contact and symplectic invariants, such as computing the $\Aug_\pm$ categories or explicitly giving compact generators of wrapped Fukaya categories.

In our particular case, the chief reason to consider dg-categories, instead of $\infty$-categories, is \cite{ToenVaquie07}. Namely, \cite[Theorem 0.2]{ToenVaquie07} establishes a number of desirable properties for the derived stack $\M_\SC$ of pseudo-perfect objects of a dg-category $\SC$ of finite type. These include $\M_\SC$ being locally geometric, locally of finite presentation and, for quasi-representable objects, an identification of the tangent complex $T_x \M_\SC\simeq\End(x)[1]$ at an object $x\in\mbox{Ob}(\SC)$ with its (shifted) endomorphism dg-algebra. The dg-derived categories in Section \ref{sec:sheaf_cat_moduli}, e.g.~$\SC=\Sh_\La^c(\R^2)$, are of finite type and thus \cite{ToenVaquie07} applies.

\begin{remark}
It might be possible to establish results for derived stacks as in \cite{ToenVaquie07} for certain compactly-generated presentable $k$-linear stable $\infty$-categories, instead of finite-type dg-categories. See e.g.~\cite{BZFN,Lurie_SAG,GaitsgoryRozenblyum1,GaitsgoryRozenblyum2} for the formalism of derived stacks using stable $\infty$-categories and \cite{Pandit_moduli}, \cite[Section 4]{Lowrey18_moduli_stack_DbCoh} and \cite[Section 7]{PortaTeyssier} for some construction and results of derived moduli stacks (partially) in the setting of $\infty$-categories.\hfill$\Box$
\end{remark}

As said above, we can develop our results within option (2): considering dg-derived categories of sheaves, so that we can use \cite{ToenVaquie07}, and not use the $\infty$-categorical setting. To finalize, we emphasize that \cite{Cohn_DG_to_stable} can also be used to prove our results via option (1): translate all dg-categorical constructions into $\infty$-categories and then either justify that the results can still be understood within the context of dg-categories or directly rectify the resulting $\infty$-categorical constructions back to the dg-setting. Observe that the dg-categories appearing in the contact and symplectic topology that are currently being studied are compactly-generated dg-categories.

This rectification is summarized as follows. Namely, \cite[Cor.~5.7]{Cohn_DG_to_stable} shows that the underlying $\infty$-category $\mbox{dg-cat}_k$ of the Morita model category structure
on the category of small dg-categories is equivalent to the $\infty$-category of compactly-generated presentable $k$-linear
stable $\infty$-categories $\operatorname{Pr}^\text{L}_{\omega,st,k}$ (with functors that preserve colimits and compact objects) or equivalently $(\operatorname{Pr}^\text{R}_{\omega,st,k})^{op}$  (with functors that preserve colimits and limits). At core, the results of \cite{Cohn_DG_to_stable} are rectification results, in particular showing that given a compactly generated $k$-linear
stable $\infty$-category, there exists a pretriangulated dg-category corresponding to it, in the appropriate homotopical sense. (See also the rectification result \cite[Theorem 1.1]{Haugseng15_Rectification}, which implies that $\infty$-categories enriched over
chain complexes are equivalent to dg-categories.)